\documentclass[3p,times]{elsarticle}
\usepackage{bbm}
\usepackage{bm}
\usepackage{color}
\usepackage{xcolor}

\usepackage{graphicx}
\usepackage{epstopdf}
\usepackage{psfrag}

\usepackage{amssymb}
\usepackage{amsmath}
\usepackage{dsfont}
\usepackage{rotating}

\usepackage{amsthm}
\usepackage{pgf,tikz}
\usetikzlibrary{arrows}

\definecolor{ttzzqq}{rgb}{0.2,0.6,0}

\definecolor{qqttcc}{rgb}{0,0.2,0.8}
\definecolor{ffxfqq}{rgb}{1,0.5,0}
\definecolor{qqttzz}{rgb}{0,0.2,0.6}
\definecolor{ffqqqq}{rgb}{1,0,0}
\definecolor{qqwuqq}{rgb}{0,0.39,0}
\definecolor{zzttqq}{rgb}{0.6,0.2,0}
\definecolor{qqqqff}{rgb}{0,0,1}
\definecolor{ttttqq}{rgb}{0.2,0.2,0}

\definecolor{qqwwtt}{rgb}{0,0.4,0.2}
\definecolor{ubqqys}{rgb}{0.29,0,0.51}

\definecolor{wwttqq}{rgb}{0.4,0.2,0}
\definecolor{uuuuuu}{rgb}{0.27,0.27,0.27}
\definecolor{qqzzff}{rgb}{0,0.6,1}
\definecolor{xdxdff}{rgb}{0.49,0.49,1}
\definecolor{ccwwqq}{rgb}{0.8,0.4,0}

\definecolor{ttqqqq}{rgb}{0.2,0,0}

\definecolor{qqzzcc}{rgb}{0,0.6,0.8}

\newcommand{\xx}{\mathbf{x}}

\newcommand{\R}{\mathbb{R}}

\newcommand{\x}{\chi}

\newcommand{\eref}[1]{$(\ref{#1})$}
\newcommand{\p}{\mathbb{\wp}}
\newcommand{\np}{p}

\newcommand{\nextud}{\vec{n}_{ij}}
\newcommand{\nstd}{\vec{n}_{j}}
\newcommand{\etah}{ \hat{\bm{p}}}
\newcommand{\bphi}{ \bm{\phi}}
\newcommand{\bpsi}{ \bm{\psi}}

\newcommand{\vvh}{\hat{\bm{\mathbf{v}}}}

\newcommand{\Fvh}{\widehat{\bm{F\mathbf{v}}}}

\newcommand{\nv}{\vec{n}}
\newcommand{\B}{\mathcal{B}}
\newcommand{\TF}{F}

\newcommand{\TT}{\mbox{\boldmath$T$}}
\newcommand{\QQ}{\mbox{\boldmath$H$}}
\newcommand{\poh}{+\frac{1}{2}}
\newcommand{\Ni}{N_e}
\newcommand{\Nj}{N_d}

\newcommand{\st}{ {st}}
\newcommand{\Ss}{ \bm{\mathcal{S}}}

\newcommand{\D}{\bm{\mathcal{D}}}
\newcommand{\Q}{\bm{\mathcal{Q}}}
\newcommand{\RM}{\bm{\mathcal{R}}}
\newcommand{\LM}{\bm{\mathcal{L}}}

\newcommand{\Mpsi}{\bm{M}}

\newcommand{\diff}[2]{\frac{\partial {#1} }{\partial {#2} } }

\newcommand{\vbar}{\overline{\mathbf{v}}}

\newtheorem{Corollary}{Corollary}

\usepackage{amssymb}
\journal{Journal of Computational Physics}

\begin{document}

\begin{frontmatter}



\title{A staggered space-time discontinuous Galerkin method for the three-dimensional incompressible Navier-Stokes equations on unstructured tetrahedral meshes}

\author[1]{Maurizio Tavelli\fnref{label1}}
\author[2]{Michael Dumbser \corref{corr1} \fnref{label2}}
\address[1]{Department of Mathematics, University of Trento, \\ Via Sommarive 14, I-38050 Trento, Italy}
\address[2]{Laboratory of Applied Mathematics, Department of Civil, Environmental and Mechanical Engineering,
          					   University of Trento, Via Mesiano 77, I-38123 Trento, Italy}

\fntext[label1]{\tt m.tavelli@unitn.it (M.~Tavelli)}
\fntext[label2]{\tt michael.dumbser@unitn.it (M.~Dumbser)}
\begin{abstract}
In this paper we propose a novel arbitrary high order accurate semi-implicit \textit{space-time} discontinuous Galerkin method for the solution of the three-dimensional incompressible Navier-Stokes  equations on \textit{staggered} unstructured curved tetrahedral meshes. As typical for space-time DG schemes, the discrete solution is represented in terms of space-time basis functions. This 
allows to achieve very high order of accuracy also in time, which is not easy to obtain for the incompressible Navier-Stokes equations. Similar to staggered finite difference schemes, in our 
approach the  discrete \textit{pressure} is defined on the \textit{primary} tetrahedral grid, while the discrete \textit{velocity} is defined on a face-based staggered \textit{dual} grid.  
While staggered meshes are state of the art in classical finite difference schemes for the incompressible Navier-Stokes equations, their use in high order DG schemes still quite rare. 
A very simple and efficient Picard iteration is used in order to derive a space-time pressure correction algorithm that achieves also high order of accuracy in time and that avoids the direct 
solution of global nonlinear systems. 
Formal substitution of the discrete momentum equation on the dual grid into the discrete continuity equation on the primary grid yields a very sparse five-point block system for the scalar 
pressure, which is conveniently solved with a matrix-free GMRES algorithm. From numerical experiments we find that the linear system seems to be reasonably well conditioned, since all simulations 
shown in this paper could be run without the use of any preconditioner, even up to very high polynomial degrees. For a piecewise constant polynomial approximation in time and if pressure boundary  conditions 
are specified at least in one point, the resulting system is, in addition, symmetric and positive definite. This allows us to use even faster iterative solvers, like the conjugate gradient method. 

The flexibility and accuracy of high order space-time DG methods on curved unstructured meshes allows to discretize even complex physical  domains with very coarse grids in both, space and time. The proposed method is verified for approximation polynomials of degree up to four in space and time by solving a series of typical 3D test problems and by comparing the obtained numerical 
results with available exact analytical solutions, or with other numerical or experimental reference data. 

To the knowledge of the authors, this is the first time that a \textit{space-time} discontinuous Galerkin finite element method is presented for the three-dimensional incompressible Navier-Stokes  equations on \textit{staggered} unstructured tetrahedral grids. 
\end{abstract}


\end{frontmatter}


\section{Introduction}
The numerical solution of the three dimensional incompressible Navier-Stokes equations represents a very important and challenging research topic, both from a numerical and from an application point of 
view. 
In the literature, there are many different approaches that have been proposed for the solution of the incompressible Navier-Stokes equations, for example using classical finite difference methods 
\cite{markerandcell,patankarspalding,patankar,vanKan} or continuous finite element schemes \cite{TaylorHood,SUPG,SUPG2,Fortin,Verfuerth,Rannacher1,Rannacher3}. 
Very recently, also different high order discontinuous Galerkin (DG) methods have been presented for the solution of the incompressible and the compressible Navier-Stokes equations. 
The first DG schemes that were able to solve the Navier-Stokes equations were those of Bassi and Rebay \cite{BassiRebay} and Baumann and Oden \cite{BaumannOden1,BaumannOden2}. 
Many other methods have been presented in the meantime, see for example \cite{BassiRebay,BaumannOden1,BaumannOden2,Bassi2006,Bassi2007,MunzDiffusionFlux,stedg2,ADERNSE,HartmannHouston1,HartmannHouston2,Shahbazi2007,LuoNSE,Ferrer2011,Nguyen2011,Crivellini2013,KleinKummerOberlack2013} for a non-exhaustive overview of the ongoing research in this very active field. 
In most DG schemes, the DG discretization is only used for space discretization, while the time discretization uses standard explicit or implicit time integrators known for ordinary 
differential equations, following the so-called method of lines approach. The method 
of lines has also been used by Cockburn and Shu in their well-known series of papers \cite{cbs2,cbs3,cbs4} on DG schemes for time-dependent nonlinear hyperbolic systems. 
In contrast to the method of lines approach, the family of space-time discontinuous Galerkin finite element schemes, which was introduced for the first time by van der Vegt et al. in 
\cite{spacetimedg1,spacetimedg2,KlaijVanDerVegt}, treats space and time in a unified manner. This is achieved by using test and basis functions that depend on both space and time, 
see \cite{Rhebergen2012,Rhebergen2013,Balazsova1,Balazsova2,DumbserFacchini,2STINS,Wang2015} for an overview of recent results. For a very early implementation of continuous 
space-time finite element schemes, the reader is also referred to \cite{OdenSTFEM}. 

From an application point of view, it is very important to consider the fully three-dimensional Navier-Stokes equations, in order to capture the relevant flow features that are observed 
in laboratory experiments, see \cite{Rehimi2008,Williamson1988,Sakamoto1990, Armaly1983}. This means that the use of a two-dimensional algorithm is in most cases inappropriate to 
reproduce the results of physical experiments, even for geometries that can be considered essentially two-dimensional. The importance of fully three-dimensional computations has been 
shown, for example, in \cite{Armaly1983,Hwar1987,Aidun1991,Brachet1983,Kanaris2011,Ribeiro2012,Arthur2000}. 
Unfortunately, the mesh generation for complex and realistic 3D geometries is still nowadays quite difficult, and the computational cost of a fully three-dimensional simulation grows 
very quickly with increasing mesh resolution. In this context, it becomes crucial to use unstructured simplex meshes, since they help to simplify the process of mesh generation significantly
compared to unstructured hexahedral meshes. Furthermore, it is at the same time also crucial to use very high order accurate methods in both space and time, since they allow to reduce the 
total number of elements significantly, compared to low order methods, while keeping at the same time a high level of accuracy of the numerical solution. 
Since the solution of the incompressible Navier-Stokes equations requires necessarily the solution of large systems of algebraic equations, it is indeed very important 
to derive a scheme that uses a stencil that is as small as possible, in order to improve the sparsity pattern of the resulting system matrices. It is also desirable to design methods that
lead to reasonably well conditioned systems that can be solved with iterative solvers, like the conjugate gradient method \cite{cgmethod} or the GMRES algorithm \cite{GMRES}. 

For structured grids, numerical schemes can be usually derived rather easily in multiple space dimensions, thanks to the particular regularity of the mesh. On the contrary, the development
of numerical schemes on general unstructured meshes in three space dimensions is not as straightforward and requires some care in the derivation and the implementation of the method. 
Particular difficulties of the incompressible Navier-Stokes equations arise from their nonlinearity and from the elliptic nature of the Poisson equation for the pressure, that is also
obtained on the discrete level when substituting the momentum equation into the discrete continuity equation. 
A unified analysis of several variants of the DG method applied to an elliptic model problem has been provided by Arnold et al. in \cite{ArnoldBrezzi}. 

While the use of \textit{staggered grids} is a very common practice in the finite difference community, its use is not so widespread in the context of high order DG schemes. 
The first staggered DG schemes, based on a \textit{vertex-based} dual grid, have been proposed in  \cite{CentralDG1,CentralDG2}. Other recent high order staggered DG schemes that 
use an \textit{edge-based} dual grid have been forwarded in \cite{StaggeredDGCE1,StaggeredDG,DumbserCasulli}. 
The advantage in using edge-based staggered grids is that they allow to improve significantly the sparsity pattern of the final linear system that has to be solved for the pressure.  
Very recently, a new family of \textit{staggered} semi-implicit DG schemes for the solution of the two dimensional shallow water equations was presented by Dumbser \&  Casulli 
\cite{DumbserCasulli} and Tavelli \& Dumbser \cite{2DSIUSW}. Subsequently, these semi-implicit staggered DG schemes have been successfully extended also to the two-dimensional 
incompressible Navier-Stokes equations by Tavelli \& Dumbser in \cite{2SINS,2STINS}. One year later, a staggered DG formulation for the 2D incompressible Navier-Stokes equations 
has been reproposed independently also in \cite{ChungNS}. 
Alternative semi-implicit discontinuous Galerkin schemes on \textit{collocated grids} have been presented, for example, in \cite{Dolejsi1,Dolejsi2,Dolejsi3,GiraldoRestelli,TumoloBonaventuraRestelli}. These semi-implicit schemes try to combine the simplicity of explicit methods for nonlinear PDE with the stability and efficiency of implicit time discretizations. 

In this paper we propose a new, arbitrary high order accurate \textit{staggered} space-time discontinuous Galerkin finite element method for the solution of the three-dimensional 
incompressible Navier-Stokes equations on \textit{curved} unstructured tetrahedral meshes, following some of the ideas outlined in \cite{2STINS} for the two-dimensional case. 
For that purpose we mimic the philosophy of staggered semi-implicit finite difference schemes, such as discussed and analyzed in 
 \cite{markerandcell,patankarspalding,patankar,vanKan,HirtNichols,CasulliCheng1992,Casulli1999,CasulliWalters2000,Casulli2009,CasulliVOF,BrugnanoCasulli,BrugnanoCasulli2, 
CasulliZanolli2010,CasulliZanolli2012}, where the discrete pressure field is defined on the primary grid, while the discrete velocity field is defined on an edge-based staggered 
dual grid. 

For the staggered space-time DG scheme proposed in this paper, we use a \textit{primal mesh} composed of (curved) tetrahedral elements, and a face-based staggered \textit{dual mesh} that 
consists of non-standard five-point hexahedral elements that are obtained by connecting the three nodes of a face of the primal mesh with the barycenters of the two tetrahedra that share
the common face. The face-based dual grid used here corresponds to the choice made also in \cite{Bermudez1998,Bermudez2014,USFORCE,StaggeredDG}. These spatial elements are then extended 
to space-time control volumes using a simple tensor product in the time direction. 

Since all quantities are readily defined where they are needed, the staggered DG scheme does \textit{not} require the use of Riemann solvers (numerical flux functions), apart from the 
nonlinear convective terms, which are treated in a standard way. For the convective part of the incompressible Navier-Stokes equations, we use a standard DG scheme for hyperbolic PDE
on the main grid, based on the local Lax-Friedrichs (Rusanov) flux \cite{Rusanov:1961a}. For that purpose, the velocity field is first interpolated from the dual grid to the main grid, 
as suggested in \cite{DumbserCasulli}. This allows us to use the same staggered space-time DG scheme \textit{again} to discretize also the viscous terms, where now the 
\textit{velocity gradient} that is needed for the evaluation of the viscous fluxes is computed on the face-based \textit{staggered} dual grid. In this way, we can avoid again the use of
numerical flux functions for the viscous fluxes, and furthermore, the structure of the resulting linear systems for the viscous terms is very similar to the pressure system. 

The discrete momentum equation is then inserted into the discrete continuity equation in order to obtain the discrete form of the pressure Poisson equation. Thanks to the use of a 
\textit{staggered grid}, this leads to a very sparse five-point block system, with the scalar pressure as the \textit{only} unknown quantity. Note that the \textit{same algorithm} on a 
\textit{collocated grid} would produce a \textbf{17-point} stencil, since it would also involve neighbors of neighbors\footnote{The discrete continuity equation of a DG scheme on a
collocated grid involves the velocity in the element itself and in its four neighbor elements, due to the numerical flux on the element boundaries. Furthermore, in the discrete 
momentum equation the velocity field in each tetrahedral element depends on the pressure in the cell itself and in its four neighbors. Inserting now the momentum equation into the 
continuity equation on the discrete level involves a total of $1+4+4\cdot 3 = 17$ elements for the pressure! \\ On a staggered mesh instead, the discrete continuity equation involves only the velocities of the four
dual elements associated with the faces of the primary element. The discrete momentum equation written on the face-based dual grid only involves the pressure of the \textit{two} tetrahedra that share 
the common face. Hence, substituting the momentum equation into the continuity equation leads to a $1+4=5$ point stencil for the pressure, which involves only the element and its four neighbors.}. On the other hand, if one does not substitute the momentum equation 
into the continuity equation on a collocated grid, one could still obtain a five point stencil, but with the pressure \textit{and} the velocity vector as unknowns, hence the final 
system to solve is four times larger than the corresponding system of our staggered DG scheme. It is therefore very clear that even in the DG context, the use of a staggered mesh is 
very beneficial, since it allows to produce a linear system with the smallest possible stencil and with the smallest number of unknowns, compared to similar approaches on a collocated 
mesh. 

Once the new pressure field is known, the velocity vector field can subsequently be updated directly. A very simple Picard iteration that embraces the entire scheme in each time step is 
used in order to achieve arbitrary high order of accuracy in time also for the nonlinear convective and viscous terms, without introducing a nonlinearity in the system for the pressure. 

The rest of the paper is organized as follows: in Section \ref{sec_1} we derive and present the new numerical method method. Section \ref{sec_nlcd} contains the details about the discretization 
of the nonlinear convective terms on the main grid, while the velocity gradients for the viscous terms are discretized again on the face-based staggered dual mesh. 
In Section \ref{sec.CNmethod} we discuss the important special case 
of a high order DG discretization in space, while using only a piecewise constant polynomial approximation in time, leading to symmetric positive definite systems for the pressure and the 
viscous terms. 
Finally, in Section \ref{sec_numres} the new numerical scheme proposed in this paper is run on a set of 3D benchmark problems, 
comparing the numerical results either with existing analytical or numerical reference solutions, or with available experimental results. The paper closes with some concluding remarks provided 
in Section \ref{sec.concl}.

\section{Staggered space-time DG scheme for the 3D incompressible Navier-Stokes equations} 
\label{sec_1}
\subsection{Governing equations}
The three-dimensional incompressible Navier-Stokes equations can be written as 
\begin{eqnarray}
    \frac{\partial \mathbf{v}}{\partial t}+\nabla \cdot \mathbf{F}_c + \nabla p= \nabla \cdot \left( \nu \nabla \mathbf{v} \right) + \mathbf{S} \label{eq:CS_2_2_0}, \\
    \nabla \cdot \mathbf{v}=0 \label{eq:CS_2},
\end{eqnarray}
where $\mathbf{x}=(x,y,z)$ is the vector of spatial coordinates and $t$ denotes the time; $p=P/\rho$ indicates the normalized fluid pressure; $P$ is the physical pressure and $\rho$ is the constant fluid density; $\nu=\mu / \rho$ is the kinematic viscosity  coefficient; $\mathbf{v}=(u,v,w)$ is the velocity vector; $u$, $v$ and $w$ are the velocity components in the $x$, $y$ and $z$ direction, respectively;
$\mathbf{S}=\mathbf{S}(\mathbf{x},t)$ is a vector of given source terms;  
$\mathbf{F}_c=\mathbf{v} \otimes \mathbf{v}$ is the flux tensor of the nonlinear convective terms, namely:
$$ \mathbf{F}_c=\left(\begin{array}{ccc} uu & uv & uw \\ vu & vv & vw\\ wu & wv & ww \end{array} \right). $$

The viscosity term can be grouped with the nonlinear convective term, i.e. the momentum Eq. \eref{eq:CS_2_2_0} then reads 
\begin{equation}
	\frac{\partial \mathbf{v}}{\partial t}+\nabla \cdot \mathbf{\TF} + \nabla p = \mathbf{S}, 
\label{eq:CS_2_2}
\end{equation}
where $\mathbf{\TF}=\mathbf{\TF}(\mathbf{v},\nabla \mathbf{v})=\mathbf{F}_c(\mathbf{v})-\nu \nabla \mathbf{v}$ is the nonlinear flux tensor that depends on the velocity and its gradient. 
% *********************************************************************************************
% ************************ SECTION 1.1 Unstructured grid **************************************
% *********************************************************************************************

\subsection{Staggered unstructured mesh and associated space-time basis functions}
Throughout this paper we use a main grid that is composed of (eventually curved) tetrahedral simplex elements, and a staggered face-based dual grid, consisting in non-standard five-point hexahedral  elements. These spatial control volumes are then extended to space-time control volumes using a tensor product in time direction. In the following, the staggered mesh in space is 
described in detail and is subsequently also extended to the time direction. The main notation is taken as the one presented for the two dimensional method proposed in \cite{2STINS} and is 
summarized here for the three dimensional case. 

\subsubsection{Staggered space-time control volumes }
The spatial computational domain $\Omega$ is covered with a set of $\Ni$ non-overlapping tetrahedral elements $\TT_i$ with $i=1 \ldots \Ni$. By denoting with $\Nj$ the total number of faces, the 
$j-$th face will be called $\Gamma_j$. $\B(\Omega)$ denotes the set of indices $j$ corresponding to boundary faces. 
The indices of the four faces of each tetrahedron $\TT_i$ constitute the set $S_i$ defined by $S_i=\{j \in [1,\Nj] \,\, | \,\, \Gamma_j \mbox{ is a face of }\TT_i \}$. For every $j\in [1\ldots \Nj]-\B(\Omega)$ there  exist two tetrahedra that share a common face $\Gamma_j$. We assign arbitrarily a left and a right element, called $\TT_{\ell(j)}$ and $\TT_{r(j)}$, respectively. The standard positive direction is  assumed to be from left to right. Let $\nv_{j}$ denote the unit normal vector defined on the face number $j$ and that is oriented with respect to the positive direction from left to right. For every  tetrahedral element number $i$ and face number $j \in S_i$, the index of the neighbor tetrahedron that shares the common face $\Gamma_j$ is denoted by $\p(i,j)$.
\par For every $j\in [1, \Nj]-\B(\Omega)$ the dual element (a non-standard 5-point hexahedron) associated with $\Gamma_j$ is called $\QQ_j$ and it is defined by the two centers of gravity of 
$\TT_{\ell(j)}$ and $\TT_{r(j)}$ and the three vertices of $\Gamma_j$, see also \cite{Bermudez1998,USFORCE,2DSIUSW}. We denote by $\TT_{i,j}=\QQ_j \cap \TT_i$ the intersection element for every 
$i$ and $j \in S_i$. Figures $\ref{fig.1}$ and $\ref{fig.1.1}$ summarize the notation used on the main tetrahedral mesh and on the associated dual grid.
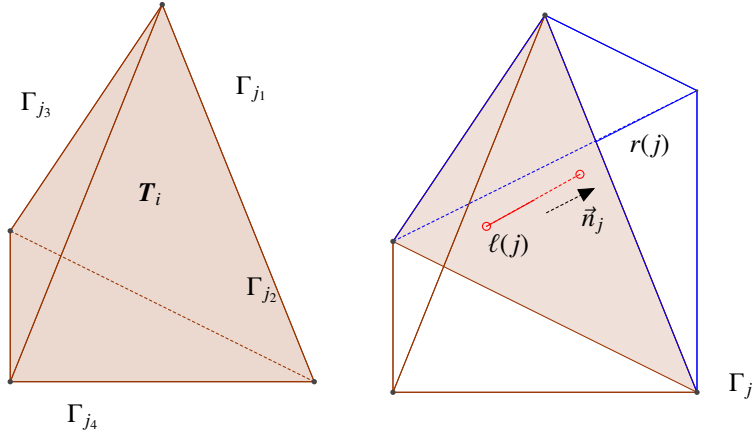
\begin{figure}[ht]
    \begin{center}
    \begin{tikzpicture}[line cap=round,line join=round,>=triangle 45,x=1.0cm,y=1.0cm]
\clip(3.26,-5.21) rectangle (8.37,1.57);
\fill[color=zzttqq,fill=zzttqq,fill opacity=0.1] (4,-4) -- (8,-4) -- (4,-2) -- cycle;
\fill[color=zzttqq,fill=zzttqq,fill opacity=0.1] (4,-2) -- (6,1) -- (8,-4) -- cycle;
\fill[color=zzttqq,fill=zzttqq,fill opacity=0.1] (4,-2) -- (6,1) -- (4,-4) -- cycle;
\fill[color=zzttqq,fill=zzttqq,fill opacity=0.1] (6,1) -- (4,-4) -- (8,-4) -- cycle;
\fill[color=zzttqq,fill=zzttqq,fill opacity=0.1] (16,-4) -- (20,-4) -- (16,-2) -- cycle;
\fill[color=zzttqq,fill=zzttqq,fill opacity=0.1] (16,-2) -- (18,1) -- (20,-4) -- cycle;
\fill[color=zzttqq,fill=zzttqq,fill opacity=0.1] (18,1) -- (16,-4) -- (20,-4) -- cycle;
\fill[color=zzttqq,fill=zzttqq,fill opacity=0.1] (16,-2) -- (12.56,-0.2) -- (16,-4) -- cycle;
\fill[color=qqqqff,fill=qqqqff,fill opacity=0.1] (12.56,-0.2) -- (15.16,-1.08) -- (16,-4) -- cycle;
\fill[color=qqqqff,fill=qqqqff,fill opacity=0.1] (15.16,-1.08) -- (16,-2) -- (12.56,-0.2) -- cycle;
\fill[color=qqqqff,fill=qqqqff,fill opacity=0.1] (15.16,-1.08) -- (16,-4) -- (16,-2) -- cycle;
\fill[color=qqqqff,fill=qqqqff,fill opacity=0.1] (12.56,-0.2) -- (13.78,-2.9) -- (16,-4) -- cycle;
\fill[color=qqqqff,fill=qqqqff,fill opacity=0.1] (13.78,-2.9) -- (16,-2) -- (12.56,-0.2) -- cycle;
\fill[color=qqqqff,fill=qqqqff,fill opacity=0.1] (13.78,-2.9) -- (16,-4) -- (16,-2) -- cycle;
\fill[color=zzttqq,fill=zzttqq,fill opacity=0.15] (26,-6) -- (28,-3) -- (30,-8) -- cycle;
\draw [color=zzttqq] (4,-4)-- (8,-4);
\draw [color=zzttqq] (4,-2)-- (4,-4);
\draw [color=zzttqq] (4,-2)-- (6,1);
\draw [color=zzttqq] (6,1)-- (8,-4);
\draw [dash pattern=on 1pt off 1pt,color=zzttqq] (8,-4)-- (4,-2);
\draw [color=zzttqq] (4,-2)-- (6,1);
\draw [color=zzttqq] (6,1)-- (4,-4);
\draw [color=zzttqq] (4,-4)-- (4,-2);
\draw [color=zzttqq] (6,1)-- (4,-4);
\draw [color=zzttqq] (4,-4)-- (8,-4);
\draw [color=zzttqq] (8,-4)-- (6,1);
\draw [color=zzttqq] (16,-4)-- (20,-4);
\draw [color=zzttqq] (16,-2)-- (16,-4);
\draw [color=zzttqq] (16,-2)-- (18,1);
\draw [color=zzttqq] (18,1)-- (20,-4);
\draw [dash pattern=on 1pt off 1pt,color=zzttqq] (20,-4)-- (16,-2);
\draw [color=zzttqq] (18,1)-- (16,-4);
\draw [color=zzttqq] (16,-4)-- (20,-4);
\draw [color=zzttqq] (20,-4)-- (18,1);
\draw [color=zzttqq] (16,-2)-- (16.67,-2.33);
\draw [shift={(16,-3)},dotted]  plot[domain=1.11:2.46,variable=\t]({1*4.47*cos(\t r)+0*4.47*sin(\t r)},{0*4.47*cos(\t r)+1*4.47*sin(\t r)});
\draw [color=zzttqq] (12.56,-0.2)-- (16,-4);
\draw [color=zzttqq] (16,-4)-- (16,-2);
\draw [shift={(16,-3)},dotted]  plot[domain=0.66:1.98,variable=\t]({1*2.12*cos(\t r)+0*2.12*sin(\t r)},{0*2.12*cos(\t r)+1*2.12*sin(\t r)});
\draw [color=qqqqff] (12.56,-0.2)-- (15.16,-1.08);
\draw [color=qqqqff] (15.16,-1.08)-- (16,-4);
\draw [color=qqqqff] (16,-4)-- (12.56,-0.2);
\draw [color=qqqqff] (15.16,-1.08)-- (16,-2);
\draw [color=qqqqff] (12.56,-0.2)-- (15.16,-1.08);
\draw [color=qqqqff] (15.16,-1.08)-- (16,-4);
\draw [color=qqqqff] (16,-4)-- (16,-2);
\draw [color=qqqqff] (16,-2)-- (15.16,-1.08);
\draw [color=qqqqff] (12.56,-0.2)-- (13.78,-2.9);
\draw [color=qqqqff] (13.78,-2.9)-- (16,-4);
\draw [color=qqqqff] (16,-4)-- (12.56,-0.2);
\draw [dash pattern=on 1pt off 1pt,color=qqqqff] (16,-2)-- (12.56,-0.2);
\draw [color=qqqqff] (12.56,-0.2)-- (13.78,-2.9);
\draw [color=qqqqff] (13.78,-2.9)-- (16,-4);
\draw [color=qqqqff] (16,-4)-- (16,-2);
\draw [dash pattern=on 1pt off 1pt,color=qqqqff] (16,-2)-- (13.78,-2.9);
\draw [color=zzttqq] (26,-8)-- (30,-8);
\draw [color=zzttqq] (26,-6)-- (26,-8);
\draw [color=zzttqq] (26,-6)-- (28,-3);
\draw [color=zzttqq] (28,-3)-- (30,-8);
\draw [color=zzttqq] (30,-8)-- (26,-6);
\draw [color=zzttqq] (26,-6)-- (28,-3);
\draw [color=zzttqq] (28,-3)-- (26,-8);
\draw [color=zzttqq] (26,-8)-- (26,-6);
\draw [color=zzttqq] (28,-3)-- (26,-8);
\draw [color=zzttqq] (26,-8)-- (30,-8);
\draw [color=zzttqq] (30,-8)-- (28,-3);
\draw [dash pattern=on 1pt off 1pt,color=qqqqff] (30,-4)-- (26,-6);
\draw [color=qqqqff] (30,-8)-- (30,-4);
\draw [color=qqqqff] (28,-3)-- (30,-4);
\draw [color=qqqqff] (30,-4)-- (30,-8);
\draw [color=qqqqff] (30,-8)-- (28,-3);
\draw [color=qqqqff] (26,-6)-- (28,-3);
\draw [color=qqqqff] (28,-3)-- (30,-4);
\draw [dash pattern=on 1pt off 1pt,color=ffqqqq] (28.46,-5.11)-- (27.23,-5.8);
\draw [color=qqqqff] (28.67,-4.68)-- (30,-4);
\draw [color=ffqqqq] (27.23,-5.8)-- (27.86,-5.45);
\draw (5.56,-1.25) node[anchor=north west] {$\TT_i$};
\draw (6.85,0.18) node[anchor=north west] {$\Gamma_{j_1}$};
\draw (6.99,-2.49) node[anchor=north west] {$\Gamma_{j_2}$};
\draw (4.02,-0.04) node[anchor=north west] {$\Gamma_{j_3}$};
\draw (4.63,-4.18) node[anchor=north west] {$\Gamma_{j_4}$};
\draw (14.4,0.1) node[anchor=north west] {$\QQ_j$};
\draw (16.02,-0.08) node[anchor=north west] {$i$};
\draw (17.82,-1.74) node[anchor=north west] {$\TT_i$};
\draw (13.21,-3.03) node[anchor=north west] {$\mathbb{\wp}(i,j)$};
\draw [->,dash pattern=on 1pt off 1pt] (28.02,-5.62) -- (28.66,-5.29);
\draw (27.13,-5.78) node[anchor=north west] {$\ell(j)$};
\draw (28.99,-4.46) node[anchor=north west] {$r(j)$};
\draw (28.36,-5.41) node[anchor=north west] {$\vec{n_j}$};
\draw (30.3,-7.62) node[anchor=north west] {$\Gamma_j$};
\begin{scriptsize}
\fill [color=uuuuuu] (4,-4) circle (1.0pt);
\fill [color=uuuuuu] (8,-4) circle (1.0pt);
\fill [color=uuuuuu] (4,-2) circle (1.0pt);
\fill [color=uuuuuu] (6,1) circle (1.0pt);
\fill [color=uuuuuu] (16,-4) circle (1.0pt);
\fill [color=uuuuuu] (20,-4) circle (1.0pt);
\fill [color=uuuuuu] (16,-2) circle (1.0pt);
\fill [color=uuuuuu] (18,1) circle (1.0pt);
\fill [color=uuuuuu] (12.56,-0.2) circle (1.0pt);
\draw [color=ffqqqq] (17.68,-1.7) circle (1.5pt);
\draw [color=ffqqqq] (15.16,-1.08) circle (1.5pt);
\draw [color=ffqqqq] (13.78,-2.9) circle (1.5pt);
\fill [color=uuuuuu] (26,-8) circle (1.0pt);
\fill [color=uuuuuu] (30,-8) circle (1.0pt);
\fill [color=uuuuuu] (26,-6) circle (1.0pt);
\fill [color=uuuuuu] (28,-3) circle (1.0pt);
\fill [color=qqqqff] (30,-4) circle (1.5pt);
\draw[color=qqqqff] (0.19,5.94) node {$L$};
\draw [color=ffqqqq] (28.46,-5.11) circle (1.5pt);
\draw [color=ffqqqq] (27.23,-5.8) circle (1.5pt);
\end{scriptsize}
\end{tikzpicture}
		\begin{tikzpicture}[line cap=round,line join=round,>=triangle 45,x=1.0cm,y=1.0cm]
\clip(25.51,-9.07) rectangle (31.11,-2.57);
\fill[color=zzttqq,fill=zzttqq,fill opacity=0.1] (4,-4) -- (8,-4) -- (4,-2) -- cycle;
\fill[color=zzttqq,fill=zzttqq,fill opacity=0.1] (4,-2) -- (6,1) -- (8,-4) -- cycle;
\fill[color=zzttqq,fill=zzttqq,fill opacity=0.1] (4,-2) -- (6,1) -- (4,-4) -- cycle;
\fill[color=zzttqq,fill=zzttqq,fill opacity=0.1] (6,1) -- (4,-4) -- (8,-4) -- cycle;
\fill[color=zzttqq,fill=zzttqq,fill opacity=0.1] (16,-4) -- (20,-4) -- (16,-2) -- cycle;
\fill[color=zzttqq,fill=zzttqq,fill opacity=0.1] (16,-2) -- (18,1) -- (20,-4) -- cycle;
\fill[color=zzttqq,fill=zzttqq,fill opacity=0.1] (18,1) -- (16,-4) -- (20,-4) -- cycle;
\fill[color=zzttqq,fill=zzttqq,fill opacity=0.1] (16,-2) -- (12.56,-0.2) -- (16,-4) -- cycle;
\fill[color=qqqqff,fill=qqqqff,fill opacity=0.1] (12.56,-0.2) -- (15.16,-1.08) -- (16,-4) -- cycle;
\fill[color=qqqqff,fill=qqqqff,fill opacity=0.1] (15.16,-1.08) -- (16,-2) -- (12.56,-0.2) -- cycle;
\fill[color=qqqqff,fill=qqqqff,fill opacity=0.1] (15.16,-1.08) -- (16,-4) -- (16,-2) -- cycle;
\fill[color=qqqqff,fill=qqqqff,fill opacity=0.1] (12.56,-0.2) -- (13.78,-2.9) -- (16,-4) -- cycle;
\fill[color=qqqqff,fill=qqqqff,fill opacity=0.1] (13.78,-2.9) -- (16,-2) -- (12.56,-0.2) -- cycle;
\fill[color=qqqqff,fill=qqqqff,fill opacity=0.1] (13.78,-2.9) -- (16,-4) -- (16,-2) -- cycle;
\fill[color=zzttqq,fill=zzttqq,fill opacity=0.15] (26,-6) -- (28,-3) -- (30,-8) -- cycle;
\draw [color=zzttqq] (4,-4)-- (8,-4);
\draw [color=zzttqq] (4,-2)-- (4,-4);
\draw [color=zzttqq] (4,-2)-- (6,1);
\draw [color=zzttqq] (6,1)-- (8,-4);
\draw [dash pattern=on 1pt off 1pt,color=zzttqq] (8,-4)-- (4,-2);
\draw [color=zzttqq] (4,-2)-- (6,1);
\draw [color=zzttqq] (6,1)-- (4,-4);
\draw [color=zzttqq] (4,-4)-- (4,-2);
\draw [color=zzttqq] (6,1)-- (4,-4);
\draw [color=zzttqq] (4,-4)-- (8,-4);
\draw [color=zzttqq] (8,-4)-- (6,1);
\draw [color=zzttqq] (16,-4)-- (20,-4);
\draw [color=zzttqq] (16,-2)-- (16,-4);
\draw [color=zzttqq] (16,-2)-- (18,1);
\draw [color=zzttqq] (18,1)-- (20,-4);
\draw [dash pattern=on 1pt off 1pt,color=zzttqq] (20,-4)-- (16,-2);
\draw [color=zzttqq] (18,1)-- (16,-4);
\draw [color=zzttqq] (16,-4)-- (20,-4);
\draw [color=zzttqq] (20,-4)-- (18,1);
\draw [color=zzttqq] (16,-2)-- (16.67,-2.33);
\draw [shift={(16,-3)},dotted]  plot[domain=1.11:2.46,variable=\t]({1*4.47*cos(\t r)+0*4.47*sin(\t r)},{0*4.47*cos(\t r)+1*4.47*sin(\t r)});
\draw [color=zzttqq] (12.56,-0.2)-- (16,-4);
\draw [color=zzttqq] (16,-4)-- (16,-2);
\draw [shift={(16,-3)},dotted]  plot[domain=0.66:1.98,variable=\t]({1*2.12*cos(\t r)+0*2.12*sin(\t r)},{0*2.12*cos(\t r)+1*2.12*sin(\t r)});
\draw [color=qqqqff] (12.56,-0.2)-- (15.16,-1.08);
\draw [color=qqqqff] (15.16,-1.08)-- (16,-4);
\draw [color=qqqqff] (16,-4)-- (12.56,-0.2);
\draw [color=qqqqff] (15.16,-1.08)-- (16,-2);
\draw [color=qqqqff] (12.56,-0.2)-- (15.16,-1.08);
\draw [color=qqqqff] (15.16,-1.08)-- (16,-4);
\draw [color=qqqqff] (16,-4)-- (16,-2);
\draw [color=qqqqff] (16,-2)-- (15.16,-1.08);
\draw [color=qqqqff] (12.56,-0.2)-- (13.78,-2.9);
\draw [color=qqqqff] (13.78,-2.9)-- (16,-4);
\draw [color=qqqqff] (16,-4)-- (12.56,-0.2);
\draw [dash pattern=on 1pt off 1pt,color=qqqqff] (16,-2)-- (12.56,-0.2);
\draw [color=qqqqff] (12.56,-0.2)-- (13.78,-2.9);
\draw [color=qqqqff] (13.78,-2.9)-- (16,-4);
\draw [color=qqqqff] (16,-4)-- (16,-2);
\draw [dash pattern=on 1pt off 1pt,color=qqqqff] (16,-2)-- (13.78,-2.9);
\draw [color=zzttqq] (26,-8)-- (30,-8);
\draw [color=zzttqq] (26,-6)-- (26,-8);
\draw [color=zzttqq] (26,-6)-- (28,-3);
\draw [color=zzttqq] (28,-3)-- (30,-8);
\draw [color=zzttqq] (30,-8)-- (26,-6);
\draw [color=zzttqq] (26,-6)-- (28,-3);
\draw [color=zzttqq] (28,-3)-- (26,-8);
\draw [color=zzttqq] (26,-8)-- (26,-6);
\draw [color=zzttqq] (28,-3)-- (26,-8);
\draw [color=zzttqq] (26,-8)-- (30,-8);
\draw [color=zzttqq] (30,-8)-- (28,-3);
\draw [dash pattern=on 1pt off 1pt,color=qqqqff] (30,-4)-- (26,-6);
\draw [color=qqqqff] (30,-8)-- (30,-4);
\draw [color=qqqqff] (28,-3)-- (30,-4);
\draw [color=qqqqff] (30,-4)-- (30,-8);
\draw [color=qqqqff] (30,-8)-- (28,-3);
\draw [color=qqqqff] (26,-6)-- (28,-3);
\draw [color=qqqqff] (28,-3)-- (30,-4);
\draw [dash pattern=on 1pt off 1pt,color=ffqqqq] (28.46,-5.11)-- (27.23,-5.8);
\draw [color=qqqqff] (28.67,-4.68)-- (30,-4);
\draw [color=ffqqqq] (27.23,-5.8)-- (27.86,-5.45);
\draw (5.56,-1.25) node[anchor=north west] {$\TT_i$};
\draw (6.85,0.18) node[anchor=north west] {$\Gamma_{j_1}$};
\draw (6.99,-2.49) node[anchor=north west] {$\Gamma_{j_2}$};
\draw (4.02,-0.04) node[anchor=north west] {$\Gamma_{j_3}$};
\draw (4.63,-4.18) node[anchor=north west] {$\Gamma_{j_4}$};
\draw (14.4,0.1) node[anchor=north west] {$\QQ_j$};
\draw (16.02,-0.08) node[anchor=north west] {$i$};
\draw (17.82,-1.74) node[anchor=north west] {$\TT_i$};
\draw (13.21,-3.03) node[anchor=north west] {$\mathbb{\wp}(i,j)$};
\draw [->,dash pattern=on 1pt off 1pt] (28.02,-5.62) -- (28.66,-5.29);
\draw (27.13,-5.78) node[anchor=north west] {$\ell(j)$};
\draw (28.99,-4.46) node[anchor=north west] {$r(j)$};
\draw (28.36,-5.41) node[anchor=north west] {$\nv_{j}$};
\draw (30.3,-7.62) node[anchor=north west] {$\Gamma_j$};
\begin{scriptsize}
\fill [color=uuuuuu] (4,-4) circle (1.0pt);
\fill [color=uuuuuu] (8,-4) circle (1.0pt);
\fill [color=uuuuuu] (4,-2) circle (1.0pt);
\fill [color=uuuuuu] (6,1) circle (1.0pt);
\fill [color=uuuuuu] (16,-4) circle (1.0pt);
\fill [color=uuuuuu] (20,-4) circle (1.0pt);
\fill [color=uuuuuu] (16,-2) circle (1.0pt);
\fill [color=uuuuuu] (18,1) circle (1.0pt);
\fill [color=uuuuuu] (12.56,-0.2) circle (1.0pt);
\draw [color=ffqqqq] (17.68,-1.7) circle (1.5pt);
\draw [color=ffqqqq] (15.16,-1.08) circle (1.5pt);
\draw [color=ffqqqq] (13.78,-2.9) circle (1.5pt);
\fill [color=uuuuuu] (26,-8) circle (1.0pt);
\fill [color=uuuuuu] (30,-8) circle (1.0pt);
\fill [color=uuuuuu] (26,-6) circle (1.0pt);
\fill [color=uuuuuu] (28,-3) circle (1.0pt);
\fill [color=uuuuuu] (30,-4) circle (0.5pt);
\draw [color=ffqqqq] (28.46,-5.11) circle (1.5pt);
\draw [color=ffqqqq] (27.23,-5.8) circle (1.5pt);
\end{scriptsize}
\end{tikzpicture}
    \caption{A tetrahedral element of the primary mesh with $S_i=\{j_1, j_2, j_3, j_4\}$ (left) and the standard orientation used throughout this paper (right).}
    \label{fig.1}
		\end{center}
\end{figure}
\begin{figure}[ht]
    \begin{center}
		\begin{tabular}{cc} 
		\input{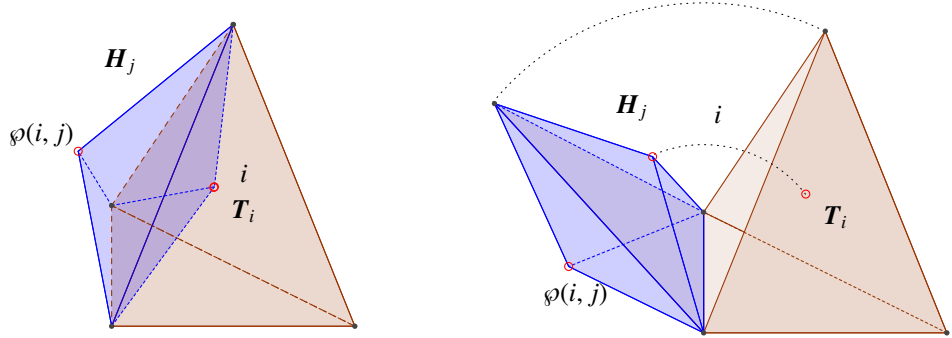} & 
		\begin{tikzpicture}[line cap=round,line join=round,>=triangle 45,x=0.8cm,y=0.8cm]
\clip(11.74,-4.85) rectangle (20.91,2.12);
\fill[color=zzttqq,fill=zzttqq,fill opacity=0.1] (4,-4) -- (8,-4) -- (4,-2) -- cycle;
\fill[color=zzttqq,fill=zzttqq,fill opacity=0.1] (4,-2) -- (6,1) -- (8,-4) -- cycle;
\fill[color=zzttqq,fill=zzttqq,fill opacity=0.1] (4,-2) -- (6,1) -- (4,-4) -- cycle;
\fill[color=zzttqq,fill=zzttqq,fill opacity=0.1] (6,1) -- (4,-4) -- (8,-4) -- cycle;
\fill[color=zzttqq,fill=zzttqq,fill opacity=0.1] (16,-4) -- (20,-4) -- (16,-2) -- cycle;
\fill[color=zzttqq,fill=zzttqq,fill opacity=0.1] (16,-2) -- (18,1) -- (20,-4) -- cycle;
\fill[color=zzttqq,fill=zzttqq,fill opacity=0.1] (18,1) -- (16,-4) -- (20,-4) -- cycle;
\fill[color=zzttqq,fill=zzttqq,fill opacity=0.1] (16,-2) -- (12.56,-0.2) -- (16,-4) -- cycle;
\fill[color=qqqqff,fill=qqqqff,fill opacity=0.1] (12.56,-0.2) -- (15.16,-1.08) -- (16,-4) -- cycle;
\fill[color=qqqqff,fill=qqqqff,fill opacity=0.1] (15.16,-1.08) -- (16,-2) -- (12.56,-0.2) -- cycle;
\fill[color=qqqqff,fill=qqqqff,fill opacity=0.1] (15.16,-1.08) -- (16,-4) -- (16,-2) -- cycle;
\fill[color=qqqqff,fill=qqqqff,fill opacity=0.1] (12.56,-0.2) -- (13.78,-2.9) -- (16,-4) -- cycle;
\fill[color=qqqqff,fill=qqqqff,fill opacity=0.1] (13.78,-2.9) -- (16,-2) -- (12.56,-0.2) -- cycle;
\fill[color=qqqqff,fill=qqqqff,fill opacity=0.1] (13.78,-2.9) -- (16,-4) -- (16,-2) -- cycle;
\fill[color=zzttqq,fill=zzttqq,fill opacity=0.15] (26,-6) -- (28,-3) -- (30,-8) -- cycle;
\draw [color=zzttqq] (4,-4)-- (8,-4);
\draw [color=zzttqq] (4,-2)-- (4,-4);
\draw [color=zzttqq] (4,-2)-- (6,1);
\draw [color=zzttqq] (6,1)-- (8,-4);
\draw [dash pattern=on 1pt off 1pt,color=zzttqq] (8,-4)-- (4,-2);
\draw [color=zzttqq] (4,-2)-- (6,1);
\draw [color=zzttqq] (6,1)-- (4,-4);
\draw [color=zzttqq] (4,-4)-- (4,-2);
\draw [color=zzttqq] (6,1)-- (4,-4);
\draw [color=zzttqq] (4,-4)-- (8,-4);
\draw [color=zzttqq] (8,-4)-- (6,1);
\draw [color=zzttqq] (16,-4)-- (20,-4);
\draw [color=zzttqq] (16,-2)-- (16,-4);
\draw [color=zzttqq] (16,-2)-- (18,1);
\draw [color=zzttqq] (18,1)-- (20,-4);
\draw [dash pattern=on 1pt off 1pt,color=zzttqq] (20,-4)-- (16,-2);
\draw [color=zzttqq] (18,1)-- (16,-4);
\draw [color=zzttqq] (16,-4)-- (20,-4);
\draw [color=zzttqq] (20,-4)-- (18,1);
\draw [color=zzttqq] (16,-2)-- (16.67,-2.33);
\draw [shift={(16,-3)},dotted]  plot[domain=1.11:2.46,variable=\t]({1*4.47*cos(\t r)+0*4.47*sin(\t r)},{0*4.47*cos(\t r)+1*4.47*sin(\t r)});
\draw [color=zzttqq] (12.56,-0.2)-- (16,-4);
\draw [color=zzttqq] (16,-4)-- (16,-2);
\draw [shift={(16,-3)},dotted]  plot[domain=0.66:1.98,variable=\t]({1*2.12*cos(\t r)+0*2.12*sin(\t r)},{0*2.12*cos(\t r)+1*2.12*sin(\t r)});
\draw [color=qqqqff] (12.56,-0.2)-- (15.16,-1.08);
\draw [color=qqqqff] (15.16,-1.08)-- (16,-4);
\draw [color=qqqqff] (16,-4)-- (12.56,-0.2);
\draw [color=qqqqff] (15.16,-1.08)-- (16,-2);
\draw [color=qqqqff] (12.56,-0.2)-- (15.16,-1.08);
\draw [color=qqqqff] (15.16,-1.08)-- (16,-4);
\draw [color=qqqqff] (16,-4)-- (16,-2);
\draw [color=qqqqff] (16,-2)-- (15.16,-1.08);
\draw [color=qqqqff] (12.56,-0.2)-- (13.78,-2.9);
\draw [color=qqqqff] (13.78,-2.9)-- (16,-4);
\draw [color=qqqqff] (16,-4)-- (12.56,-0.2);
\draw [dash pattern=on 1pt off 1pt,color=qqqqff] (16,-2)-- (12.56,-0.2);
\draw [color=qqqqff] (12.56,-0.2)-- (13.78,-2.9);
\draw [color=qqqqff] (13.78,-2.9)-- (16,-4);
\draw [color=qqqqff] (16,-4)-- (16,-2);
\draw [dash pattern=on 1pt off 1pt,color=qqqqff] (16,-2)-- (13.78,-2.9);
\draw [color=zzttqq] (26,-8)-- (30,-8);
\draw [color=zzttqq] (26,-6)-- (26,-8);
\draw [color=zzttqq] (26,-6)-- (28,-3);
\draw [color=zzttqq] (28,-3)-- (30,-8);
\draw [color=zzttqq] (30,-8)-- (26,-6);
\draw [color=zzttqq] (26,-6)-- (28,-3);
\draw [color=zzttqq] (28,-3)-- (26,-8);
\draw [color=zzttqq] (26,-8)-- (26,-6);
\draw [color=zzttqq] (28,-3)-- (26,-8);
\draw [color=zzttqq] (26,-8)-- (30,-8);
\draw [color=zzttqq] (30,-8)-- (28,-3);
\draw [dash pattern=on 1pt off 1pt,color=qqqqff] (30,-4)-- (26,-6);
\draw [color=qqqqff] (30,-8)-- (30,-4);
\draw [color=qqqqff] (28,-3)-- (30,-4);
\draw [color=qqqqff] (30,-4)-- (30,-8);
\draw [color=qqqqff] (30,-8)-- (28,-3);
\draw [color=qqqqff] (26,-6)-- (28,-3);
\draw [color=qqqqff] (28,-3)-- (30,-4);
\draw [dash pattern=on 1pt off 1pt,color=ffqqqq] (28.46,-5.11)-- (27.23,-5.8);
\draw [color=qqqqff] (28.67,-4.68)-- (30,-4);
\draw [color=ffqqqq] (27.23,-5.8)-- (27.86,-5.45);
\draw (5.56,-1.25) node[anchor=north west] {$\TT_i$};
\draw (6.85,0.18) node[anchor=north west] {$\Gamma_{j_1}$};
\draw (6.99,-2.49) node[anchor=north west] {$\Gamma_{j_2}$};
\draw (4.02,-0.04) node[anchor=north west] {$\Gamma_{j_3}$};
\draw (4.63,-4.18) node[anchor=north west] {$\Gamma_{j_4}$};
\draw (14.4,0.1) node[anchor=north west] {$\QQ_j$};
\draw (16.02,-0.08) node[anchor=north west] {$i$};
\draw (17.82,-1.74) node[anchor=north west] {$\TT_i$};
\draw (13.21,-3.03) node[anchor=north west] {$\mathbb{\wp}(i,j)$};
\draw [->,dash pattern=on 1pt off 1pt] (28.02,-5.62) -- (28.66,-5.29);
\draw (27.13,-5.78) node[anchor=north west] {$\ell(j)$};
\draw (28.99,-4.46) node[anchor=north west] {$r(j)$};
\draw (28.36,-5.41) node[anchor=north west] {$\vec{n_j}$};
\draw (30.3,-7.62) node[anchor=north west] {$\Gamma_j$};
\begin{scriptsize}
\fill [color=uuuuuu] (4,-4) circle (1.0pt);
\fill [color=uuuuuu] (8,-4) circle (1.0pt);
\fill [color=uuuuuu] (4,-2) circle (1.0pt);
\fill [color=uuuuuu] (6,1) circle (1.0pt);
\fill [color=uuuuuu] (16,-4) circle (1.0pt);
\fill [color=uuuuuu] (20,-4) circle (1.0pt);
\fill [color=uuuuuu] (16,-2) circle (1.0pt);
\fill [color=uuuuuu] (18,1) circle (1.0pt);
\fill [color=uuuuuu] (12.56,-0.2) circle (1.0pt);
\draw [color=ffqqqq] (17.68,-1.7) circle (1.5pt);
\draw [color=ffqqqq] (15.16,-1.08) circle (1.5pt);
\draw [color=ffqqqq] (13.78,-2.9) circle (1.5pt);
\fill [color=uuuuuu] (26,-8) circle (1.0pt);
\fill [color=uuuuuu] (30,-8) circle (1.0pt);
\fill [color=uuuuuu] (26,-6) circle (1.0pt);
\fill [color=uuuuuu] (28,-3) circle (1.0pt);
\fill [color=qqqqff] (30,-4) circle (1.5pt);
\draw[color=qqqqff] (0.19,5.94) node {$L$};
\draw [color=ffqqqq] (28.46,-5.11) circle (1.5pt);
\draw [color=ffqqqq] (27.23,-5.8) circle (1.5pt);
\end{scriptsize}
\end{tikzpicture}
		\end{tabular} 
    \caption{An example of a dual element (a non-standard 5-point hexahedron, highlighted in blue) associated with the face $\Gamma_j$.}
    \label{fig.1.1}
		\end{center}
\end{figure}
We exdend our definitions on the main grid to the dual one, namely: $N_l$ is the total amount of sides of $\QQ_j$; $\Gamma_l$ indicates the $l$-th side; $\forall j$, the set of sides $l$ of $j$ is indicated with $S_j$; $\forall l$, $\ell_{jl}(l)$ and $r_{jl}(l)$ are the left and the right hexahedral element, respectively; $\nv_{l}$ is the standard normal vector defined on $l$ and assumed positive with respect to the standard orientation on $l$ (defined, as for the main grid, from the left to the right).

In the time direction we cover the time interval $[0,T]$ with a sequence of times $0=t^0<t^1<t^2 \ldots <t^N<t^{N+1}=T$. We denote the time step between $t^n$ and $t^{n+1}$ by 
$\Delta t^{n+1} = t^{n+1}-t^{n} $ and the associated time interval by $T^{n+1}=[t^{n}, t^{n+1}]$, for $n=0 \ldots N$. In order to ease the notation, sometimes we will simply
write $\Delta t= \Delta t^{n+1}$. In this way the generic space-time element defined in the time interval $[t^n, t^{n+1}]$ is given by $\TT_i^\st=\TT_i \times T^{n+1}$ 
for the main grid and $\QQ_j^\st=\QQ_j \times T^{n+1}$ for the dual grid. 
% *********************************************************************************************
% ************************ SECTION 1.2 Basis Functions ****************************************
% *********************************************************************************************

\subsection{Space-time basis functions}
\label{sec222}
We first construct the spatial basis functions and then we extend them to the time direction using a simple tensor product. For tetrahedral elements, the basis functions are generated on a
standard reference tetrahedron, defined by $T_{ref}=\{(\xi,\eta,\zeta) \in \R^{3} \,\, | \,\, 0 \leq \xi+\eta+\zeta \leq 1 \}$. 
We write the basis function on the reference element as 
\begin{equation}
	\phi_k(\xi,\eta,\zeta)= \sum_{r_1=0}^p  \sum_{r_2=0}^{p-r_1} \sum_{r_3=0}^{p-r_2-r_1}   \alpha_{kr} \, \xi^{r_1}\eta^{r_2}\zeta^{r_3} := \alpha_{kr} \boldsymbol{\xi}^r, 
\label{eq:bf_1}
\end{equation}
for some coefficients $\alpha_{kr}$ and the multi-index $r=(r_1,r_2,r_3)$. We then define $N_\phi=\frac{(p+1)(p+2)(p+3)}{6}$ nodal points $\boldsymbol{\xi}_j=(\xi_{j_1},\eta_{j_2},\zeta_{j_3})=(j_1/p,j_2/p,j_3/p)$, with the multi-index $j=(j_1,j_2,j_3)$ and $0 \leq j_1+j_2+j_3 \leq p$, as in standard conforming finite elements. We then impose the classical interpolation condition for nodal finite elements 
$\phi_k(\boldsymbol{\xi}_j) = \delta_{kj}$, with the usual Kronecker symbol $\delta_{kj}$. This means that we have chosen a nodal basis which is defined by the Lagrange interpolation polynomials that pass through the nodes given by the  standard nodes of conforming finite elements. This leads to the linear system $\alpha_{kr} \boldsymbol{\xi}_{\!j}^r = \delta_{kj}$ for the coefficients $\alpha_{kr}$ that can be solved analytically for every polynomial degree $p$ on the reference 
tetrahedron. 
In this way we obtain $N_\phi$ basis functions on $T_{ref}$, $\{\phi_k\}_{k \in [1, N_\phi]}$. The connection between the reference coordinates $\boldsymbol{\xi}$ and the physical coordinates 
$\mathbf{x}$ is performed by the map $T(\cdot,\TT_i)=T_i:\TT_i \longrightarrow T_{ref}$ for every $i =1 \ldots \Ni$ and its inverse, called $T^{-1}(\cdot,\TT_i)=T_i^{-1}:\TT_i \longleftarrow T_{ref}$. The maps from the physical to the reference coordinates can be constructed following a classical sub-parametric or a complete iso-parametric approach and in general we will write, for all $i=1\ldots \Ni$, $\phi^{(i)}_k(x,y,z)=\phi_k(T_i(x,y,z))$.

\par Unfortunately, it is not so easy to construct a similar nodal basis on the dual mesh, due to the use of non-standard 5-point hexahedral elements. 
As discussed in \cite{Chatzi2000}, the definition of basis functions based on Lagrange interpolation polynomials on this kind of element is problematic, 
since for special configurations of the vertex coordinates of the dual elements, the linear system to be solved for the classical interpolation condition 
of a nodal basis can become singular. This does not allow the construction of a nodal polynomial basis for a generic element $\QQ_j$ and therefore one 
has to pass to rational functions of polynomials instead of using simple polynomial functions in that case.

Therefore, for the basis functions on the dual grid directly we choose a simple Taylor-type modal basis \cite{LuoTaylor} directly in the physical space, hence the basis functions will 
consequently depend on the element $j \in [1, \Nj]$. The basis functions read  
\begin{equation}
	\psi_k^{(j)}(\xx)= \frac{\left(x-x_0^{(j)}\right)^{k_1}\left(y-y_0^{(j)}\right)^{k_2}\left(z-z_0^{(j)}\right)^{k_3}}{h_j^{k_1+k_2+k_3}},
	\label{eq:bf_2}
\end{equation} 
where $\xx_0^j=(x,y,z)_0^{(j)}$ is the center of the dual element and $h_j$ is a characteristic length of $\QQ_j$ used for scaling the basis. 
Here, $0 \leq k_1 + k_2 + k_3 \leq p$, i.e. we use the optimal number of polynomials of degree $p$ in three space dimensions, namely $N_\psi = N_\phi$. 
With this choice we get only a modal basis for the dual hexaxedral elements, i.e. if the convective term is directly computed on the dual mesh according to 
the natural extension of the method proposed in \cite{2STINS}, then it has to be computed according to a modal approach, which is more expensive than 
a nodal one. 

Finally, the time basis functions  are constructed on a reference interval $I=[0,1]$ for polynomials of degree $p_\gamma$. 
In this case the resulting $N_\gamma=p_\gamma+1$ basis functions $\{\gamma_k\}_{k \in [1, N_\gamma]}$ are defined as the Lagrange interpolation polynomials passing through the Gauss-Legendre quadrature points for the unit interval. For every time interval $[t^n, t^{n+1}]$, the map between the reference interval and the physical one is simply given by $t=t^n+\tau \Delta t^{n+1}$ for every $\tau \in [0,1]$.
Using the tensor product we can finally construct the basis functions on the space-time elements $\TT_i^\st$ and $\QQ_j^\st$ such as $\tilde{\phi}(\xi,\eta,\zeta,\tau)=\phi(\xi, \eta, \zeta) \cdot \gamma(\tau)$ and $\tilde{\psi}^{(j)}(x, y,z,t)=\psi^{(j)}(x, y, z) \cdot \gamma(\tau(t))$. The total number of basis functions becomes $N_\phi^\st=N_\phi \cdot N_\gamma$ and $N_\psi^\st=N_\psi \cdot N_\gamma$. 

%By introducing two sorting functions $\ell_1(\,\, ,N_\cdot^\st):[1,N_\cdot^\st] \rightarrow [1,N_\cdot]$ and $\ell_2(\,\, ,N_\cdot^\st):[1,N_\cdot^\st] \rightarrow [1,N_\gamma]$, defined as
%\begin{eqnarray}
    %\ell_2(k,N) &=& int\left[\frac{k-1}{N}\right]+1 \nonumber \\
    %\ell_1(k,N) &=& k-(\ell_2(k,N)-1)\cdot N
%\label{eq:DST0}
%\end{eqnarray}

 %we can explicit the form of $\tilde{\phi}_k$ and $\tilde{\psi}_l$ for $k=1 \ldots N_\phi^\st$ and $l=1 \ldots N_\psi^\st$ in terms of space and time basis functions:
%\begin{eqnarray}
    %\tilde{\phi}_k(\xi,\gamma,\delta,\tau) &=& \phi_{\ell_1(k,N_\phi^\st)}(\xi,\gamma,\delta) \gamma_{\ell_2(k,N_\phi^\st)}(\tau) \qquad \forall k\in [1, N_\phi^\st] \nonumber \\
    %\tilde{\psi}_k(x,y,z,t) &=& \psi_{\ell_1(k,N_\psi^\st)}(x,y,z) \gamma_{\ell_2(k,N_\psi^\st)}(\tau(t)) \qquad \forall k\in [1, N_\psi^\st] \nonumber
%\label{eq:DST1}
%\end{eqnarray}
%Remark how $\ell_2$ can be seen such as a layer selector function, so all the indexes $k$ such that $l_2(k,\cdot)=l$ represent all the spacial degrees of freedom (DoF) at the layer $l$, for every fixed $l=1 \ldots N_\gamma$. In the same way $l_1(k,\cdot)=m$ represents the time evolution of the DoF $m$ inside the space-time element $\TT_i^\st$. 

\subsection{Staggered semi-implicit space-time DG scheme}
\label{sec_semi_imp_dg}
The discrete pressure $p_h$ is defined on the main grid, namely $p_h(\xx,t)|_{\TT_i^\st}=p_i(\xx,t)=p_i$, while the discrete
velocity vector field $\mathbf{v}_h$ is defined on the dual grid, namely  $\mathbf{v}_h(\xx,t)|_{\QQ_j^\st}=\mathbf{v}_j(\xx,t)=\mathbf{v}_j$ . \par

The numerical solution of \eref{eq:CS_2}-\eref{eq:CS_2_2} is represented in each space-time element $\TT_i^\st$ and $\QQ_j^\st$ between times $t^n$ and $t^{n+1}$ by piecewise 
polynomials as  
\begin{equation}
	p_i(\xx,t)=\sum\limits_{l=1}^{N_\phi^\st} \tilde{\phi}_l^{(i)}(\xx,t)\hat{p}_{l,i}^{n+1}=:\tilde{\bphi}^{(i)}(\xx,t)\etah_i^{n+1},
\label{eq:D_1}
\end{equation}
\begin{equation}
	\mathbf{v}_j(\xx,t)=\sum\limits_{l=1}^{N_\psi^\st} \tilde{\psi}_l^{(j)}(\xx,t) \hat{\mathbf{v}}_{l,j}^{n+1}=:\tilde{\bpsi}^{(j)}(\xx,t)\vvh_j^{n+1},
\label{eq:D_3}
\end{equation}
where the vector of basis functions $\tilde{\bphi}(\xx,t)$ is generated from $\tilde{\bphi}(\xi,\eta,\zeta,\tau)$ on $T_{std}\times [0,1]$ while $\tilde{\bpsi}^{(j)}(\xx,t)$ is 
defined for every $j\in [1 \ldots \Nj]-\B(\Omega)$ directly in the physical space.

A weak formulation of the continuity equation \eref{eq:CS_2} is obtained by multiplying it with a test function $\tilde{\phi}_k^{(i)}$ and integrating over the space-time 
control volume $\TT_i^\st$, for every $k=1\ldots N_\phi^\st$. The resulting weak formulation reads
\begin{equation}
\int\limits_{\TT_i^\st}{\tilde{\phi}_k^{(i)} \nabla \cdot \mathbf{v} \, d\xx  dt}=0,
\label{eq:CS_4}
\end{equation}
with $d\xx = dx dy dz$. 
Similarly, multiplication of the momentum equation \eref{eq:CS_2_2} by the test function $\tilde{\psi}_k^{(j)}$ and integrating over a control volume $\QQ_j^\st$ yields 
\begin{equation}
\int\limits_{\QQ_j^\st}{ \tilde{\psi}_k^{(j)}\left( \diff{\mathbf{v}}{t}+\nabla \cdot \mathbf{\TF} \right)  \, d\xx  dt}+\int\limits_{\QQ_j^\st}{\tilde{\psi}_k^{(j)} \nabla \np \, \, d\xx  dt}=\int\limits_{\QQ_j^\st}{\tilde{\psi}_k^{(j)} \mathbf{S} \, d\xx  dt},
\label{eq:CS_5}
\end{equation}
for every $j=1 \ldots \Nj$ and $k=1 \ldots N_\psi^\st$. Using integration by parts Eq. \eref{eq:CS_4} reads 
\begin{equation}
\oint\limits_{\partial \TT_i^\st}{\tilde{\phi}_k^{(i)} \mathbf{v} \cdot \nv_{i} \, dS dt }-\int\limits_{\TT_i^\st}{\nabla \tilde{\phi}_k^{(i)} \cdot \mathbf{v} \, d\xx dt }  =0,
\label{eq:CS_6}
\end{equation}
where $\nv_{i}$ indicates the outward pointing unit normal vector.
%We can extend the velocity vector field to a space-time vector $\mathbf{v}=(u,v,w,0)$ and hence use the divergence theorem on the space-time control volume $\TT_i^\st$. In this way, the divergence free condition is intrinsically required at every fixed time since we do not have time fluxes.
Due to the discontinuity of $p_h$ and $\mathbf{v}_h$, equations \eref{eq:CS_5} and \eref{eq:CS_6} have to be split as follows:

\begin{equation}
\sum\limits_{j \in S_i}\left( \int\limits_{\Gamma_j^\st}{\tilde{\phi}_k^{(i)} \mathbf{v}_j \cdot \nextud \, dS dt}-\int\limits_{\TT_{i,j}^\st}{\nabla \tilde{\phi}_k^{(i)} \cdot \mathbf{v}_j \, d\xx  dt}  \right)=0,
\label{eq:CS_8}
\end{equation}
and
\begin{equation}
\int\limits_{\QQ_j^\st}{\tilde{\psi}_k^{(j)}\left( \diff{\mathbf{v}_j}{t}+\nabla \cdot \mathbf{\TF} \right)  \, d\xx  dt}
+\hspace{-3mm} \int\limits_{\TT_{\ell(j),j}^\st}{\tilde{\psi}_k^{(j)} \nabla \np_{\ell(j)} \, d\xx  dt} 
+\hspace{-3mm} \int\limits_{\TT_{r(j),j}^\st}{\tilde{\psi}_k^{(j)} \nabla \np_{r(j)} \, d\xx  dt} + 
\int\limits_{\Gamma_j^\st}{\tilde{\psi}_k^{(j)} \left(\np_{r(j)}-\np_{\ell(j)}\right) \nstd \, dS dt}=\int\limits_{\QQ_j^\st}{\tilde{\psi}_k^{(j)} \mathbf{S} \, d\xx  dt},
\label{eq:CS_9}
\end{equation}
where $\nextud=\nv_{i}|_{\Gamma_j^\st}$; $\TT_{i,j}^\st=\TT_{i,j} \times T^{n+1}$; and $\Gamma_j^\st=\Gamma_j \times T^{n+1}$. Note that the pressure has a discontinuity along $\Gamma_j^{st}$ 
inside the hexahedral element $\QQ_j^\st$ and hence the pressure gradient in \eref{eq:CS_5} needs to be interpreted in the sense of distributions, as in path-conservative finite volume
schemes \cite{Castro2006,Pares2006}. This leads to the jump terms present in \eref{eq:CS_9}, see \cite{2STINS}. Alternatively, the same jump term can be produced also via forward and backward 
integration by parts, see e.g. the well-known work of Bassi and Rebay \cite{BassiRebay}. 
%On the other hand, since the term $\nabla \phi_\cdot$ is continuous on the main tetrahedral mesh, it can be conveniently written such as sum of intersection contributions, see Eq. \eref{eq:CS_8}.
Using definitions \eref{eq:D_1} and \eref{eq:D_3}, we rewrite the above equations as
\begin{eqnarray}
\sum\limits_{j \in S_i}\left(\int\limits_{\Gamma_j^\st}{\tilde{\phi}_k^{(i)}\tilde{\psi}_l^{(j)} \nextud dS dt}\, \cdot \hat{\mathbf{v}}_{l,j}^{n+1} -\int\limits_{\TT_{i,j}^\st}{\nabla \tilde{\phi}_k^{(i)}\tilde{\psi}_l^{(j)}\, d\xx  dt} \, \cdot \hat{\mathbf{v}}_{l,j}^{n+1} \right)=0,
\label{eq:CS_10}
\end{eqnarray}
and
\begin{eqnarray}
\int\limits_{\QQ_j^\st}{\tilde{\psi}_k^{(j)}  \diff{\mathbf{v}_j}{t}  \, d\xx  dt} + 
\int\limits_{\QQ_j^\st}{\tilde{\psi}_k^{(j)} \nabla \cdot \mathbf{\TF}   \, d\xx  dt} 
+\int\limits_{\TT_{\ell(j),j}^\st}{\tilde{\psi}_k^{(j)} \nabla \tilde{\phi}_{l}^{(\ell(j))}  \, d\xx  dt} \,  \, \hat \np_{l,\ell(j)}^{n+1}
+\int\limits_{\TT_{r(j),j}^\st}{   \tilde{\psi}_k^{(j)} \nabla \tilde{\phi}_{l}^{(r(j))}     \, d\xx  dt} \,  \, \hat \np_{l,r(j)}^{n+1}    \nonumber \\
 +\int\limits_{\Gamma_j^\st}{\tilde{\psi}_k^{(j)} \tilde{\phi}_{l}^{(r(j))}    \nstd dS dt} \,  \hat \np_{l,r(j)}^{n+1}
 -\int\limits_{\Gamma_j^\st}{\tilde{\psi}_k^{(j)} \tilde{\phi}_{l}^{(\ell(j))} \nstd dS dt} \,  \hat \np_{l,\ell(j)}^{n+1}=\int\limits_{\QQ_j^\st}{\tilde{\psi}_k^{(j)} \mathbf{S} \, d\xx  dt}, \nonumber \\
\label{eq:CS_11_1}
\end{eqnarray}
where we have used the standard summation convention for the repeated index $l$.
Integrating the first integral in \eref{eq:CS_11_1} by parts in time we obtain 
\begin{equation}
\int\limits_{\QQ_j^\st}{\tilde{\psi}_k^{(j)} \diff{\mathbf{v}_j}{t} \, d\xx  dt} = 
  \int\limits_{\QQ_j}{\tilde{\psi}_k^{(j)}(\mathbf{x},t^{n+1}) \mathbf{v}_j(\mathbf{x},t^{n+1}) \, d\xx  } 
- \int\limits_{\QQ_j}{\tilde{\psi}_k^{(j)}(\mathbf{x},t^{n})   \mathbf{v}_j(\mathbf{x},t^{n }) \, d\xx   }  
- \int\limits_{\QQ_j^\st}{ \diff{\tilde{\psi}_k^{(j)}}{t} \mathbf{v}_j}(\xx,t) \, d\xx  dt.  
\label{eq:CS_11_2}
\end{equation}
In Eq. \eref{eq:CS_11_2} we can recognize the fluxes between the current space-time element $\QQ_j \times T^{n+1}$, the future and the past space-time elements, as well as an internal contribution 
that connects in an asymmetric way the degrees of freedom inside the element $\QQ_j^\st$. Note that the asymmetry appears only in the volume contribution in \eref{eq:CS_11_2}. For the spatial
integral at time $t^n$ we will insert the boundary-extrapolated numerical solution from the previous time step, which corresponds to upwinding in time direction due to the causality principle.   
%and is due to the nature of the the time derivative that has a natural positive direction. Consequently, since it is possible to require a symmetric action between degree of freedom at the same time as shown in \cite{2STINS}, it is not possible to recovery the same properties between DoF at different times but in the same space-time control volume.
By substituting Eq. \eref{eq:CS_11_2} into \eref{eq:CS_11_1} and using the causality principle, we obtain the following weak formulation of the momentum equation: 
\begin{eqnarray}
\left( \int\limits_{\QQ_j}{\tilde{\psi}_k^{(j)}(\xx,t^{n+1}) \tilde{\psi}_l^{(j)}(\xx,t^{n+1}) \, d\xx} - 
       \int\limits_{\QQ_j^\st}{\diff{\tilde{\psi}_k^{(j)}}{t} \tilde{\psi}_l^{(j)} \, d\xx  dt} \right)  \hat{\mathbf{v}}_{l,j}^{n+1} 
 -     \int\limits_{\QQ_j}{\tilde{\psi}_k^{(j)}(\xx,t^{n}) \tilde{\psi}_l^{(j)}(\xx,t^{n}) \, d\xx  } \, \hat{\mathbf{v}}_{l,j}^{n} +
\int\limits_{\QQ_j^\st}{\tilde{\psi}_k^{(j)} \nabla \cdot \mathbf{\TF} \, d\xx } \nonumber \\
+\int\limits_{\TT_{\ell(j),j}^\st}{\tilde{\psi}_k^{(j)} \nabla \tilde{\phi}_{l}^{(\ell(j))}  \, d\xx } \,  \, \hat \np_{l,\ell(j)}^{n+1}
+\int\limits_{\TT_{r(j),j}^\st}{   \tilde{\psi}_k^{(j)} \nabla \tilde{\phi}_{l}^{(r(j))}     \, d\xx } \,  \, \hat \np_{l,r(j)}^{n+1}   
 +\int\limits_{\Gamma_j^\st}{\tilde{\psi}_k^{(j)} \tilde{\phi}_{l}^{(r(j))}    \nstd dS} \,  \hat \np_{l,r(j)}^{n+1}
 -\int\limits_{\Gamma_j^\st}{\tilde{\psi}_k^{(j)} \tilde{\phi}_{l}^{(\ell(j))} \nstd dS} \,  \hat \np_{l,\ell(j)}^{n+1} \nonumber \\ 
=\int\limits_{\QQ_j^\st}{\tilde{\psi}_k^{(j)} \mathbf{S} \, d\xx  dt},  
\label{eq:CS_11}
\end{eqnarray}
For every $i$ and $j$, Eqs. \eref{eq:CS_10} and \eref{eq:CS_11} can be written in a compact matrix form as
\begin{eqnarray}
    \sum\limits_{j \in S_i}\D_{i,j}\vvh_j^{n+1}=0 \label{eq:CS_12},
\end{eqnarray}
and
\begin{eqnarray}
    \left(\Mpsi_j^+ - \Mpsi_j^\circ \right) \vvh_j^{n+1} - \Mpsi_j^-\vvh_j^{n} + \Upsilon_j(\mathbf{v}, \nabla \mathbf{v}) + \RM_j \etah_{r(j)} ^{n+1} - \LM_j \etah_{\ell(j)}^{n+1} =\Ss_j, \label{eq:CS_12_1}
\end{eqnarray}
respectively, where:
\begin{eqnarray}
	\Mpsi_j^+ &=& \int\limits_{\QQ_j}{\tilde{\psi}_k^{(j)}(\xx,t(1))\tilde{\psi}_l^{(j)}(\xx,t(1))  \, d\xx }, \label{eq:MD_2} \\
    \Mpsi_j^- &=& \int\limits_{\QQ_j}{\tilde{\psi}_k^{(j)}(\xx,t(0))\tilde{\psi}_l^{(j)}(\xx,t(1))  \, d\xx }, \label{eq:MD_2_1} \\
    \Mpsi_j^\circ &=& \int\limits_{\QQ_j^\st}{\diff{\tilde{\psi}_k^{(j)}}{t} \tilde{\psi}_l^{(j)} \, d\xx  dt}, \label{eq:MD_2_2} \\
    \Upsilon_j &=& \int\limits_{\QQ_j^\st}{\tilde{\psi}_k^{(j)} \nabla \cdot \mathbf{\TF} \, d\xx  dt}
\end{eqnarray}

\begin{equation}
	\D_{i,j}=\int\limits_{\Gamma_j^\st}{\tilde{\phi}_k^{(i)}\tilde{\psi}_l^{(j)}\nextud dS dt}-\int\limits_{\TT_{i,j}^\st}{\nabla \tilde{\phi}_k^{(i)}\tilde{\psi}_l^{(j)}\, d\xx  dt},
\label{eq:MD_3}
\end{equation}

\begin{equation}
	\RM_{j}=\int\limits_{\Gamma_j^\st}{\tilde{\psi}_k^{(j)} \tilde{\phi}_{l}^{(r(j))}\nstd dS dt}+\int\limits_{\TT_{r(j),j}^\st}{\tilde{\psi}_k^{(j)} \nabla \tilde{\phi}_{l}^{(r(j))}  \, d\xx  dt},
\label{eq:MD_4}
\end{equation}

\begin{equation}
	\LM_{j}=\int\limits_{\Gamma_j^\st}{\tilde{\psi}_k^{(j)} \tilde{\phi}_{l}^{(\ell(j))}\nstd dS dt}-\int\limits_{\TT_{\ell(j),j}^\st}{\tilde{\psi}_k^{(j)} \nabla \tilde{\phi}_{l}^{(\ell(j))}  \, d\xx  dt},
\label{eq:MD_5}
\end{equation}

\begin{equation}
	\Ss_j=\int\limits_{\QQ_j^\st}{\tilde{\psi}_k^{(j)} \mathbf{S} \, d\xx  dt}.
\label{eq:MD_5_2}
\end{equation}
%{\color{red}
%Here the following convention for vector value matrix and tensor is used: let ${\bf M}:\R^n \times \R^m \rightarrow \R^2$ and ${\bf v}:\R^m \rightarrow \R^2$, if $M_{ij}=m=(m^x,m^y)\in \R^2$ %and $v_j=v=(v^x,v^y)$, then $({\bf Mv})_{ij}=m \cdot v = m^xv^x+m^yv^y \in \R$
%}
%The notation for the vector, matrix and tensor multiplications used here becomes clear from \eref{eq:CS_10}-\eref{eq:CS_11} and follows the one introduced in \cite{DumbserCasulli}.
Note how $\Mpsi_j^\circ$ introduces, for $p_\gamma>0$, an asymmetric contribution that will lead to an asymmetry of the main system for the discrete pressure. 
%This is due to the physics of the equations that does not allow a symmetric action on two different space-time degree of freedoms but only on space DoF.
The action of matrices $\LM$ and $\RM$ can be generalized by introducing the new matrix $\Q_{i,j}$, defined as
\begin{equation}
	\Q_{i,j}=\int\limits_{\TT_{i,j}^\st}{\tilde{\psi}_k^{(j)} \nabla \tilde{\phi}_{l}^{(i)}  \, d\xx  dt}-\int\limits_{\Gamma_j^\st}{\tilde{\psi}_k^{(j)} \tilde{\phi}_{l}^{(i)}\sigma_{i,j} \nstd ds dt},
\label{eq:MD_6}
\end{equation}
where $\sigma_{i,j}$ is a sign function defined by
\begin{equation}
	\sigma_{i,j}=\frac{r(j)-2i+\ell(j)}{r(j)-\ell(j)}.
\label{eq:SD_1}
\end{equation}
In this way $\Q_{\ell(j),j}=-\LM_j$ and $\Q_{r(j),j}=\RM_j$, and then Eq. \eref{eq:CS_12_1} becomes in terms of $\Q$
\begin{equation}
	\left(\Mpsi_j^+ - \Mpsi_j^\circ \right) \vvh_j^{n+1} - \Mpsi_j^-\vvh_j^{n} + \Upsilon_j(\mathbf{v}, \nabla \mathbf{v})  + \Q_{r(j),j}  \etah_{r(j)}^{n+1} + \Q_{\ell(j),j} \etah_{\ell(j)}^{n+1} =\Ss_j,
\label{eq:CS_12_2}
\end{equation}
or, equivalently,
\begin{equation}
\left(\Mpsi_j^+ - \Mpsi_j^\circ \right) \vvh_j^{n+1} - \Mpsi_j^-\vvh_j^{n} + \Upsilon_j(\mathbf{v}, \nabla \mathbf{v})  + \Q_{i,j} \etah_{i}^{n+1} + \Q_{\p(i,j),j} \etah_{\p(i,j)}^{n+1} =\Ss_j.
\label{eq:CS_13}
\end{equation}

% ------------------------------------------ Semi-implicit --------------------------------------------
In order to ease the notation we will use $\Mpsi_j = \Mpsi_j^+ - \Mpsi_j^\circ$.
Hence, the discrete equations \eref{eq:CS_12}-\eref{eq:CS_12_1} read as follows:
\begin{eqnarray}
	\sum\limits_{j \in S_i}\D_{i,j}\vvh_j ^{{n+1}}=0,\label{eq:CS_15}	\\
	\Mpsi_j	\vvh_j^{n+1}-\Mpsi_j	\Fvh_j + \Q_{r(j),j} \etah_{r(j)}^{{n+1}}+ \Q_{\ell(j),j} \etah_{\ell(j)} ^{{ n+1}} =0,
\label{eq:CS_16}
\end{eqnarray}
where $\Fvh_j$ is an appropriate discretization of the nonlinear convective, viscous and source terms that will be presented later. 
Formal substitution of the discrete velocity field given by the momentum equation \eref{eq:CS_16} into the discrete continuity equation \eref{eq:CS_15}, 
see also \cite{CasulliCheng1992,DumbserCasulli}, yields   
\begin{eqnarray}
 \sum\limits_{j\in S_i}\D_{i,j}\Mpsi_j^{-1}\Q_{i,j} \etah_i^{n+1}
 +\sum\limits_{j\in S_i} \D_{i,j}\Mpsi_j ^{-1}\Q_{\p(i,j),j} \etah_{\p(i,j)}^{n+1}=\sum\limits_{j \in S_i} \D_{i,j} \Fvh_j. 
\label{eq:CS_19}
\end{eqnarray}
Eq. \eref{eq:CS_19} above represents a block five-point system for the pressure $\etah_i^{n+1}$.

\subsection{Nonlinear convective and viscous terms}
\label{sec_nlcd}

We now have to choose a proper discretization for the nonlinear convective and viscous terms. 
%As discussed in \cite{2STINS}, the simplest choice is to take $\Fvh_j$ explicitly, so in this case $\sum\limits_{j \in S_i} \D_{i,j} \Fvh_j^{n}$ becomes a known term at time $t^n$ and hence Eq. \eref{eq:CS_19} would represents a block five-points system for the new pressure $\etah_i^{n+1}$ as proposed in \cite{2STINS}. Unfortunately, in problem when the convective-viscous effects cannot be neglected, this will produce only a low order method in time. The problem in this case is that the convective-viscous contribution in the time interval $T^{n+1}$ is based on the old information $T^n$ and does not see the effects of the pressure in the time interval $T^{n+1}$. Furthermore if we take $\Fvh_j$ implicitly, then system \eref{eq:CS_19} becomes nonlinear and it would not be so simple to be solved. In order to overtake this problem we introduce a Picard iterations to update the information about the pressure but without introduce nonlinearity. Hence, for $k=1, N_{pic}$, we rewrite system \eref{eq:CS_19} such as
As discussed in \cite{2STINS} we introduce a simple Picard iteration to update the information about the pressure, but without introducing any nonlinearity into the 
final system for the pressure. Hence, for $k=1, N_{pic}$, we rewrite system \eref{eq:CS_19} as 
\begin{eqnarray}
 \sum\limits_{j\in S_i}\D_{i,j}\Mpsi_j^{-1}\Q_{i,j} \etah_i^{n+1,k+1}
 +\sum\limits_{j\in S_i} \D_{i,j}\Mpsi_j ^{-1}\Q_{\p(i,j),j} \etah_{\p(i,j)}^{n+1,k+1}=\sum\limits_{j \in S_i} \D_{i,j} \Fvh_j^{n+1,k \poh}.
\label{eq:CS_19_2}
\end{eqnarray}
%
%or, by introducing the boundary elements (see e.g. \cite{2STINS}),
%\begin{eqnarray}
%\left[ \sum\limits_{j\in S_i\cap \B(\Omega)}\D_{i,j}^{\partial}\Mpsi_j^{-1}\Q_{i,j}^{\partial} -\sum\limits_{j\in S_i-\B(\Omega)}\D_{i,j}\Mpsi_j^{-1}\Q_{i,j} \right] \etah_i^{n+1,k+1} \nonumber \\
%- \sum\limits_{j\in S_i-\B(\Omega)} \D_{i,j}\Mpsi_j^{-1}\Q_{\p(i,j),j} \etah_{\p(i,j)}^{n+1,k+1}= \nonumber \\
%-\sum\limits_{j \in S_i-\B(\Omega)} \D_{i,j} \Fvh_j^{n+1,k \poh}  +\sum\limits_{j \in S_i\cap\B(\Omega)} \D_{i,j}^{\partial} \Fvh_j^{n+1,k \poh},
%\label{eq:82_2}
%\end{eqnarray}
%Eq. \eref{eq:CS_19} represents a block four-diagonal system for the new pressure $\etah_i^{n+1}$. One can also see how the above system, as written in Eq. \eref{eq:CS_19}, is singular due to exact balance between main and other diagonals. Even if the shape of the equations are formally the same introduced in \cite{2DSIUSW}, this aspect represents an important difference between this method and the one used for the shallow water equartion where the mass matrix $\Mphi$ is sufficient to make the system, at this level, non-singular. When the incompressible Navier-Stokes equations are discretized, the introduction of the boundary elements is crucial to make system \eref{eq:CS_19} non-singular and then allow to solve it.

%Let us now consider the boundary elements.
%where $\D_{i,j}^{\partial}$ and $\Q_{i,j}^{\partial}$ are the natural extension of $\D$ and $\Q$ on tetrahedral dual boundary elements, see e.g. \cite{2STINS}.
The right side of Eq. \eref{eq:CS_19_2} can be computed by using the velocity field at the Picard iteration $k$ and including the viscous effect implicitly, using a fractional step procedure 
detailed later. Once the new pressure field is known, the velocity 
vector field at the new Picard iteration $\vvh^{n+1,k+1}$ can be readily updated from the discrete momentum equation \eref{eq:CS_16}.

To close the problem it remains to specify how to construct the nonlinear convective-diffusion operator $\Fvh_j^{n+1,k\poh}$. At this point one can try to extend the procedure already used 
in \cite{2STINS} to 3D. However, in this case there are some issues that have to be taken into account. In particular, since we are using a modal basis on the staggered dual non-standard 
5-point hexahedral mesh, we cannot use the simple nodal approximation for the nonlinear convective term $\hat{\mathbf{F}}_c=\mathbf{F}_c(\hat{\mathbf{v}})$ that consists in a trivial 
point-wise evaluation of the nonlinear operator $\mathbf{F}_c$. 
%Indeed, observe how naturally $\mathbf{F}(\mathbf{v}_h)$ lives in $V_h(2p)$. We can then locally project $\mathbf{F}(\mathbf{v}_h)$ to a higher order polynomial space and then use the $L_2$ projection on the subspace $V_h(p)$ in order to reconstruct $\hat{\mathbf{F}}$. In this way the coefficients of $\hat{\mathbf{F}}$ become optimal in the sense that they represent, by definition, the best $V_h(p)$ approximation of $\mathbf{F}(\mathbf{v}_h)$. In particular, after some computations, we get
%\begin{eqnarray}
		%\hat{\mathbf{F}}=\mathcal{M}_j^{-1}\TC_j \hat{\mathbf{v}}_j \otimes \hat{\mathbf{v}}_j
%\label{eq:TC_1}
%\end{eqnarray}
%where
%\begin{eqnarray}
		%\mathcal{M}_j=\int\limits_{\QQ_j^\st}{\tilde{\psi}_k^{(j)} \tilde{\psi}_l^{(j)} \dxdt} \qquad \qquad \TC_j=\int\limits_{\QQ_j^\st}{\tilde{\psi}_k^{(j)} \tilde{\psi}_l^{(j)} \tilde{\psi}_m^{(j)} \dxdt}
%\label{eq:TC_2}
%\end{eqnarray}
%From a computational point of view, this will require a lot of effort as well as a lot of memory usage in order to store the tensor $\TC_j$ for all $j \in [1,\Nj]$. 
%We propose here a new procedure for the computation of the nonlinear convective-viscous term.
%Following the same idea of \cite{DumbserCasulli} and inspired by the good properties achieved by the use of staggered grid, we propose a new procedure for the computation of the nonlinear convective-viscous term.
Inspired by the good properties obtained by the use of staggered grids, here we propose a new procedure for the computation of the nonlinear convective and viscous terms. For that purpose,
the velocity field is first interpolated from the dual grid to the main grid. The nonlinear convective terms can then be easily discretized with a standard (space-time) DG scheme on the main grid. 
Then, the staggered mesh is used \textit{again} in order to define the gradient of the velocity on the dual elements, which allows us to produce a very simple and sparse system for the 
discretization of the viscous terms. 

%\subsubsection{A symmetric semi-implicit implementation for the viscous contribution}
%\label{sec_ViscoSymm}
%In this section we want to introduce an alternative treatment for the nonlinear viscous contribution. A fully implicit implementation based on the dual grid leads to a linear system for each velocity component that is a five-point non symmetric system that is, however, well conditioned since it can be written as a $\nu$ perturbation of the identity matrix, see e.g. \cite{2STINS}.  
%Here, we will develop a system for the viscous term that is a four point one and, more important, is symmetric and at least semi-positive definite for low order in time method but is still better conditioned also in the complete case.
%In this section we want to introduce an alternative treatment for the nonlinear convective and viscous contribution based on the high order projection of the velocity field on the main grid. This procedure was necessary in the formulation \cite{DumbserCasulli} since the velocity was splitted in a couple of staggered grids and here can be used as an alternative treatment of the nonlinear convective-viscous contribution. 

An implicit discretization of the viscous terms on the dual grid leads to a linear system for each velocity component that is a seven-point non symmetric block system which is, however, 
well conditioned since it can be written as a $\nu$ perturbation of the identity matrix, see e.g. \cite{2STINS}. 
Here, we will develop a discretization of the viscous terms that leads only to a five-point block system and, more importantly, is symmetric and positive definite for 
$\nu>0$ and $p_\gamma=0$, but is still better conditioned also in the general case $p_\gamma > 0$. 

Given a discrete velocity field $\mathbf{v}_h$ on the dual grid in the time interval $[t^n,t^{n+1}]$, we can project the velocity field from the dual mesh to the main grid (denoted by 
$\bar{\mathbf{v}}$) via standard $L_2$ projection, 
\begin{eqnarray}
%	\overline{\Fvc}_i &=& \Mpsi^{-1}_i \sum\limits_{j \in S_i} \Mpsi_{i,j} \Fvc_j \nonumber \\
	\bar{\mathbf{v}}_i^{n+1} &=& \Mpsi^{-1}_i \sum\limits_{j \in S_i} \Mpsi_{i,j} \hat{\mathbf{v}}_j^{n+1}, \qquad \forall i \in [1,\Ni], 
\label{eq:vis1}
\end{eqnarray}
where $\bar{\mathbf{v}}_i^{n+1}$ denote the degrees of freedom of the velocity on the main grid and  
\begin{eqnarray}
   \Mpsi_{i}= \int\limits_{\TT_{i}^\st} \tilde{\phi}_k^{(i)} \tilde{\phi}_l^{(i)} d\xx dt, \qquad \Mpsi_{i,j}= \int\limits_{\TT_{i,j}^\st}\tilde{\phi}_k^{(i)}\tilde{\psi}_l^{(j)} d\xx dt. 
\label{eq:vis3}
\end{eqnarray}
The projection back onto the dual grid is given by 
\begin{eqnarray}
%	\Fvc_j &=& \Mpsi^{-1}_j\left( \Mpsi_{\ell(j),j}^\top \overline{\Fvc}_{\ell(j)}+\Mpsi_{r(j),j}^\top \overline{\Fvc}_{r(j)} \right) \nonumber \\
	\hat{\mathbf{v}}_j^{n+1} &=& \overline{\Mpsi}_j^{\, -1} \left( \Mpsi_{\ell(j),j}^\top \overline{\mathbf{v}}_{\ell(j)}^{n+1} + \Mpsi_{r(j),j}^\top \overline{\mathbf{v}}_{r(j)}^{n+1} \right). 
\label{eq:vis2}
\end{eqnarray}
with 
\begin{eqnarray}
	\overline{\Mpsi}_j &=& \int\limits_{\QQ_j^\st}{\tilde{\psi}_k^{(j)}\tilde{\psi}_l^{(j)} d \xx dt. }
\label{eq:mj}
\end{eqnarray}

% ------->> Changed from here
%The same projection can also be performed for the discretization operator of the convective and viscous terms. 

We can rewrite the nonlinear convective and viscous part of the momentum equation by introducing the viscous stress tensor $\boldsymbol{\sigma} = -\nu \nabla \mathbf{v}$ as 
auxiliary variable. The convective and viscous subsystem of the momentum equation then reads  
\begin{eqnarray}
		\frac{\partial \mathbf{v} }{\partial t} + \nabla \cdot \mathbf{F}_c + \nabla \cdot \boldsymbol{\sigma} & = & 0,  \nonumber \\
		\boldsymbol{\sigma} & = & -\nu \nabla \mathbf{v}. 
\label{eq:vis9}
\end{eqnarray}
With the averaged velocity $\bar{\mathbf{v}}_i^{n+1}=\tilde{\phi}_l^{(i)} \bar{\mathbf{v}}_{l,i}^{n+1}$ defined on the main grid and the viscous stress tensor 
$\boldsymbol{\sigma}^{n+1}_j = \tilde{\psi}_l^{(j)} \boldsymbol{\sigma}_{l,j}^{n+1}$ defined on the dual grid, we obtain the following
weak formulation of \eqref{eq:vis9}: 
\begin{eqnarray}
    \int \limits_{\TT_i} \tilde{\phi}_k^{(i)}(\xx,t^{n+1}) \bar{\mathbf{v}}_i^{n+1} \, d \xx - 
		\int \limits_{\TT_i} \tilde{\phi}_k^{(i)}(\xx,t^{n})   \bar{\mathbf{v}}_i^{n}   \, d \xx - 
		\int \limits_{\TT_i^\st} \frac{\partial \tilde{\phi}_k^{(i)} }{\partial t} \bar{\mathbf{v}}_i^{n+1} \, d \xx dt + 
		\nonumber \\  
		\int\limits_{\partial \TT_{i}^\st}{\tilde{\phi}_k^{(i)} \mathbf{F}_c^\textnormal{RS} \left(\bar{\mathbf{v}}^-,\bar{\mathbf{v}}^+\right) \cdot \vec{n}_i \, dS dt}  
    - \int\limits_{\TT_{i}^\st}{\nabla \tilde{\phi}_k^{(i)} \cdot \mathbf{F}_c( \bar{\mathbf{v}}_i^{n+1} ) \, d \xx dt}	 			
		+ \sum_{j \in S_i} \left( \int\limits_{\Gamma_{j}^\st}{\tilde{\phi}_k^{(i)} \boldsymbol{\sigma}_j^{n+1} \cdot \vec{n}_{ij} \, dS dt}  
    - \int\limits_{\TT_{i,j}^\st}{\nabla \tilde{\phi}_k^{(i)} \cdot \boldsymbol{\sigma}_j^{n+1} \, d \xx dt} \right) 	
		  =  0,  \nonumber \\
		\int \limits_{\QQ_{j}^\st} \tilde{\psi}_k^{(j)}(\xx,t^{n+1}) \boldsymbol{\sigma}_j^{n+1} \, d \xx  =  - \nu \left( \, 
		 \int\limits_{\TT_{\ell(j),j}^\st}{\tilde{\psi}_k^{(j)} \nabla \bar{\mathbf{v}}_{\ell(j)}^{n+1}  \, d\xx  dt} 
    +\int\limits_{\TT_{r(j),j}^\st}{   \tilde{\psi}_k^{(j)} \nabla \bar{\mathbf{v}}_{r(j)}^{n+1}  \, d\xx  dt}     
 +\int\limits_{\Gamma_j^\st}{\tilde{\psi}_k^{(j)} \left( \bar{\mathbf{v}}_{r(j)}^{n+1} - \bar{\mathbf{v}}_{\ell(j)}^{n+1} \right) \otimes  \nstd \, dS dt}   \right). 
\label{eq:vis10a}
\end{eqnarray}
In a more compact matrix notation, \eref{eq:vis10a} can be written as:
\begin{eqnarray}
	\left( \overline{\Mpsi}_i^{\, +} - \overline{\Mpsi}_i^{\, o} \right) \vbar_i^{n+1}-\overline{\Mpsi}_i^{\, -} \vbar_i^{n}+ \sum\limits_{j \in S_i} \D_{i,j} \boldsymbol{\sigma}_j^{n+1}
	+ \overline{\Upsilon}_i^c = 0, \nonumber \\
	\overline{\Mpsi}_j \boldsymbol{\sigma}_j^{n+1}= -\nu \left( \Q_{\ell(j),j} \vbar_{\ell(j)}^{n+1}+\Q_{r(j),j} \vbar_{r(j)}^{n+1} \right),
\label{eq:vis10}
\end{eqnarray}
where
\begin{eqnarray}
	\overline{\Mpsi}_i^{\, +} &=& \int\limits_{\TT_{i}}{\tilde{\phi}_k^{(i)}(\xx,t(1)) \tilde{\phi}_l^{(i)}(\xx,t(1)) d \xx, }   \label{eq:vis11_1}  \\
	\overline{\Mpsi}_i^{\, -} &=& \int\limits_{\TT_{i}}{\tilde{\phi}_k^{(i)}(\xx,t(0)) \tilde{\phi}_l^{(i)}(\xx,t(1)) d \xx, }  \label{eq:vis11_2} \\
	\overline{\Mpsi}_i^{\, o} &=& \int\limits_{\TT_{i}^\st}{\diff{\tilde{\phi}_k^{(i)}}{t}\tilde{\phi}_l^{(i)} d \xx dt. }  \\
\label{eq:vis11}
\end{eqnarray}
In \eqref{eq:vis10} we have defined the operator $\overline{\Upsilon}_i^c(\bar{\mathbf{v}})$, which is a standard DG discretization of the nonlinear convective terms on the tetrahedral
elements of the \textit{main grid},  
\begin{equation}
 \overline{\Upsilon}_i^c(\bar{\mathbf{v}}) = \int\limits_{\partial \TT_{i}^\st}{\tilde{\phi}_k^{(i)} \mathbf{F}_c^\textnormal{RS} \left(\bar{\mathbf{v}}^-,\bar{\mathbf{v}}^+\right) \cdot \vec{n}_i \, dS dt}  
                        - \int\limits_{\TT_{i}^\st}{\nabla \tilde{\phi}_k^{(i)} \cdot \mathbf{F}_c( \bar{\mathbf{v}} ) \, d \xx dt}, 
\label{eqn.convdg} 
\end{equation} 
with the the boundary extrapolated values $\mathbf{v}^-$ and $\mathbf{v}^+$ from within the cell and from the neighbors, respectively. Here, the approximate Riemann solver 
$\mathbf{F}_c^\textnormal{RS}$ used at the element boundaries is given by the simple Rusanov flux \cite{Rusanov:1961a} 
\begin{equation}
 \mathbf{F}_c^\textnormal{RS} \left(\bar{\mathbf{v}}^-,\bar{\mathbf{v}}^+\right) \cdot \vec{n}_i = 
  \frac{1}{2}\left( \mathbf{F}_c(\bar{\mathbf{v}}^+) + \mathbf{F}_c(\bar{\mathbf{v}}^-) \right) \cdot \vec{n}_i \, - 
	\frac{1}{2} s_{\max} \left( \bar{\mathbf{v}}^+ - \bar{\mathbf{v}}^- \right),
\end{equation} 
where $s_{\max} = 2 \max\left( |\bar{\mathbf{v}}^+|, |\bar{\mathbf{v}}^-| \right)$ is the maximum eigenvalue of the convective operator $\mathbf{F}_c$. 
The final system for the variable $\vbar$ can be found by formal substitution of $\boldsymbol{\sigma}$ given in the second equation of \eqref{eq:vis10} into the first one:
\begin{eqnarray}
\left( \overline{\Mpsi}_i-\nu \sum\limits_{j \in S_i} \D_{i,j} \overline{\Mpsi}_j^{\, -1} \Q_{i,j} \right) \vbar_i^{n+1}-\nu \sum\limits_{j \in S_i} \D_{i,j} \overline{\Mpsi}_j^{\, -1} \Q_{\p(i,j),j} \vbar_{\p(i,j)}^{n+1}= \overline{\Mpsi}_i^{\,-} \vbar_i^n - \overline{\Upsilon}_i^c(\bar{\mathbf{v}}^{n+1}),  
\label{eq:vis12}
\end{eqnarray}
where we use the abbreviation $\overline{\Mpsi}_i=\overline{\Mpsi}_i^{\,+} - \overline{\Mpsi}_i^{\,o}$. 
What we obtain is a discretization of the nonlinear convective and viscous terms on the main grid, where the stress tensor $\boldsymbol{\sigma}$ has been computed on the face-based dual mesh. 
In order to avoid the solution of a nonlinear system due to the nonlinear operator $\overline{\Upsilon}_i^c(\bar{\mathbf{v}}^{n+1})$, we introduce a fractional step scheme combined with an outer
Picard iteration. Using the notation introduced in \cite{2STINS}, we get 
\begin{eqnarray}
	\left( \overline{\Mpsi}_i-\nu \sum\limits_{j \in S_i} \D_{i,j} \overline{\Mpsi}_j^{\, -1} \Q_{i,j} \right) \vbar_i^{n+1,k\poh}-\nu \sum\limits_{j \in S_i} \D_{i,j} \overline{\Mpsi}_j^{\, -1} \Q_{\p(i,j),j} \vbar_{\p(i,j)}^{n+1,k\poh}=\overline{\Mpsi}_i^{\,-} \vbar_i^n- \overline{\Upsilon}_i^c\left(\vbar^{n+1,k}\right). 
\label{eq:Ad2}
\end{eqnarray}

\subsection{Final space-time pressure correction formulation}
\label{sec_2.6}
%The algorithm described above, as described by Eqs. \eref{eq:Ad2}, \eref{eq:82_2}, \eref{eq:CS_16}, does not work when viscous dominated problems are considered. 
%Indeed, remark how the convective-viscous operator $\Upsilon(\mathbf{v}, \nabla \mathbf{v})$ defined in \eref{eq:Ad3} can be splitted in a convective and viscous part such as $\Upsilon(\mathbf{v}, \nabla \mathbf{v})=\Upsilon^c(\mathbf{v})+\Upsilon^v(\nabla \mathbf{v})$ since $\mathbf{\TF}=\mathbf{F}_c(\mathbf{v})-\nu \nabla \mathbf{v}$ and also $s_{max}$ can be splitted in a similar way. When the convective effects can be neglected (i.e. $\Upsilon^c(\mathbf{v})=0$), 
As already discussed in \cite{2STINS}, the computation of the nonlinear convective and viscous terms presented in Eq. \eref{eq:Ad2} does not depend explicitly on the pressure of the 
previous Picard iteration, and hence it does not see the effect of the pressure in the time interval $T^{n+1}$, which is, however, needed to get a high order accurate scheme also in 
time. 
In order to overcome the problem, we introduce directly into Eq. \eref{eq:Ad2} the contribution of the pressure in the time interval $T^{n+1}$, but at the previous Picard iteration. 
Then, we update the velocity with the pressure correction $\etah_i^{n+1,k+1}-\etah_i^{n+1,k}$. The final equations \eref{eq:Ad2}, \eref{eq:CS_16} and \eref{eq:CS_19} to be solved for 
each Picard iteration $k$ of our staggered semi-implicit space-time DG method therefore read:
\begin{equation}
 \bar{\mathbf{v}}_i^{n+1,k} = \Mpsi^{-1}_i \sum\limits_{j \in S_i} \Mpsi_{i,j} \hat{\mathbf{v}}_j^{n+1,k},
\label{eqn.average} 
\end{equation} 
\begin{equation}
 \boldsymbol{\Lambda}_i(\etah^{n+1,k}) = \Mpsi^{-1}_i \sum\limits_{j \in S_i} \Mpsi_{i,j} \left( \Mpsi_j^{-1}
 \left( \Q_{r(j),j} \etah_{r(j)}^{{n+1,k}} + \Q_{\ell(j),j} \etah_{\ell(j)}^{{n+1,k}} \right) \right),
\label{eqn.prevpress} 
\end{equation} 

\begin{eqnarray}
	\left( \overline{\Mpsi}_i-\nu \sum\limits_{j \in S_i} \D_{i,j} \overline{\Mpsi}_j^{\, -1} \Q_{i,j} \right) \vbar_i^{n+1,k\poh}-\nu \sum\limits_{j \in S_i} \D_{i,j} \overline{\Mpsi}_j^{\, -1} \Q_{\p(i,j),j} \vbar_{\p(i,j)}^{n+1,k\poh}=\overline{\Mpsi}_i^{\, -} \vbar_i^n- \overline{\Upsilon}_i^c\left(\vbar^{n+1,k}\right) - \overline{\Mpsi}_i \boldsymbol{\Lambda}_i(\etah^{n+1,k}),
\label{eq:A1_mod}
\end{eqnarray}
\begin{equation}
 \Fvh_j^{n+1,k \poh} = \overline{\Mpsi}^{\, -1}_j\left( \Mpsi_{\ell(j),j}^\top \overline{\mathbf{v}}_{\ell(j)}^{n+1,k\poh} + \Mpsi_{r(j),j}^\top \overline{\mathbf{v}}_{r(j)}^{n+1,k\poh} \right),
\label{eqn.goback} 
\end{equation} 
\begin{equation}
 \sum\limits_{j\in S_i}\D_{i,j}\Mpsi_j^{-1}\Q_{i,j} \left( \etah_i^{n+1,k+1} - \etah_i^{n+1,k} \right) 
 +\sum\limits_{j\in S_i} \D_{i,j}\Mpsi_j ^{-1}\Q_{\p(i,j),j} \left( \etah_{\p(i,j)}^{n+1,k+1} - \etah_{\p(i,j)}^{n+1,k} \right) 	 
 =\sum\limits_{j \in S_i} \D_{i,j} \Fvh_j^{n+1,k \poh},
\label{eq:A2_mod}
\end{equation}
\begin{equation}
		\vvh_j^{n+1,k+1} = \Fvh_j^{n+1,k \poh} - \Mpsi_j^{-1}\left( \Q_{r(j),j} \left(\etah_{r(j)}^{{n+1,k+1}}- \etah_{r(j)}^{{n+1,k}} \right)+ \Q_{\ell(j),j} \left(\etah_{\ell(j)} ^{{ n+1, k+1}}-\etah_{\ell(j)} ^{{ n+1, k}}\right)\right), 
\label{eq:A3_mod}
\end{equation}
where $\boldsymbol{\Lambda}_i(\etah^{n+1,k})$ represents the same additional contribution subtracted in \eref{eq:A2_mod} that lives on the dual mesh, passed through the mean maps from 
the dual to the main grid. 
%As initial guess for the pressure we take $\etah^{{n+1,0}}=0$ and so obtain the classical first order method such as in \cite{2STINS} for the first picard iteration. Eventually, we can take the lagrange interpolation of $\etah^{n}$ on $T^{n+1}$.
As initial guess for the pressure we simply take $\etah^{{n+1,0}}=0$, while for the velocity field we simply take the velocity field at the previous time step. 
As an alternative, one could also take an extrapolation of pressure and velocity from the previous time interval. A summary of the algorithm reads: 

\begin{enumerate}
 \setcounter{enumi}{-1}
 \item Choose an initial guess for the pressure and the velocity.  
 \item average the velocity field from the dual grid to the main grid using \eqref{eqn.average} and compute the contribution of the pressure gradient of the previous Picard iteration on the main
       grid using \eqref{eqn.prevpress}; 
 \item with the averaged velocity on the main grid, compute the nonlinear convective terms via \eqref{eqn.convdg};  
 \item solve the linear systems for the viscous terms \eqref{eq:A1_mod} on the main grid; 
 \item compute the term $\Fvh_j^{n+1,k \poh}$ on the dual grid via \eqref{eqn.goback}; 
 \item solve the linear system for the pressure correction \eqref{eq:A2_mod} on the main grid; 
 \item update the velocity field according to \eqref{eq:A3_mod} using the previously obtained pressure correction. 
\end{enumerate}
Steps 1-6 are repeated for a total number of Picard iterations of $N_{pic}=p_\gamma + 1$, since a standard Picard process applied to an ODE allows to gain one 
order of accuracy per iteration. 

%The final algorithm is then summarized as:
%\begin{enumerate}
	%\item Initialize $\mathbf{v}_h^{n+1,0}$ and $\etah^{{n+1,0}}$ according to the information in $T^n$
	%\item Loop over $k=0\ldots N_{pic}$:
		%\begin{enumerate}
			%\item Compute $\mathbf{v}_h^{n+1,k  \poh}$ using \eref{eq:A1_mod} and set $\Fvh_j^{n+1,k \poh}=\mathbf{v}_h^{n+1,k \poh}$
			%\item Compute $\etah^{{n+1,k+1}}$ using \eref{eq:A2_mod}
			%\item Update $\vvh_j^{n+1,k+1}$ explicitly from \eref{eq:A3_mod}
		%\end{enumerate}	
	%\item Set $\vvh_j^{n+1}=\vvh_j^{n+1,k+1}$	and  $\etah^{{n+1}}=\etah^{{n+1,k+1}}$
%\end{enumerate}

%Note how, if we use the natural extension of the algorithm presented in \cite{2STINS} to compute the nonlinear convective-viscous term, then we recover the stability property as a corollary of Theorem 1 presented in \cite{2STINS}. Otherwise a stability proof has to take into account the high $L_2$-projections \eref{eq:vis1} and \eref{eq:vis2} that in general do not act as the identity, for instance 
%$$\xi_i \neq \Mpsi^{-1}_i \sum\limits_{j \in S_i} \Mpsi_{i,j}\Mpsi^{-1}_j\left[\Mpsi_{\ell(j),j}^\top \xi_{\ell(j)}+\Mpsi_{r(j),j}^\top \xi_{r(j)} \right] $$
%for $\xi$ defined on the main grid and vice versa. We can still recover a stability result in the case of a projection that acts as the identity following the idea of a finite volume scheme.

\subsection{Remarks on the special case of piecewise constant polynomials in time ($p_\gamma=0$)}
\label{sec.CNmethod}

The method presented in the previous sections can be seen, for $p_\gamma=0$, as the extension of \cite{2SINS} to three space dimensions. 
This particular case is, in general, only first order accurate in time but high order accurate in space. %and becomes particular suitable when we consider solutions with complex small structures in space and slow or steady evolution in time. 
In this case, we can recover several good properties for the main system for the pressure and for the linear systems that need to be solved for the implicit discretization 
of the viscous terms. 

\subsubsection{Pressure system} 
For $p_\gamma=0$ we have $\Mpsi_j^\circ=0$ then $\Mpsi_j = \Mpsi_j^+ = \Mpsi_j^-$ is symmetric for all $j \in 1 \ldots \Nj$. Consequently, the system 
\eref{eq:CS_15}-\eref{eq:CS_16} formally becomes the same method as in \cite{2SINS}.  
% where we use only an additional equation and a different number of basis functions $N_\psi$ and $N_\phi$. 
The following results can therefore be readily obtained as corollaries of the theorems given in \cite{2SINS} regarding the system matrix $\mathcal{A}$ of the main system for 
the pressure \eqref{eq:CS_19}:

\begin{Corollary}[Symmetry]
\label{cor1}
	Let $p_\gamma=0$, the system matrix $\mathcal{A}$ of the main system for the pressure is symmetric.
\end{Corollary}

\begin{Corollary}[Positive semi-definiteness]
\label{cor2}
	Let $p_\gamma=0$, the system matrix $\mathcal{A}$ of the main system for the pressure is in general positive semi-definite.
\end{Corollary}

This means that in this particular case we can use faster iterative linear solvers, like the conjugate gradient (CG) method \cite{cgmethod} to solve the main system for the pressure \eqref{eq:CS_19}. 
This advantage makes the case $p_\gamma=0$ particularly suitable for steady or almost steady problems. In order to recover some precision in time we can extend the algorithm by introducing a 
semi-implicit discretization, as suggested in \cite{2SINS}. In ths case, system \eref{eq:CS_15}-\eref{eq:CS_16} has to be discretized as 
\begin{eqnarray}
	\sum\limits_{j \in S_i}\D_{i,j}\vvh_j ^{{n+1}}=0,\label{eq:CS_15_theta}	\\
	\Mpsi_j	\vvh_j^{n+1}-\Mpsi_j	\Fvh_j^n + \Delta t \Q_{r(j),j} \etah_{r(j)}^{{n+\theta}}+ \Delta t \Q_{\ell(j),j} \etah_{\ell(j)} ^{{ n+\theta}} =0,
\label{eq:CS_16_theta}
\end{eqnarray}
where $\etah^{{ n+\theta}}=\theta \etah^{{ n+1}}+(1-\theta)\etah^{{ n}}$ and $\theta$ is an implicitness factor to be taken in the range $\theta \in [\frac{1}{2},1]$, see e.g. \cite{CasulliCattani}. 
For $\theta=\frac{1}{2}$, the Crank-Nicolson method is recovered. 
%The resulting main system for the pressure then becomes:
%\begin{eqnarray}
%\theta \Delta t \left[- \sum\limits_{j\in S_i\cap \B(\Omega)}\D_{i,j}^{\partial}\Mpsi_j^{-1}\Q_{i,j}^{\partial} -\sum\limits_{j\in S_i-\B(\Omega)}\D_{i,j}\Mpsi_j^{-1}\Q_{i,j} \right] \etah_i^{n+1} %\nonumber \\
%-\theta \Delta t \sum\limits_{j\in S_i-\B(\Omega)} \D_{i,j}\Mpsi_j^{-1}\Q_{\p(i,j),j} \etah_{\p(i,j)}^{n+1}=\tilde{\mathbf{b}}_i^n,
%\label{eq:82_theta}
%\end{eqnarray}
%where now the vector of known terms is 
%\begin{eqnarray}
	%\tilde{\mathbf{b}}_i^n & = &- \sum\limits_{j \in S_i-\B(\Omega)} \D_{i,j} \Fvh_j^n
  %+ \sum\limits_{j \in S_i\cap\B(\Omega)} \D_{i,j}^{\partial} \Fvh_j^n \nonumber \\
%&& +(1-\theta) \Delta t \sum\limits_{j\in S_i\cap \B(\Omega)}\D_{i,j}^{\partial}\Mpsi_j^{-1}\Q_{i,j}^{\partial}\etah_{i}^{n}  %\nonumber \\
%+(1-\theta) \Delta t \sum\limits_{ j \in S_i-\B(\Omega)} \D_{i,j} \Mpsi_j^{-1} \left( \Q_{i,j} \etah_{i}^{n}+ \Q_{\p(i,j),j} \etah_{\p(i,j)}^{n} \right). \nonumber \\
%\label{eq:83_theta}
%\end{eqnarray}
In this way we gain some extra precision in time without affecting the computational effort and using the same advantages given by Corollary $\ref{cor1}$ and $\ref{cor2}$ that can 
be easily extended for this case. 

\subsubsection{Viscous system} 
In the special case of piecewise constant polynomials in time ($p_\gamma=0$), we get $\overline{\Mpsi}_i=\Mpsi_i$ and $\overline{\Mpsi}_j=\Mpsi_j$, so that the following results 
about the viscous system \eref{eq:Ad2} can be derived: 
\begin{Corollary}[Symmetry]
\label{cor3}
	If $p_\gamma =0$ then the system \eref{eq:Ad2} is symmetric.
\end{Corollary}
\begin{proof}
	We can write the system matrix of system \eref{eq:Ad2} as $(M+\nu \mathcal{A})$, where $M$ is a block diagonal matrix with $\{\Mpsi_i\}_{i=1\ldots \Ni}$ on the diagonal and $\mathcal{A}$ 
	is the matrix of the pressure system \eref{eq:CS_19}. Thanks to the results obtained in Corollary $\ref{cor1}$, $\mathcal{A}$ is symmetric and also $M$ is symmetric, 
	since $\Mpsi_i=\Mpsi_i^\top$, see \eref{eq:vis11_1}.
\end{proof}
\begin{Corollary}[Positive definiteness]
	If $p_\gamma =0$ then the system \eref{eq:Ad2} is positive definite.
\end{Corollary}
\begin{proof}
	As used in Corollary \ref{cor3}, we can write the system such as $M+\nu \mathcal{A}$ and we know, thanks to Corollary \ref{cor2}, that $\mathcal{A}$ is in general positive semi-definite. 
	A simple computation leads to 
	\begin{eqnarray}
	x(M+\nu \mathcal{A})x^\top=x M x^\top+ \nu x \mathcal{A} x^\top>0
	\label{eq:vis8}
	\end{eqnarray}
	since $\nu x \mathcal{A} x^\top \geq 0$ and $x M x^\top>0$ we have that the complete system is also positive definite.
\end{proof}

In the general case of $p_\gamma>0$ it is not true that we recover the pressure system, since $\overline{\Mpsi}_\cdot \neq \Mpsi_\cdot$. In this case, we can observe how the non symmetric 
contribution affects only $\overline{\Mpsi}_i$. This allows us to write the previous system as $T+\nu H$ where $T$ is a block diagonal non symmetric matrix and $H$ is symmetric and positive
semi-definite. %This means that the new system for the convective-viscous contribution is a weak non symmetric one in the sense that the non symmetry affect only the block diagonal.

\subsection{Extension to curved elements}
The method described in the previous sections can readily be generalized by introducing also curved elements inside the computational domain following an iso-parametric approach. 
This generalization will affect only the pre-processing step. The extension is quite similar to the one introduced in \cite{2STINS,2DSIUSW} for the two dimensional case, but there 
are some differences due to the three dimensionality of the problem. 

\begin{figure}[ht]
    \begin{center}
    \includegraphics[width=0.4\textwidth]{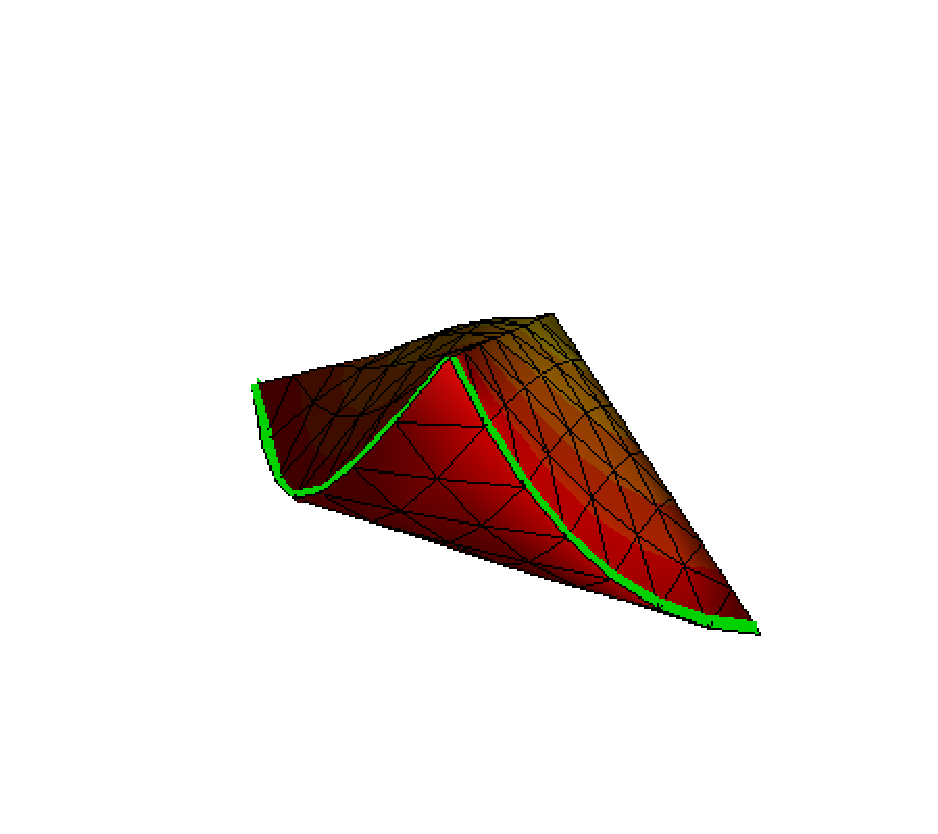}
		\includegraphics[width=0.3\textwidth]{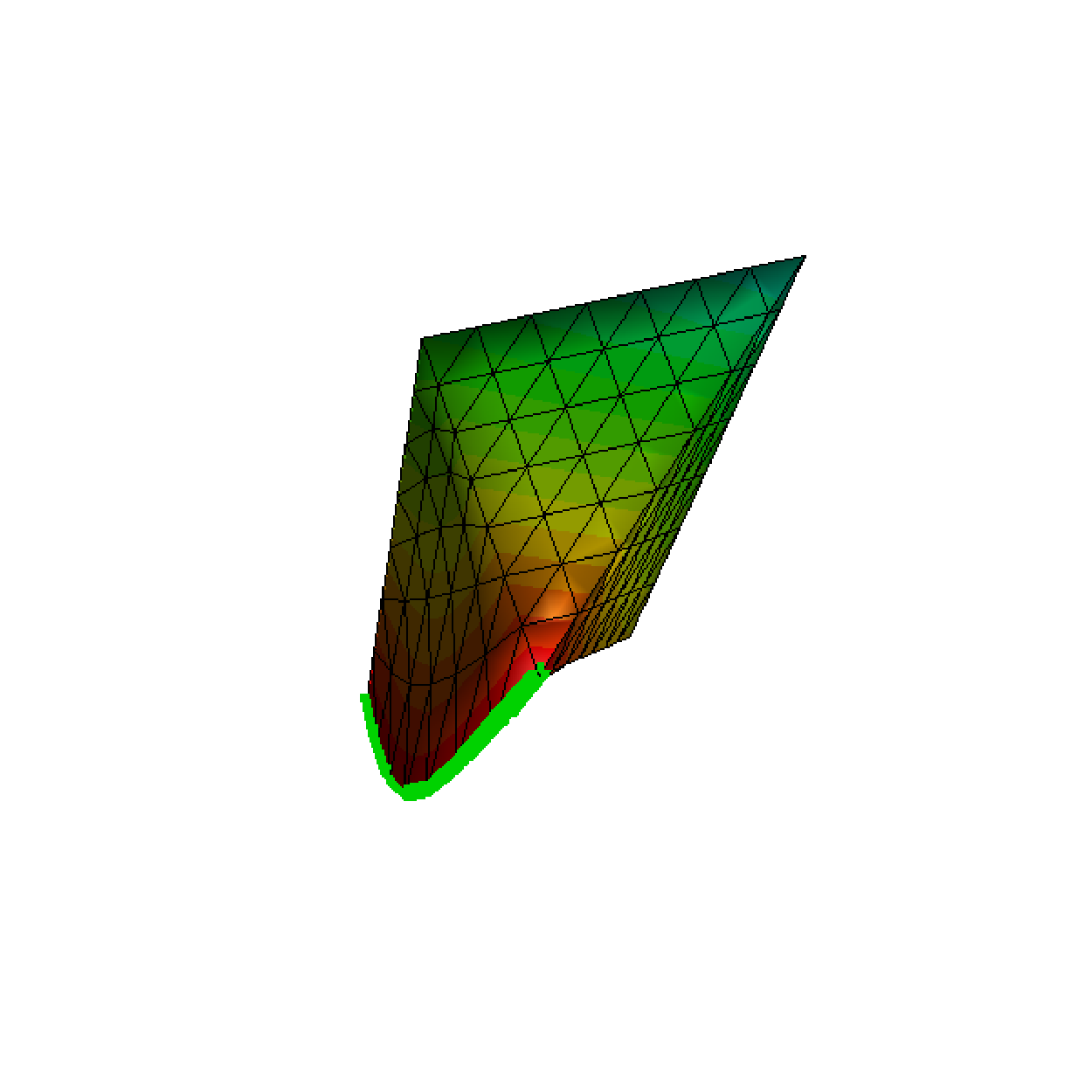}
    \caption{Isoparametric dual element example. On the left, a dual element with a 2D face on the curved boundary; on the right, an internal dual element, but with a 1D edge on the 
		curved boundary.}
    \label{fig.ISO.1}
		\end{center}
\end{figure}  

First of all, in the two dimensional case one could eventually consider as curved only the primary elements that touch a curved boundary, as well as the associated dual elements 
such that $j \in \B(\Omega)$. In the $3D$ case we have to curve also those internal elements which touch the boundary with an edge, see for example Fig.\ref{fig.ISO.1}. Each 
tetrahedral main element is then characterized by $N_\phi$ nodes $\{(X,Y,Z)_k^\cdot\}_{k=1,N_\phi}$, while the dual hexahedral elements are split into a left and a right tetrahedron, 
i.e. $\QQ_j^{iso}=\TT_{\ell(j),j}^{iso} \cup \TT_{r(j),j}^{iso}$ and the points that lie on $\Gamma_j^{iso}$ are physically joined. In this way we have a full characterization 
of the left and the right sub-tetrahedron of the dual hexahedral element, needed to compute properly the integral contributions in the algorithm. 

%In case of non-convex domains it is not possible to project only the boundary DoF on the physical boundary since it is very easy to generate invalid elements. In addition it is still not sufficient, in general, to modify in an elastic way the internal DoF of boundary and near-boundary elements since the curvature can completely invalidate some boundary elements. So we have to move the entire grid. 
In order to compute the position of the grid points in the presence of curved boundaries, we start from an initial tetrahedrization with piecewise linear faces, as given by a standard
mesh generator. Then, we produce a fine sub-tetrahedrization that involves all the degrees of freedom inside the domain $\Omega$ and we solve a simple Laplace equation for the displacement 
using a classical $P1$ continuous finite element method, imposing the projection onto the curved physical boundaries as boundary conditions for the Laplace equation. This procedure 
produces a regular distribution of nodes inside the computational domain in the presence of curved boundaries. 
As shown in \cite{2STINS}, the possibility to curve the grid is crucial when we try to represent complex domains with a very coarse grid. In any case, we emphasize that this generalization does not affect the computational cost during run-time, since it affects only the construction of the main matrices that can be done in a preprocessing step.

\section{Numerical test problems}
\label{sec_numres}
\subsection{Three-dimensional lid driven cavity}
In this section we present some results regarding the three-dimensional lid-driven cavity problem. 
In the literature there are a lot of well known results and reference solutions for the two-dimensional as well as for the fully three-dimensional case, see 
\cite{Ghia1982,Erturk2005,Albensoeder2005,Hwar1987,Aidun1991}.
%For the three-dimensional case there are both, numerical and experimental results. We will use here three different domains with different length-depth ratio in order to approach the experimental and numerical reference data. 
We take a classical cubic cavity $\Omega=[-0.5, 0.5]^3$ and we discretize it with a very coarse tetrahedral mesh with characteristic mesh size $h=0.2$. We set as initial conditions $p=1$; $u=v=w=0$. As boundary condition we impose velocity $(u,v,w)=(1,0,0)$ at $y=0.5$ while no-slip boundary conditions are used on the other boundaries. 
Since we are interested in steady state solutions, we take for the current test $p=4$, $p_\gamma=0$, $\theta=1$, and several different values for the kinematic viscosity in order to 
obtain different Reynolds numbers. 

\begin{figure}[ht]
    \begin{center}
    \includegraphics[width=0.48\textwidth]{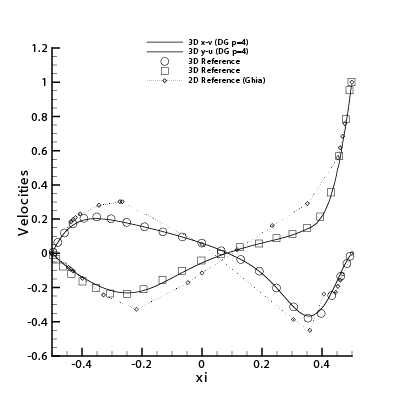}
    \includegraphics[width=0.48\textwidth]{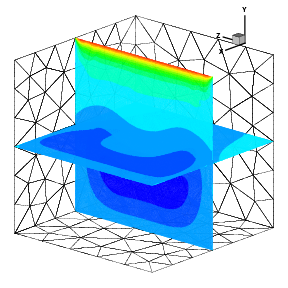} \\
    \includegraphics[width=0.328\textwidth]{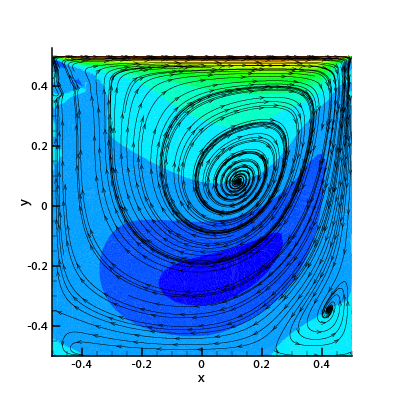}
    \includegraphics[width=0.328\textwidth]{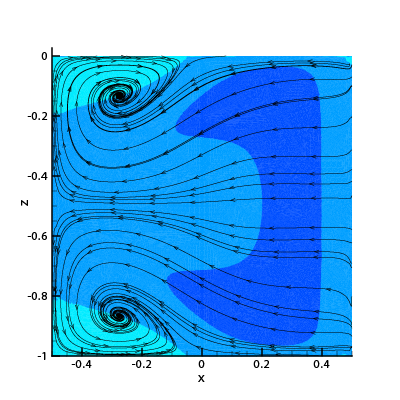}
		\includegraphics[width=0.328\textwidth]{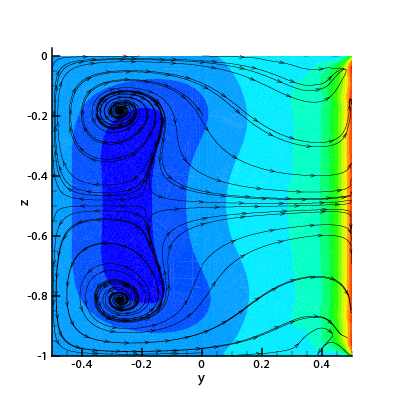}
    \caption{3D lid-driven cavity. From top left to bottom right: Comparison between our numerical results, the one obtained by Albensoeder et al in \cite{Albensoeder2005}, and the two dimensional data from Ghia et al \cite{Ghia1982} at $Re=400$; three-dimensional plot of the two secondary slices and grid spacing; streamlines and magnitude of $u$ on slices $x-y$, $x-z$ and $y-z$.}
    \label{fig.C.1}
		\end{center}
\end{figure}

\begin{figure}[ht]
    \begin{center}
    \includegraphics[width=0.48\textwidth]{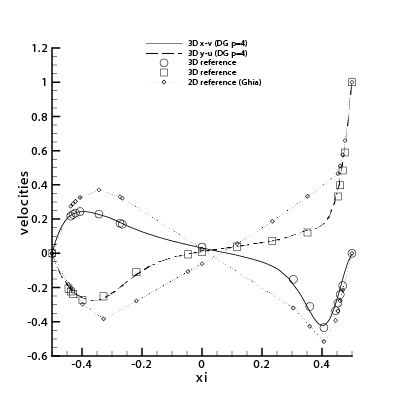}
    \includegraphics[width=0.48\textwidth]{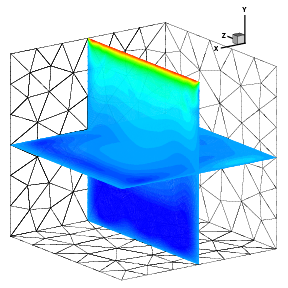} \\
    \includegraphics[width=0.328\textwidth]{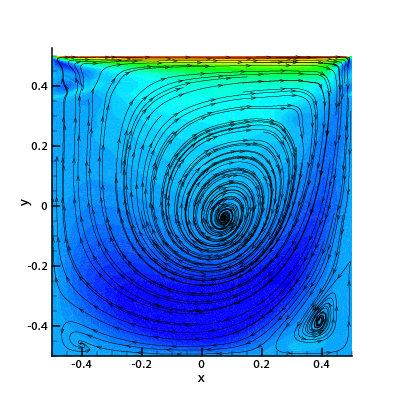}
    \includegraphics[width=0.328\textwidth]{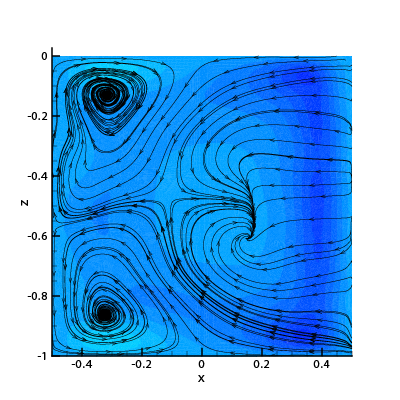}
		\includegraphics[width=0.328\textwidth]{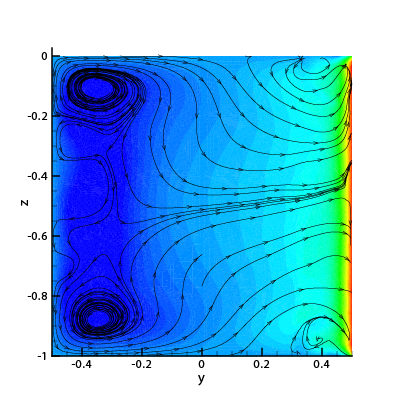}
    \caption{3D lid-driven cavity. From top left to bottom right: Comparison between our numerical results, the one obtained by Albensoeder et al in \cite{Albensoeder2005}, and the two dimensional data from Ghia et al \cite{Ghia1982} at $Re=1000$; three-dimensional plot of the two secondary slices and grid spacing; streamlines and magnitude of $u$ on slices $x-y$, $x-z$ and $y-z$.}
    \label{fig.C.2}
		\end{center}
\end{figure}

In Figure \ref{fig.C.1} the results are shown at a final time of $t_{end}=30$ for $Re=400$. In Figure \ref{fig.C.2} the same plots are given for $t_{end}=40$ and $Re=1000$. In the top left panel
of each plot we report our numerical results and compare them against the reference solution obtained in \cite{Albensoeder2005} for the fully three-dimensional case and the data given 
by Ghia et al. \cite{Ghia1982} for the two dimensional cavity at the same Reynolds number. We note a very good agreement with the 3D reference solution, despite the use of an extremely coarse mesh. 
The data show that the presence of the third space dimension significantly modifies the velocity profiles compared to the 2D case. 
Furthermore, several Taylor-G{\"o}rtler like vortices appear in the secondary planes in a very similar way as observed in other numerical and experimental investigations of 
this problem, see e.g. \cite{Hwar1987,Aidun1991}.

\subsection{Convergence test}
In this test we will investigate the Arnold-Beltrami-Childress flow that was originally introduced by Arnold in \cite{Arnold65} and Childress in \cite{Childress70} as an interesting class of Beltrami flows and successively studied in a series of papers, see e.g. \cite{Dombre86,Podvigina99,Podvigina2006,Ershkov15}. In particular we consider:
\begin{eqnarray}
	u(x,y,z,t) &=& \left[\sin(z)+ \cos(y) \right] e^{-\nu t}, \nonumber \\
	v(x,y,z,t) &=& \left[\sin(x)+ \cos(z) \right] e^{-\nu t}, \nonumber \\
	w(x,y,z,t) &=& \left[\sin(y)+ \cos(x) \right] e^{-\nu t}, \nonumber \\
	p(x,y,z,t) &=& -\left[\cos(x)\sin(y)+ \sin(x)\cos(z)+ \sin(z)\cos(y) \right] e^{-2\nu t}+c
\label{eq:CT_1}
\end{eqnarray}
where $c \in \R$. One can check that this is an exact solution for the complete three dimensional incompressible Navier-Stokes equations in a periodic domain, so this smooth configuration is suitable 
for numerical convergence tests. In particular if $\nu=0$ we can check the accuracy of the spatial part of the algorithm, i.e. $p_\gamma=0$, since the solution is a steady one. We take as computational  domain $\Omega=[-\pi,\pi]^3$ and we extend it using periodic boundary conditions everywhere. We use increasing values of the polynomial degree $p$ and use a sequence of successively refined meshes,  starting from a regular initial mesh. Simulations are performed up to $t_{end}=0.1$. The time step $\Delta t$ is chosen according to the CFL time restriction for explicit DG schemes based on the 
magnitude of the flow velocity. Since we have periodic boundary conditions everywhere, we have a set of solutions for the pressure given by \eref{eq:CT_1} up to a constant. In order to verify 
that also the pressure field is correct, we choose $c$ in \eref{eq:CT_1} a posteriori according to the mean value of the resulting numerical pressure. 

\begin{table}[!htb]
 \caption{Numerical convergence results for the steady 3D ABC flow ($\nu=0$).}
 \begin{center}
 \begin{tabular}{ccccccc}
     \hline
 $p$ & $p_\gamma$ & $\Ni$ & $\epsilon(p)$& $\epsilon(\mathbf{v})$ & $\sigma(p)$& $\sigma(\mathbf{v})$  \\
 \hline \hline
1  &0    &   7986    &    7.4349E-01  &    3.7768E-01  & - & -   \\
1  &0    &   10368    &    6.2638E-01  &    3.1662E-01  & 2.0  & 2.0  \\
1  &0    &   13182    &    5.3318E-01  &    2.7046E-01  & 2.0  & 2.0 \\
1  &0    &   16464    &    4.6155E-01  &    2.3309E-01  & 2.0  & 2.0  \\
 \hline \hline
2  &0    &   7986    &    8.6472E-02  &    5.0920E-02  & 3.0  & 2.4    \\
2  &0    &   10368    &    6.7178E-02  &    4.1417E-02  & 2.9  & 2.4  \\
2  &0    &   13182    &    5.2651E-02  &    3.4271E-02  & 3.0  & 2.4 \\
2  &0    &   16464    &    4.2520E-02  &    2.8499E-02  & 2.9  & 2.5  \\
 \hline \hline
3  &0    &   7986    &    6.6133E-03  &    3.5899E-03  & 3.9  & 3.4  \\
3  &0    &   10368    &    4.7069E-03  &    2.6619E-03  & 3.9  & 3.4 \\
3  &0    &   13182    &    3.4219E-03  &    2.0294E-03  & 4.0  & 3.4   \\
3  &0    &   16464    &    2.5604E-03  &    1.5727E-03  & 3.9  & 3.4  \\
 \hline \hline
4  &0    &   6000    &    8.4806E-04  &    6.7156E-04  & 4.9  & 4.1   \\
4  &0    &   7986    &    5.3156E-04  &    4.5361E-04  & 4.9  & 4.1  \\
4  &0    &   10368    &    3.4667E-04  &    3.1585E-04  & 4.9  & 4.2   \\
4  &0    &   13182    &    2.3307E-04  &    2.2733E-04  & 5.0  & 4.1   \\
 \hline \hline
5  &0    &   4374    &    1.5777E-04  &    1.6300E-04  & 5.9  & 5.1   \\
5  &0    &   6000    &    8.4744E-05  &    9.4463E-05  & 5.9  & 5.2   \\
5  &0    &   7986    &    4.8228E-05  &    5.7433E-05  & 5.9  & 5.2  \\
5  &0    &   10368    &   2.8868E-05  &    3.6318E-05  & 5.9  & 5.2  \\
     \hline
     \end{tabular}
 \end{center}
 \label{tab:CT_1}
 \end{table}

The resulting vorticity, pressure and streamlines are plotted in Figure \ref{fig.CT.1}, while in Table \ref{tab:CT_1} the resulting $L_2$ error norms are reported for the steady case 
$\nu=0$. We observe how the optimal order of convergence is obtained for this steady problem for the pressure, while a suboptimal order of convergence can be observed for the velocity field.  

In the second test case we turn on the viscosity in order to make the problem unsteady. For this kind of problem we use the space-time DG implementation of the algorithm and we set the 
number of Picard iterations to $N_{pic}=p_\gamma+1$. Unfortunately, as soon as we use a high order polynomial in time, the resulting main system looses the symmetry property and hence 
we have to use a slower linear solver, such as the GMRES method. Since the viscosity contribution is discretized implicitly, we can take very large values for the kinematic viscosity 
and maintain the same CFL time restriction for the simulation. The chosen viscosity for this test is $\nu=1$ and we test the method for $p=p_\gamma=1\ldots 4$ on a sequence of 
successively refined grids. The resulting convergence rates, as well as the $L_2$ error norms, are shown in Table \ref{tab:CT_2}. In this case an order of $p+\frac{1}{2}$ is achieved 
for the pressure, while order $p+1$ can be observed for the velocity. 

 \begin{table}[!htb]
 \caption{Numerical convergence results for the unsteady ABC flow ($\nu=1$).}
 \begin{center}
 \begin{tabular}{ccccccc}
     \hline
 $p$ & $p_\gamma$ & $\Ni$ & $\epsilon(p)$& $\epsilon(\mathbf{v})$ & $\sigma(p)$& $\sigma(\mathbf{v})$  \\
 \hline \hline
1  &1    &   10368    &    1.1713E+00  &    2.4695E-01  & 1.6  & 2.0  \\
1  &1    &   13182    &    1.0388E+00  &    2.1017E-01  & 1.5  & 2.0    \\
1  &1    &   16464     &   9.2718E-01  &    1.8075E-01  & 1.5  & 2.0  \\
1  &1    &   20250    &    8.3860E-01  &    1.5730E-01  & 1.5  & 2.0  \\
		\hline \hline
2  &2    &   10368    &    1.7339E-01  &    1.4475E-02  & 2.8  & 3.1  \\
2  &2    &   13182    &    1.4060E-01  &    1.1291E-02  & 2.6  & 3.1   \\
2  &2    &   16464     &    1.1470E-01  &    8.9676E-03  & 2.8  & 3.1  \\
2  &2    &   20250    &    9.5780E-02  &    7.2516E-03 & 2.6  & 3.1  \\
 \hline \hline
3  &3    &   6000    &    1.6219E-02  &    1.5469E-03  & 3.8  & 4.1    \\
3  &3    &   7986    &    1.1454E-02  &    1.0494E-03  & 3.7  & 4.1   \\
3  &3    &   10368    &    8.2191E-03  &    7.3591E-04  & 3.8  & 4.1    \\
3  &3    &   13182    &    6.1399E-03  &    5.3142E-04  & 3.6  & 4.1    \\
 \hline \hline
4  &4    &    750    &    4.5578E-02  &    3.2574E-03  & 4.7  & 4.8  \\
4  &4    &    1296    &    1.9664E-02  &    1.2957E-03  & 4.6  & 5.1 \\
4  &4    &    2058    &    9.3757E-03  &    5.9049E-04  & 4.8  & 5.1   \\
4  &4    &    3072    &    5.0553E-03  &    2.9738E-04  & 4.6  & 5.1  \\
     \hline
     \end{tabular}
 \end{center}
 \label{tab:CT_2}
 \end{table}

\begin{figure}[ht]
    \begin{center}
    \includegraphics[width=0.45\textwidth]{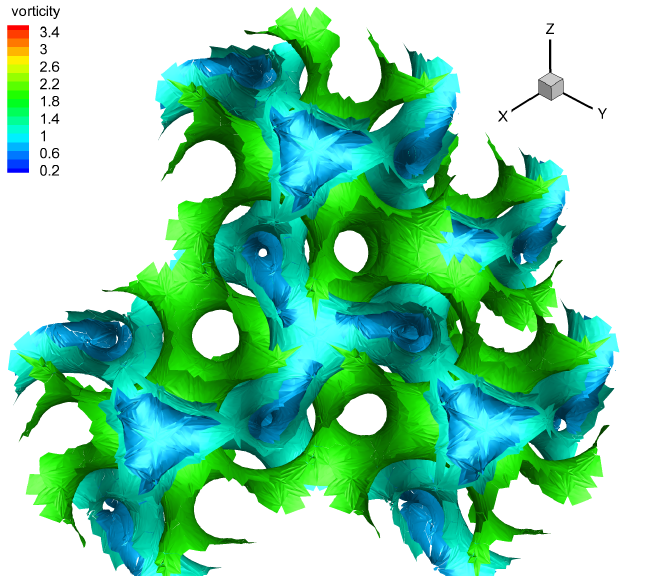}
		\includegraphics[width=0.45\textwidth]{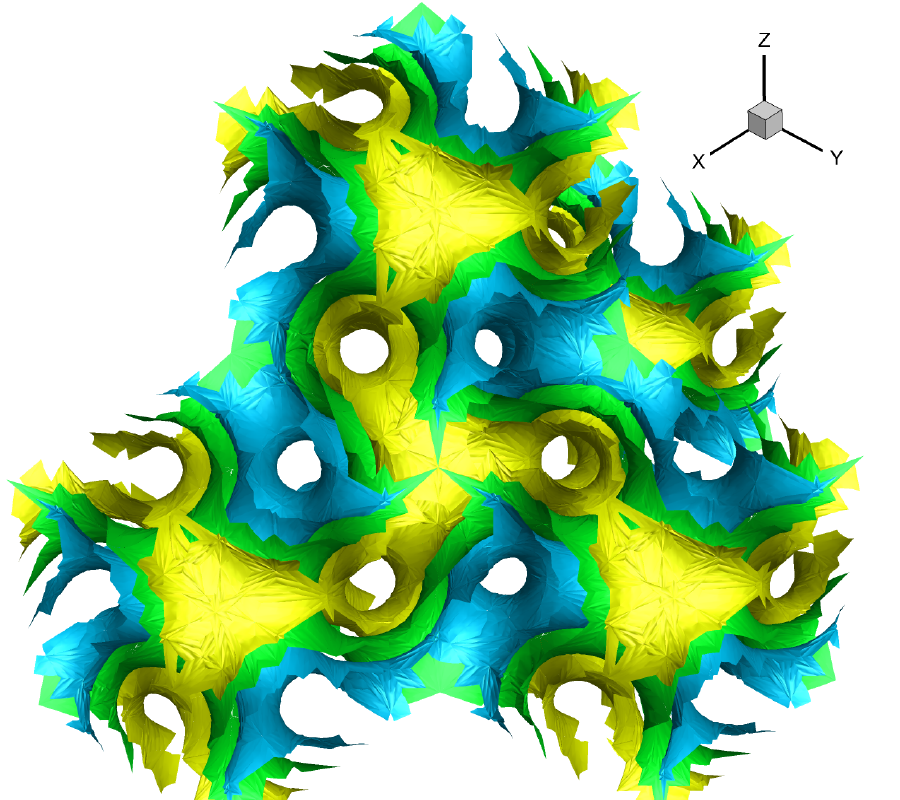}
		\includegraphics[width=0.45\textwidth]{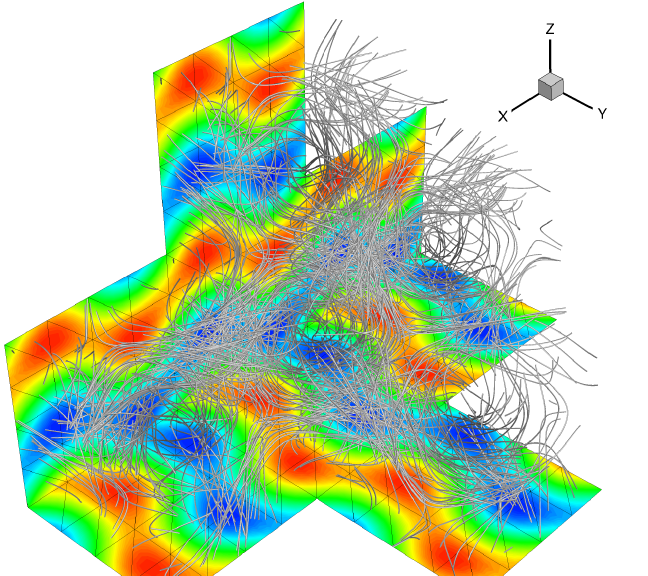}
    \caption{From top left to bottom: Vorticity isosurfaces $[0.8,1.2,2.0]$; pressure isosurfaces $p=[-0.8,0.0,0.8]$ and streamlines in order to show the three-dimensionality of the 
		ABC flow problem.}
    \label{fig.CT.1}
		\end{center}
\end{figure}

\subsection{Taylor-Green Vortex}
\label{sec_3dtgv}
In this section we investigate another typical benchmark problem, namely the classical 3D Taylor-Green vortex. In this test case a very simple initial analytical solution degenerates quickly to a  turbulent flow with very complex flow structures. We take the initial condition as given in \cite{Oriol2015}: 
\begin{eqnarray}
	u(x,y,z,t) &=& \sin(x)\cos(y)\cos(z), \nonumber \\
	v(x,y,z,t) &=& -\cos(x)\sin(y)\cos(z), \nonumber \\
	w(x,y,z,t) &=& 0, \nonumber \\
	p(x,y,z,t) &=& p_0+\frac{1}{16}\left(\cos(2x)+\cos(2y) \right)\left(\cos(2z)+2) \right),
\label{eq:TGV_1}
\end{eqnarray}
in $\Omega=[\pi,\pi]^3$ and periodic boundary conditions everywhere. As numerical parameters we take $(p,p_\gamma)=(4,0)$; $N_i=494592$ tetrahedral elements; $\theta=0.51$; $\Delta t$ according to 
the CFL time restriction; $t_{end}=10$; %$\nu=2.5\cdot 10^{-3}$, $\nu=1.25\cdot 10^{-3}$ and $\nu=6.75\cdot 10^{-4}$, so that the resulting Reynolds numbers are $Re=400$, $Re=800$ and $Re=1600$, respectively.
and several values of $\nu$ so that the Reynolds numbers under consideration are $Re=400$, $Re=800$ and $Re=1600$, respectively.

\begin{figure}[ht]
    \begin{center}
    \includegraphics[width=0.3\textwidth]{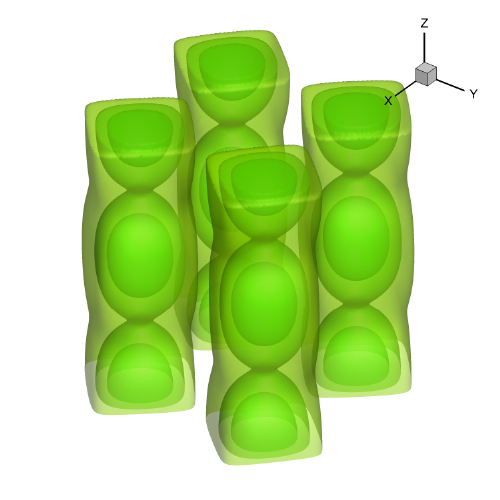}
		\includegraphics[width=0.3\textwidth]{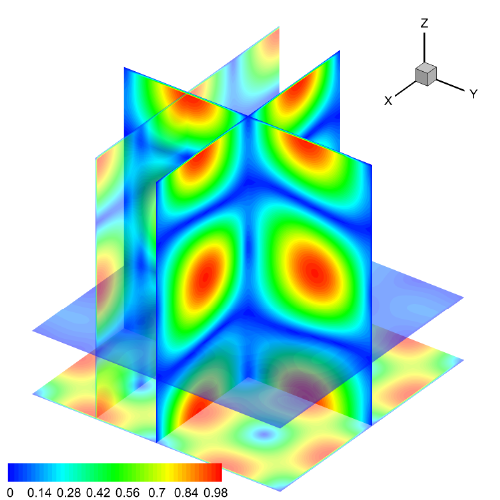}
		\includegraphics[width=0.3\textwidth]{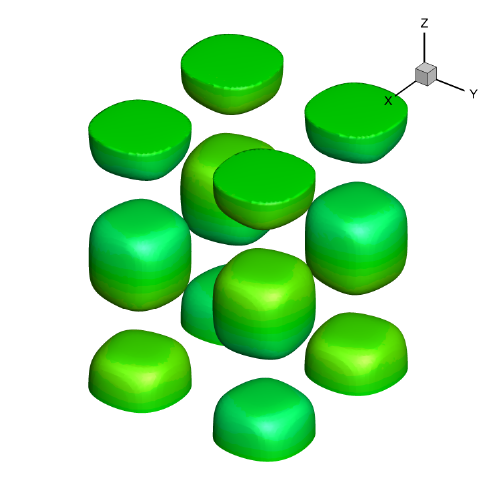} \\
		\includegraphics[width=0.3\textwidth]{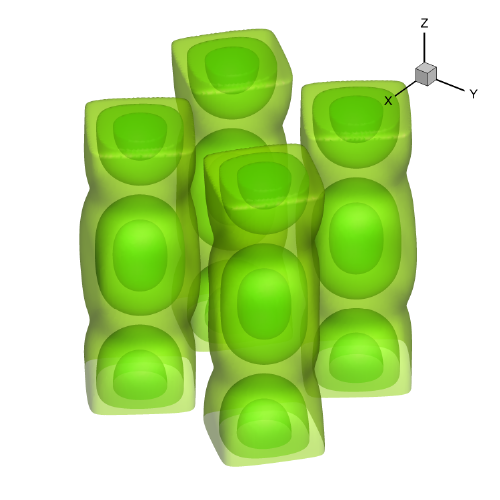}
		\includegraphics[width=0.3\textwidth]{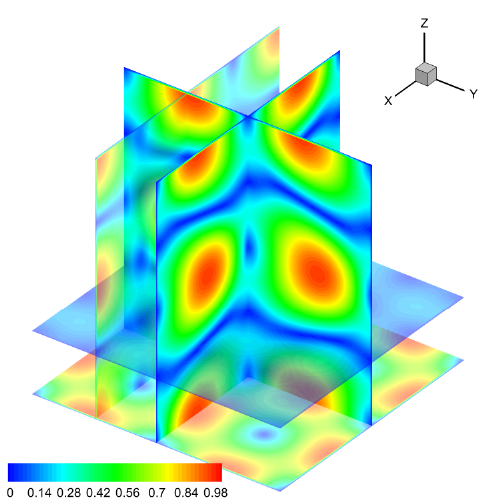}
		\includegraphics[width=0.3\textwidth]{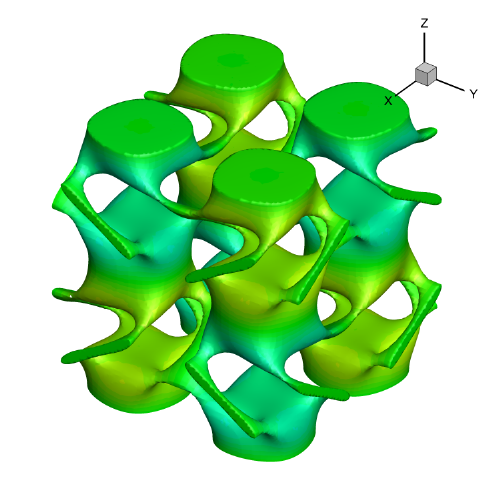} 
    \caption{3D Taylor-Green vortex at $Re=800$. From left to right: Pressure isosurfaces, velocity magnitude and vorticity isosurfaces at times $t=0.5$ (top) and $t=1.0$ (bottom).}
    \label{fig.TGV.1}
		\end{center}
\end{figure}

\begin{figure}[ht]
    \begin{center}
		\includegraphics[width=0.3\textwidth]{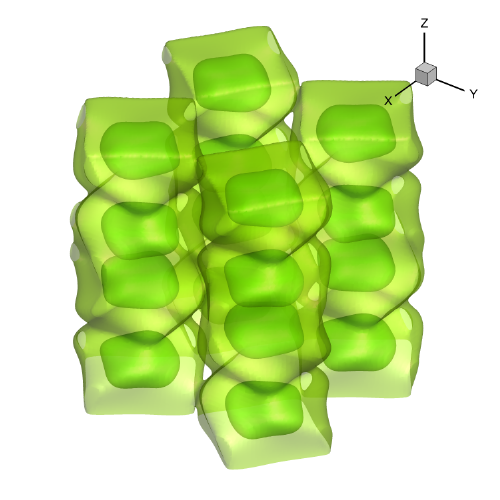}
		\includegraphics[width=0.3\textwidth]{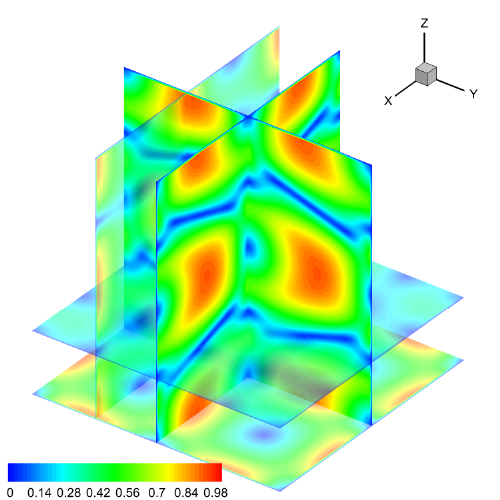}
		\includegraphics[width=0.3\textwidth]{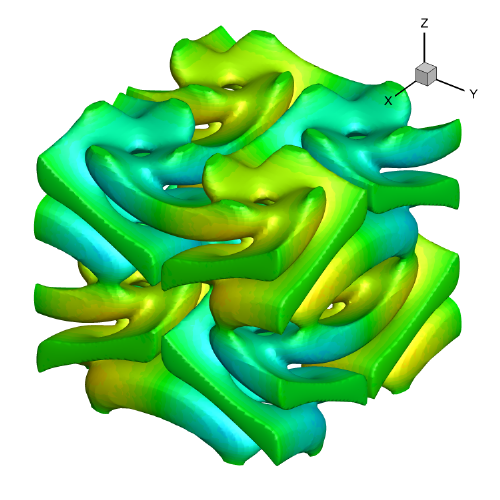} \\
		\includegraphics[width=0.3\textwidth]{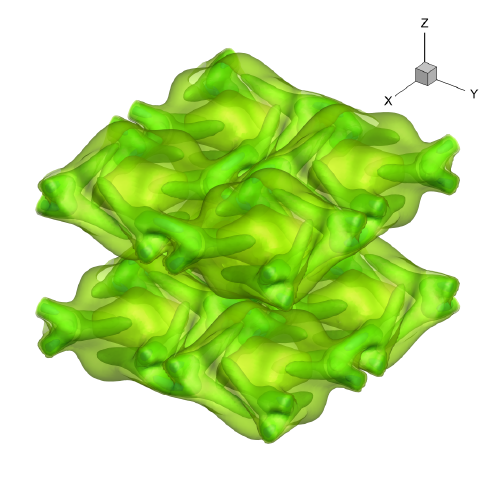}
		\includegraphics[width=0.3\textwidth]{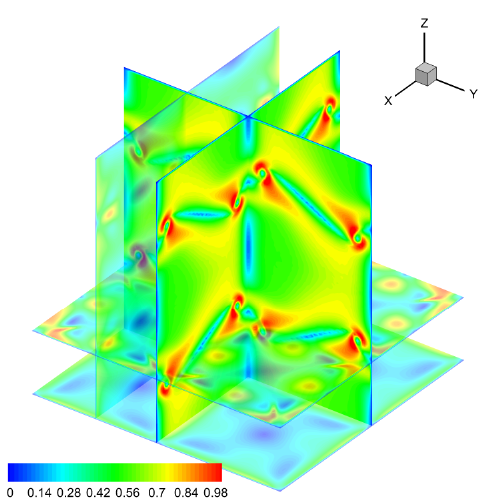}
		\includegraphics[width=0.3\textwidth]{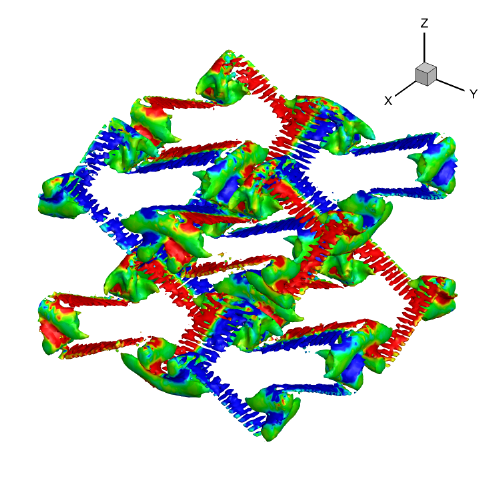} \\
		\includegraphics[width=0.3\textwidth]{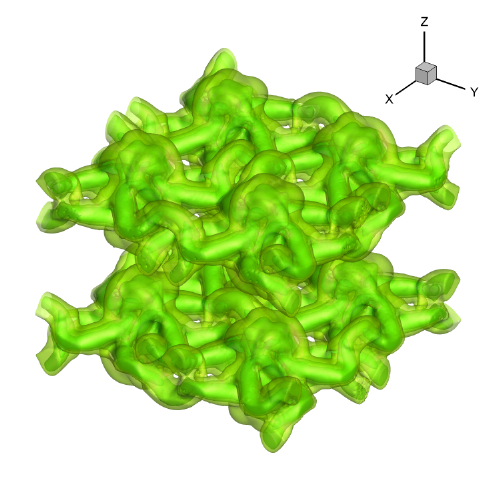}
		\includegraphics[width=0.3\textwidth]{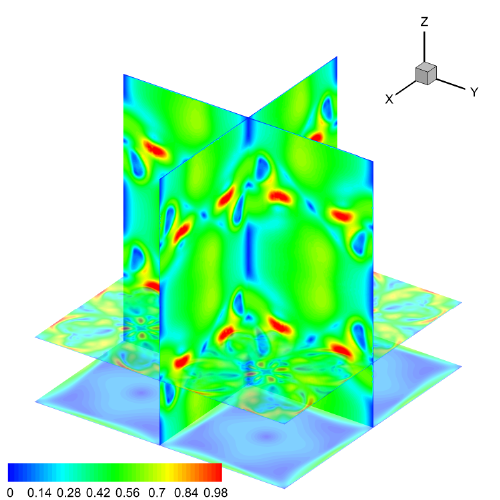}
		\includegraphics[width=0.3\textwidth]{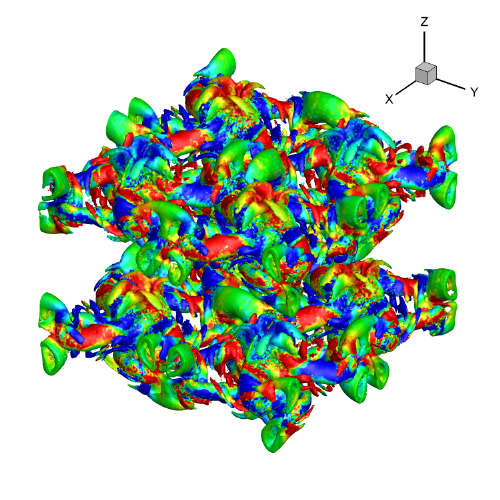}
    \caption{3D Taylor-Green vortex at $Re=800$. From left to right: Pressure isosurfaces, velocity magnitude and vorticity isosurfaces at times $t=2.1$ (top), $t=4.8$ (center) and $t=9.0$ (bottom).}
    \label{fig.TGV.1a}
		\end{center}
\end{figure}

\begin{figure}[ht]
    \begin{center}
		\includegraphics[width=0.6\textwidth]{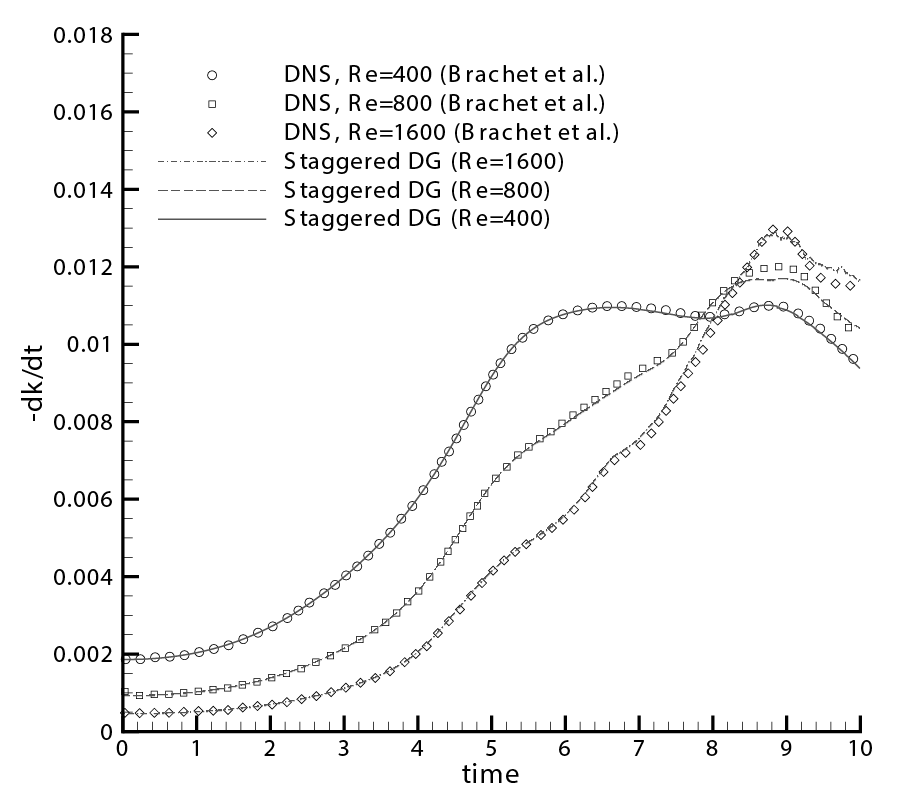}
    \caption{Time evolution of the kinetic energy dissipation rate $-dk/dt$ for the 3D Taylor-Green vortex, compared with available DNS data of Brachet et al \cite{Brachet1983} 
		for $Re=400,800$ and $Re=1600$.} 
    \label{fig.TGV.2}
		\end{center}
\end{figure}

%\begin{figure}[ht]
    %\begin{center}
    %\includegraphics[width=0.45\textwidth]{./Figures/TGvortex/TGV_vort1.png}
		%\includegraphics[width=0.45\textwidth]{./Figures/TGvortex/TGV_vort2.png}
		%\includegraphics[width=0.45\textwidth]{./Figures/TGvortex/TGV_vort3.png}
		%\includegraphics[width=0.45\textwidth]{./Figures/TGvortex/TGV_dk.png}
    %\caption{From the top left to bottom right: Vorticity isolines for $\omega=1.05$ at times $t=[0.4, 1.0, 1.9]$ colored using helicity; Time evolution of the gradient of total Kinetic energy $-dk/dt$ compared with DNS available data}
    %\label{fig.TGV.1}
		%\end{center}
%\end{figure}
%The total Kinetic energy is computed such as
%\begin{eqnarray}
	%k(t)=\frac{1}{2|\Omega|}\int\limits_\Omega{\mathbf{v}\cdot \mathbf{v} d\omega}.
%\label{eq:TGV_2}
%\end{eqnarray}
A plot of the time evolution of the pressure field, the velocity magnitude and the vorticity pattern is shown in Figure \ref{fig.TGV.1} for several times, as well as time series of the total kinematic dissipation rates compared with available DNS data given by Brachet et al in \cite{Brachet1983} in Figure \ref{fig.TGV.2}.
%A plot of the initial vorticity is shown in Figure \ref{fig.TGV.1} as well as time series of the total kinematic dissipation rates compared with available DNS data given by Brachet et al in \cite{Brachet1983}. A relatively good agreement can be observed with reference data even if a coarse grid has been used.
A good agreement between reference data and our numerical results can be observed. In Figure \ref{fig.TGV.1} the vorticity pattern shows a really complex behavior that appears after a certain 
time. 

%\begin{figure}[ht]
    %\begin{center}
		%\includegraphics[width=0.4\textwidth]{./Figures/TGVn/vort/x0.png}
		%\includegraphics[width=0.4\textwidth]{./Figures/TGVn/vort/xpi8.png} \\
		%\includegraphics[width=0.4\textwidth]{./Figures/TGVn/vort/xpi4.png}
		%\includegraphics[width=0.4\textwidth]{./Figures/TGVn/vort/xpi2.png}
    %\caption{Vorticity pattern in the plane $x=0,\frac{\pi}{8},\frac{\pi}{4},\frac{\pi}{2}$ from top left to bottom right for $Re=800$, $t=9s$.}
    %\label{fig.TGV.3a}
		%\end{center}
%\end{figure}
%
%\begin{figure}[ht]
    %\begin{center}
		%\includegraphics[width=0.4\textwidth]{./Figures/TGVn/vort/z0.png}
		%\includegraphics[width=0.4\textwidth]{./Figures/TGVn/vort/zpi8.png} \\
		%\includegraphics[width=0.4\textwidth]{./Figures/TGVn/vort/zpi4.png}
		%\includegraphics[width=0.4\textwidth]{./Figures/TGVn/vort/zpi2.png}
    %\caption{Vorticity pattern in the plane $z=0,\frac{\pi}{8},\frac{\pi}{4},\frac{\pi}{2}$ from top left to bottom right for $Re=800$, $t=9s$.}
    %\label{fig.TGV.3b}
		%\end{center}
%\end{figure}

%As shown in Figure $\ref{fig.TGV.1a}$ the vorticity pattern becomes very complex close to $t=9s$. Some details of this pattern for several slice positions are shown in Figures $\ref{fig.TGV.3a}$ 
%and $\ref{fig.TGV.3b}$ and underlining several vortical structures that are difficult to be seen from the three-dimensional plot. 
In this particular test it is very important to resolve the small scale structures that, close to $t=9$, constitute the main contribution to the total kinetic energy dissipation. 
The mean number of iterations needed to solve the linear system for the pressure at $Re=1600$ and a tolerance of $tol=10^{-8}$ is $I_{mean}=290.7$. In general we observe a number of 
iterations of the linear solver in the range $I \in [93,2516]$ for this test case, without the use of any preconditioner. 

\subsection{Womersley flow}
In this section the proposed algorithm is verified against the exact solution for an oscillating flow in a rigid tube of length $L$ with circular cross section of diameter $D$. 
The unsteady flow is driven by a sinusoidal pressure gradient on the inlet and outlet boundaries
\begin{equation}
	\frac{p_{out}(t)-p_{inlet}(t)}{L}=\frac{\tilde{p}}{\rho}e^{i \omega t},
\label{eq:W_1}
\end{equation}
where $\tilde{p}$ is the amplitude of the pressure gradient; $\rho$ is the fluid density; $\omega$ is the frequency of the oscillation; $i$ indicates the imaginary unit; $p_{inlet}$ and $p_{out}$ are the inlet and outlet pressures, respectively. The analytical solution was derived by Womersley in \cite{Womersley1995}.
According to \cite{Womersley1995,FambriDumbserCasulli} no convective contribution is considered.
By imposing Eq. \eref{eq:W_1} at the tube ends, the resulting unsteady velocity field is uniform in the axial direction and is given by
\begin{equation}
	u_{e}(\xx,t)=\frac{\tilde{p}}{\rho}\frac{1}{i \omega}\left[ 1- \frac{J_0\left(\alpha \zeta i^{\frac{3}{2}}\right)}{J_0\left(\alpha i^{\frac{3}{2}}\right)} \right]e^{i \omega t}\,\, ; \,\, 
	v_{e}(\xx,t)=w_{e}(\xx,t)=0,
\label{eq:W_2}
\end{equation}
where $\zeta=2r/D$ with $r=\sqrt{y^2+z^2}$ is the dimensionless radial coordinate; $D$ is the diameter of the tube; $\alpha=\frac{D}{2}\sqrt{\frac{\omega}{\nu}}$ is a constant; and $J_0$ is the zero-th order Bessel function of the first kind. For the present test we take $\Omega$ as a cylinder (aligned with the $x$-axis) of length $L=4$ and diameter $D=2$; $\tilde{p}=1000$; $\rho=1000$; $\omega=2 \pi$; and $\nu=0.04$. The computational domain $\Omega$ is covered with a total number of only $\Ni=1185$ tetrahedra and the time step size is chosen as $\Delta t=0.3$, which is $30\%$ of one 
oscillation period. For this test we take $(p, p_\gamma)=(4,3)$ in order to produce a good solution also with the chosen time step $\Delta t$, which can be considered as very large for this problem. %Remark how, since the viscosity contribution is discretized implicitly, the scheme with no convection is unconditionally stable and this allows to use large time steps.

\begin{figure}[ht]
    \begin{center}
    \includegraphics[width=0.3\textwidth]{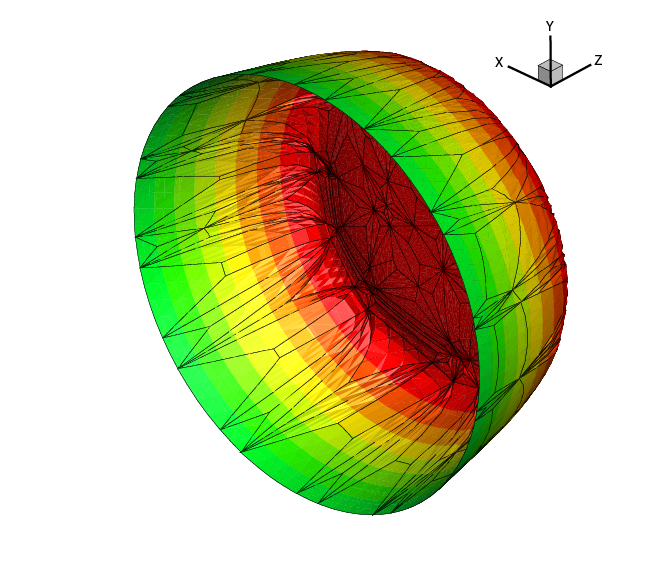}
		\includegraphics[width=0.3\textwidth]{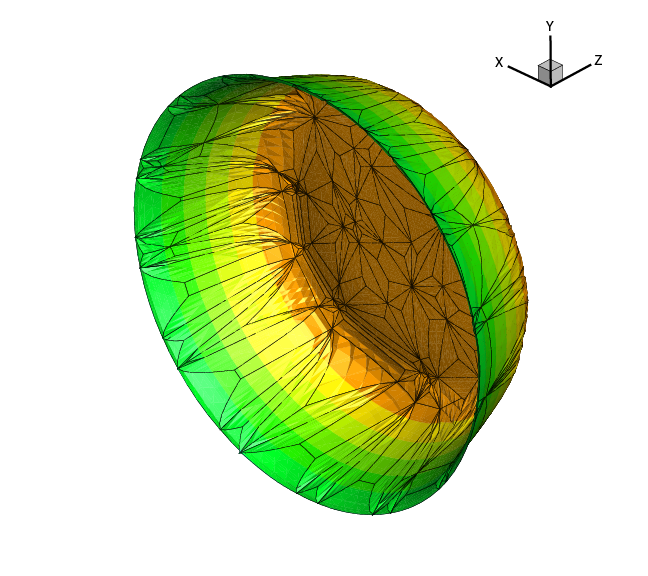}
		\includegraphics[width=0.3\textwidth]{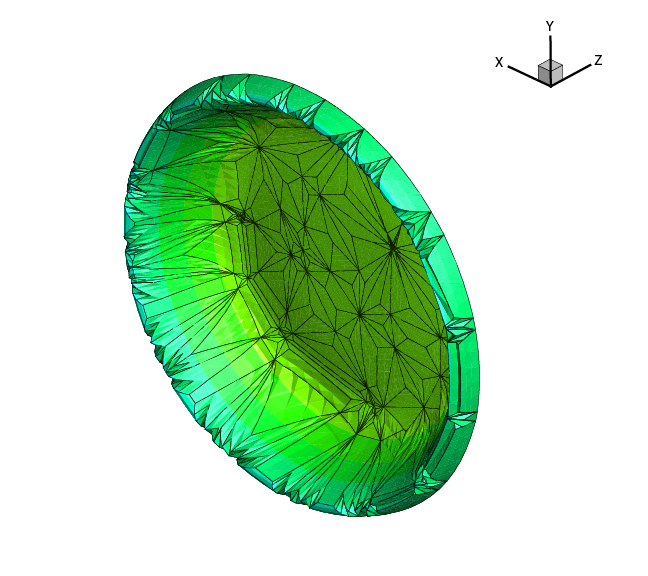}
		\includegraphics[width=0.3\textwidth]{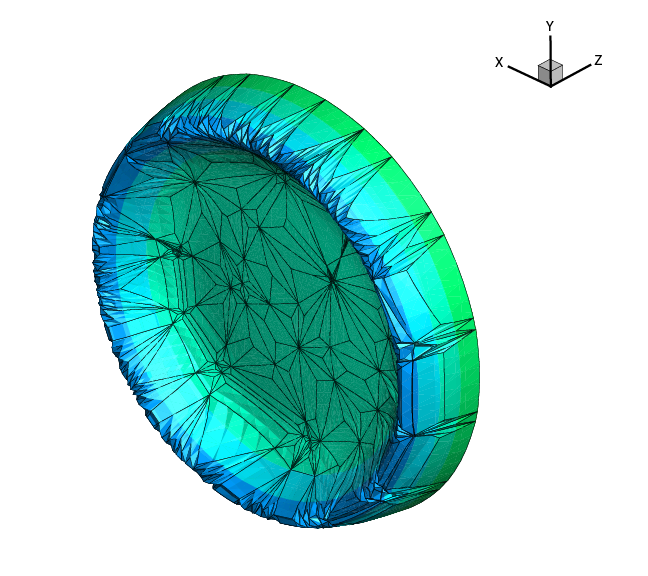}
		\includegraphics[width=0.3\textwidth]{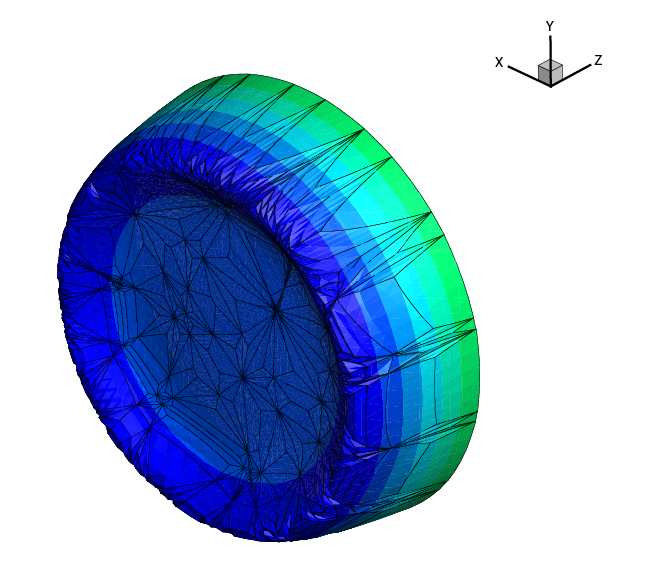}
    \caption{3D Womersley flow. Plot of $u$ in the middle of the tube in \textit{one single} time control volume $T=[0.3,0.6]$. From top left to bottom right we plot the discrete solution 
		at intermediate time levels $t=[0.3,0.375,0.45,0.525,0.6]$.}
    \label{fig.Wom.1}
		\end{center}
\end{figure}

\begin{figure}[ht]
    \begin{center}
    \includegraphics[width=0.4\textwidth]{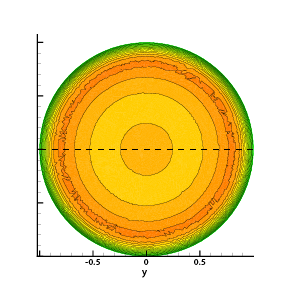}
		\includegraphics[width=0.4\textwidth]{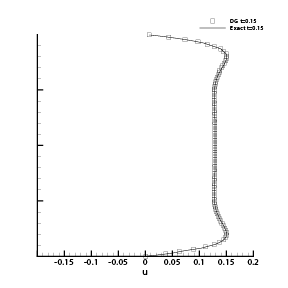} \\
		\includegraphics[width=0.4\textwidth]{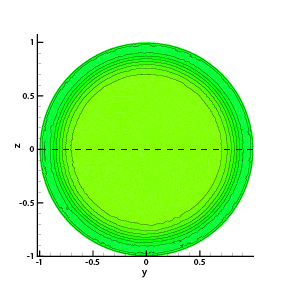}
		\includegraphics[width=0.4\textwidth]{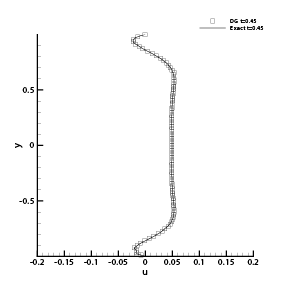} \\
		\includegraphics[width=0.4\textwidth]{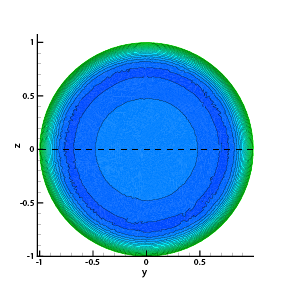}
		\includegraphics[width=0.4\textwidth]{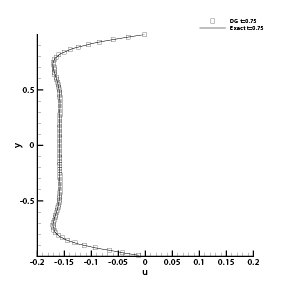}

    \caption{3D Womersley flow. Axial velocity contours in the plane $x=2$ (left column) and comparison of the velocity against the exact solution at $x=2$ and $z=0$ (right column) at times, from top to bottom, $t=[0.15, 0.45, 0.75]$.}
    \label{fig.Wom.2}
		\end{center}
\end{figure}

\begin{figure}[ht]
    \begin{center}
		\includegraphics[width=0.45\textwidth]{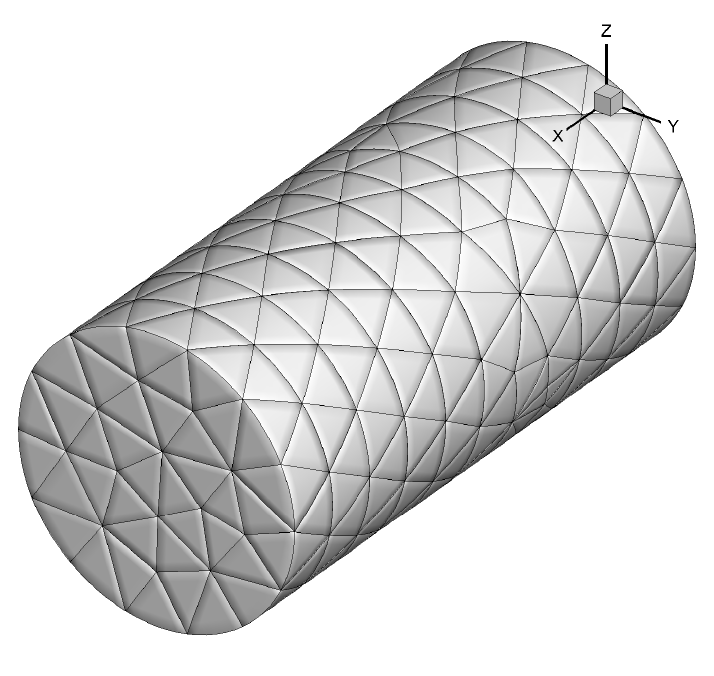}
    \includegraphics[width=0.45\textwidth]{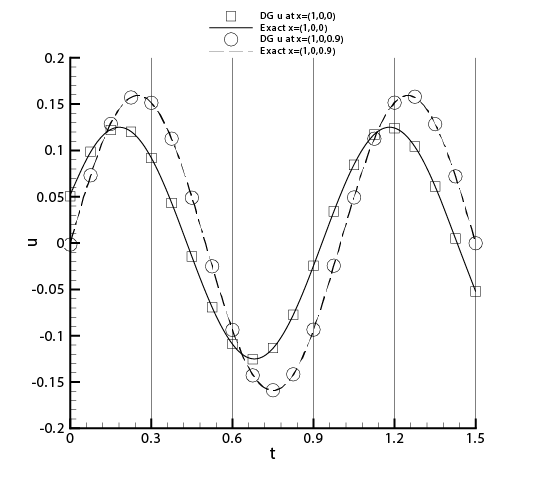}
    \caption{3D Womersley flow. Three dimensional view of the isoparametric grid used in this test case (left); 
		Time series of u in the plane $x=1$, $(y,z)=(0,0)$ and $(y,z)=(0,0.9)$ (right). 
		The vertical lines represent the very large time step size of $\Delta t=0.3$ used in this simulation.}
    \label{fig.Wom.3}
		\end{center}
\end{figure}
Due to the curved geometry of the problem we use a fully isoparametric approach to fit the cylinder. A plot of the isoparametric grid that has been used here is reported in Figure \ref{fig.Wom.3} on the left.
We test our numerical solution in the cutting slice $\Gamma=\{x=2\}$ and successively on the line given by $(x,z)=(2,0) \in \Gamma$. Figure \ref{fig.Wom.1} shows the evolution of the velocity profile $u$ on $\Gamma$ solved in a single time cell $\Gamma^{st}(t,\vec{x})=\Gamma(\vec{x}) \times [0.3,0.6]$ evaluated at several intermediate times. A comparison between numerical and exact solution is reported in Figure \ref{fig.Wom.2} as well as the plot of $\Gamma$, in order to show the axial symmetry of the solution, that is not trivial to obtain for the chosen discretization (very coarse unstructured mesh and
very large time steps). 
Finally, a plot of the time series of the velocity $u$ computed in $\xx=(1,0,0)$ and $\xx=(1,0,0.9)$ is reported in Fig. \ref{fig.Wom.3} and is compared with the exact solution. 
It is clear from Figures \ref{fig.Wom.2} and \ref{fig.Wom.3} that this test with the chosen time step can reproduce good results only if we use high order polynomials also in time; 
indeed, the solution for a first order method in time would look piecewise constant within each time step.

\subsection{Blasius boundary layer}
We consider here a classical benchmark for viscous incompressible fluids. 
For the particular case of laminar stationary flow over a flat plate, a solution of Prandtl's boundary layer equations was found by Blasius in \cite{Blasius1908} and is given by the solution of 
a nonlinear third-order ODE, namely: 
\begin{equation}
\left\{
\begin{array}{l}
    f'''+ff''=0, \\
    f(0)=0, \quad f'(0)=0, \quad \lim \limits_{\xi \rightarrow \infty} f'(\xi)=1,
\end{array}
\right.
\end{equation}
where $\xi=y \sqrt{\frac{u_{\infty}}{2\nu x}}$ is the Blasius coordinate; $f'=\frac{u}{u_\infty}$; and $u_\infty$ is the far field velocity. The reference solution is computed here using a tenth-order DG ODE solver, see e.g. \cite{ADERNSE}, together with a classical shooting method.
In order to obtain the Blasius velocity profile in our simulations we consider a steady flow over a flat plate. As a result of the viscosity, a boundary layer appears along the no-slip wall. For the current test, we consider $\Omega=[-0.2, 0.8]\times [-0.2, 0.2]^2$. An initially uniform flow $u(x,y,z,0)=u_\infty=1$ , $v(x,y,z,0)=w(x,y,z,0)=0$ and $p(x,y,0)=1$ is imposed as initial condition, while an inflow boundary is imposed on the left boundary; no slip boundary condition is considered in the flat plane $\Gamma=\{ (x,y,z) | \quad x\geq 0 \,\, y=y_{\min}\}$; slip boundary conditions are imposed 
at $z=z_{\min}$ and $z=z_{\max}$ and transmissive boundary conditions are imposed at the upper face $y=y_{\max}$ .  
We consider here an extreme case of a very coarse mesh, where we cover our domain $\Omega$ with a set of only $\Ni=1522$ tetrahedra, whose characteristic length is $h=0.07$. 
The chosen polynomial degree is $(p,p_\gamma)=(4,0)$, the final simulation time is $t_{end}=10$ and the viscosity is $\nu=3\cdot 10^{-4}$. 

\begin{figure}[ht]
    \begin{center}
    \includegraphics[width=0.98\textwidth]{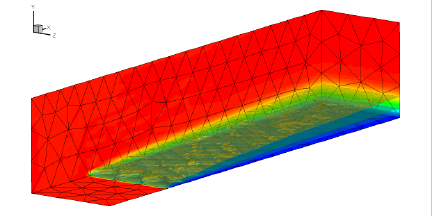}
    \caption{Blasius boundary layer: $3D$ plot of the domain $\Omega$ and sketch of the mesh on the boundary; the plotted iso-surfaces are corresponding to $u=0.2,0.4,0.8$}
    \label{fig.B.1}
		\end{center}
\end{figure}

%\begin{figure}[ht]
    %\begin{center}
    %\includegraphics[width=0.7\textwidth]{./Figures/Blasius3d/x_xi_plot.png}
    %\caption{$2D$ plot on the symmetry plane $z=0$ where the velocity $u$ in the $x-\xi$ plane  is plotted.}
    %\label{fig.B.2}
		%\end{center}
%\end{figure}

\begin{figure}[ht]
    \begin{center}
		\includegraphics[width=0.38\textwidth]{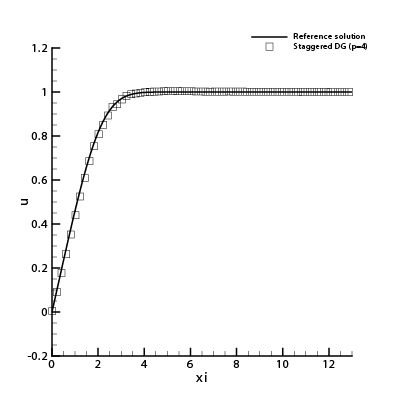}
    \caption{Blasius boundary layer: numerical solution and reference solution taken on the line $(x,y,z)=(0.4,y,0)$.} 
    \label{fig.B.3}
		\end{center}
\end{figure}

The resulting Blasius velocity profile is shown in Figure $\ref{fig.B.1}$ where also a sketch of the grid is reported. 
%The profile with respect to the Blasius coordinate $\xi$ is shown on the left of Figure $\ref{fig.B.2}$ in order to verify whether the obtained solution is self-similar with respect to $\xi$. 
A comparison between the numerical results presented here and the Blasius solution is depicted in Figure $\ref{fig.B.3}$. 
A very good agreement between numerical and reference solution can be observed, which is quite remarkable, if we take into account the mesh size and considering that 
the major part of the boundary layer is essentially resolved in only one single control volume. 

\subsection{Backward-facing step.}
In this section, the three-dimensional numerical solution for the fluid flow over a backward-facing step is considered. For this test problem, both experimental and numerical results are available at several Reynolds numbers, see e.g. \cite{Armaly1983,Erturk2008}. In particular, it is known that two dimensional simulations are in good agreement with experimental evidence only up to $Re=400$. Beyond this critical value, two dimensional simulations present a large secondary recirculation zone that reduces the main recirculation zone. On the contrary, experimental results show that this secondary vortex appears only at higher Reynolds number due to three-dimensional effects (see e.g. \cite{Armaly1983}). The used step size is of $S=0.49$ and the ratio between the total height $H$ and the inlet height $h_{in}$ is of $H/h_{in}=1.9423$. We consider here a smaller domain with respect to the experimental setup of Armaly in \cite{Armaly1983}, but sufficient to see the three-dimensional effects. In particular $\frac{x}{S}\in [-10,20]$, $y\in [-0.49,0.51]$ and $\frac{z}{S}\in [0,12]$. The domain is covered using $\Ni=19872$ terahedral elements and we take $(p,p_\gamma)=(4,0)$ and $Re=600$. We impose the exact Poiseuille profile in the $y$-direction at the tube inlet, transmissive boundary conditions at the tube outlet and no-slip boundary conditions otherwise. For the current test $\Delta t$ is taken according to the CFL time restriction based on the magnitude of the flow velocity and $t_{end}=80$.

\begin{figure}[ht]
    \begin{center}
    \includegraphics[width=0.7\textwidth]{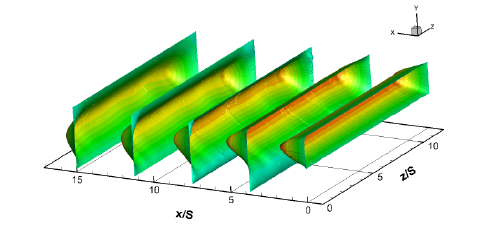} \\
    \caption{3D backward facing step. Value of $u$ in the $(y,z)$-plane at $x=[0,3.75,7.5,11.25,15]$.}
    \label{fig.BFS.1}
		\end{center}
\end{figure}

A plot of the velocity profile at several values of $x/S$ is shown in Figure $\ref{fig.BFS.1}$. 
The resulting recirculation zones in the symmetry plane and close to the side wall $\frac{z_{max}}{S}$ are shown in Figure $\ref{fig.BFS.2}$, as well as the equivalent in the plane $(\frac{x}{S},\frac{z}{S})$ close to the bottom and the top wall in Figure $\ref{fig.BFS.3}$. As we can see, no important secondary recirculation zones appear in the symmetric plane, while a couple of recirculations appear close to the side walls. The presence of these secondary recirculations seem to reduce the reattachment point for the main recirculation close to the side walls (see Figure $\ref{fig.BFS.3}$ top). On the contrary, a larger recirculation zone can be seen in the middle of the channel. The resulting reattachment point in the symmetry plane is $\frac{x_1}{S}=11.2$, that is really close to the one obtained in  the experimental case, whose value is $\frac{x_1}{S}=11.24$. Note that the two dimensional numerical simulation, as presented in \cite{2SINS}, leads to a reattachment point of $\frac{x_1}{S}=9.4$, 
which completely underestimates the experimental one.

\begin{figure}[ht]
    \begin{center}
    \includegraphics[width=0.7\textwidth]{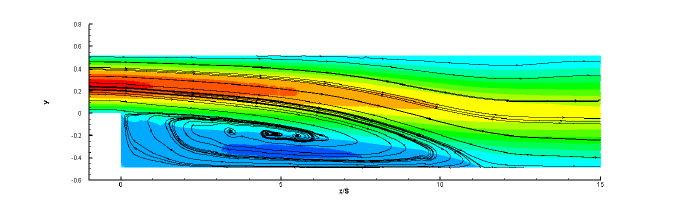} \\
		\includegraphics[width=0.7\textwidth]{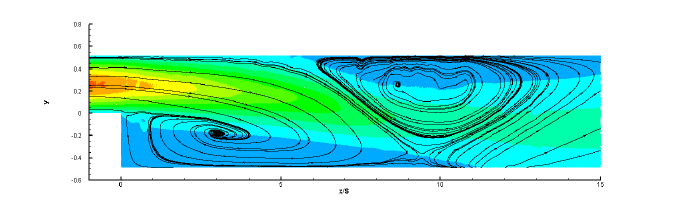} 
    \caption{3D backward facing step. Recirculation zones in the plane $(\frac{x}{S},y)$  in the symmetry plane (top) and close to the side wall at $\frac{z}{S}=12$ (bottom). }
    \label{fig.BFS.2}
		\end{center}
\end{figure}

\begin{figure}[ht]
    \begin{center}
    \includegraphics[width=0.7\textwidth]{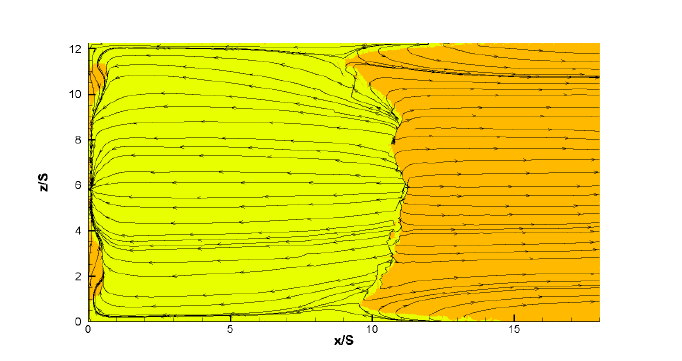} \\
		\includegraphics[width=0.7\textwidth]{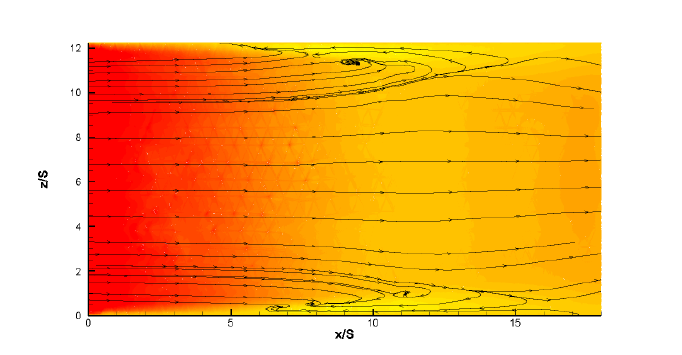} 
    \caption{3D backward facing step. Recirculation zones in the plane $(\frac{x}{S},\frac{z}{S})$ close to the bottom and close to the top wall. }
    \label{fig.BFS.3}
		\end{center}
\end{figure}

\subsection{Flow around a sphere}
In this section we consider the flow around a sphere. In particular we take as computational domain $\Omega=\mathcal{S}_{10} \cup \mathcal{C}_{10,15}-\mathcal{S}_{0.5}$, where $\mathcal{S}_r$ is a generic sphere with center $\vec{0}$ and radius $r$; $\mathcal{C}_{r,H}$ is a cylinder with circular basis on the $yz$-plane, radius $r$ and height $H$. We use a very coarse grid that is composed by a total number of $\Ni=14403$ tetrahedra whose characteristic length is $h=0.2$ close to the sphere, while it is only $h=0.8$ away from the sphere. A sketch of the grid is shown in Figure $\ref{fig.Sph.1}$.

\begin{figure}[ht]
    \begin{center}
    \includegraphics[width=0.6\textwidth]{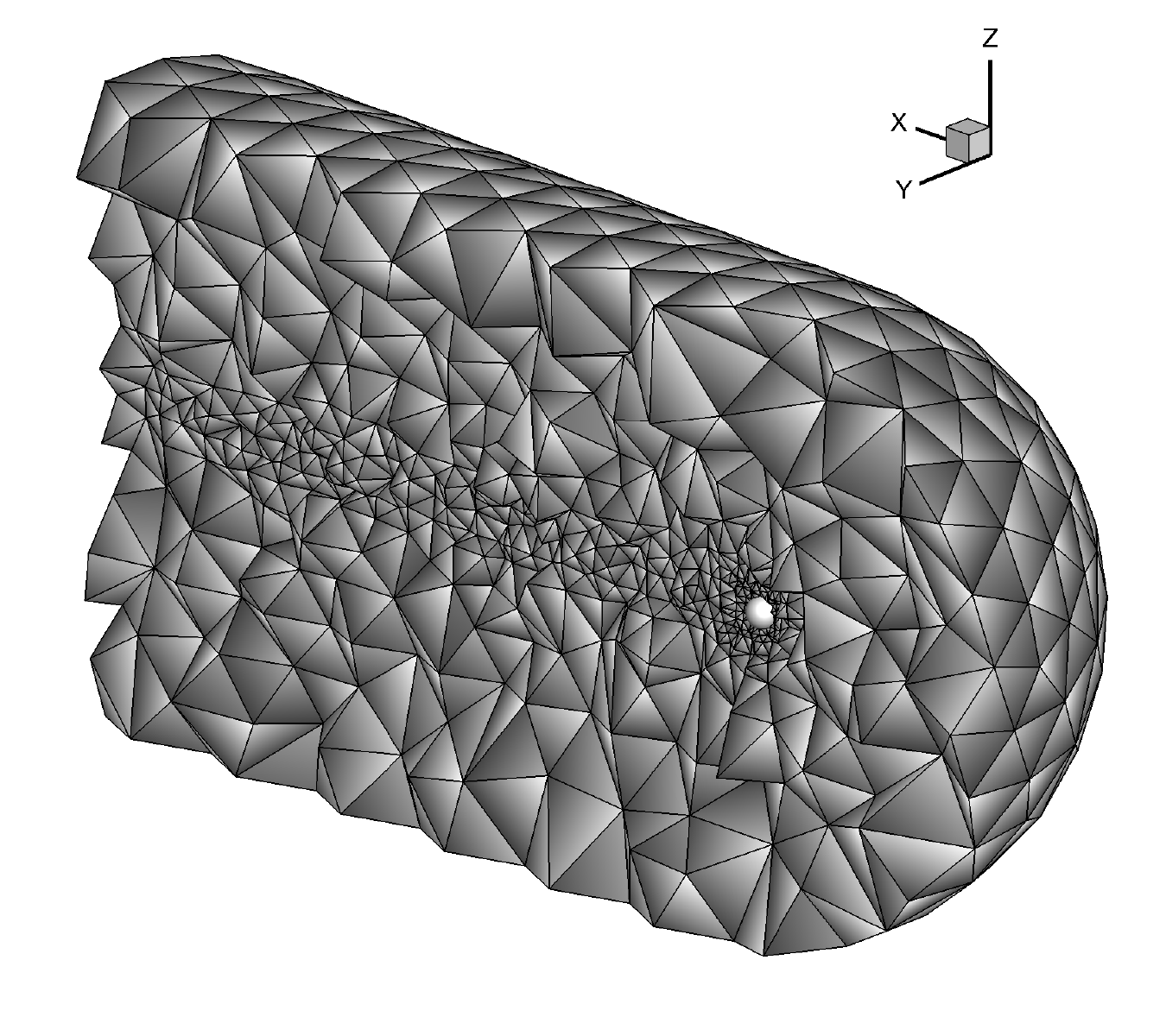}
    \caption{Flow around a sphere. Cut view of the computational domain with $\Ni=14403$.}
    \label{fig.Sph.1}
		\end{center}
\end{figure}

We start from an initial steady flow of magnitude $\mathbf{v}_0=(u_\infty,0,0)$ with $u_\infty=0.5$ and we impose $u_\infty$ on $\mathcal{S}_{10}\cap \{x\leq 0\}$ as boundary condition; transmissive boundary condition on $\mathcal{C}_{10,15}$ and no-slip condition on $\mathcal{S}_{0.5}$. We use a polynomial degree $(p,p_\gamma)=(3,0)$ and $\theta=0.51$ using the method explained in section $\ref{sec.CNmethod}$; $Re=300$; $t_{end}=300$ and $\Delta t$ is taken according to the CFL time restriction for the convective term. 

A plot of the spanwise velocity contours for $v$ is reported in Figure $\ref{fig.Sph.theta.1}$ at $t=300$ and shows a very complex and three-dimensional behavior of the numerical solution. 
The mean number of iterations needed to solve the pressure system with a tolerance of $tol=10^{-8}$ is $I_{mean}=201.8$ for this test problem. The maximum number of iterations is 
$I_{max}=2552$ and is observed only at the beginning of the simulation, when the constant initial condition for the velocity has to be adjusted. Instead, the minimum number of iterations 
$I_{min}=62$ is observed when the Von Karman vortex street is fully developed. 

\begin{figure}[ht]
    \begin{center}
		\includegraphics[width=0.8\textwidth]{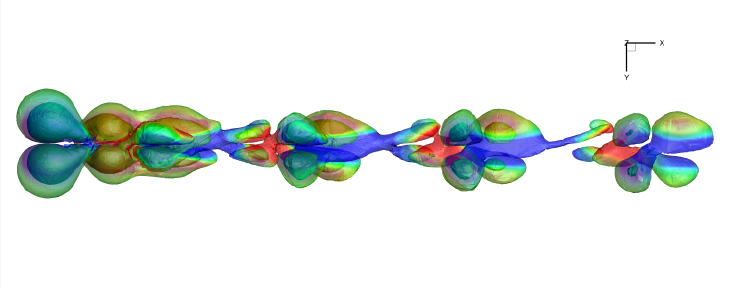}
		\includegraphics[width=0.8\textwidth]{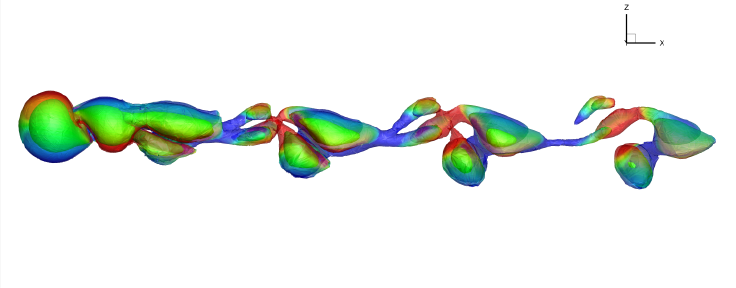}
		\includegraphics[width=0.8\textwidth]{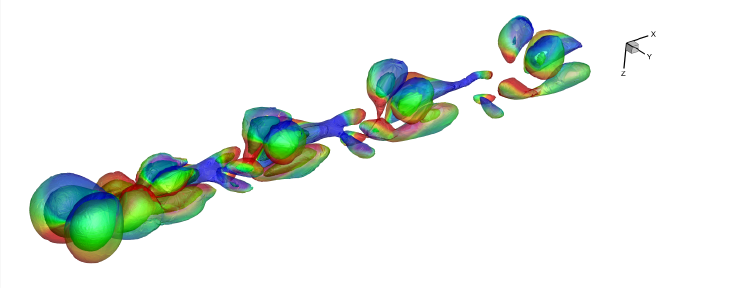}
    \caption{Flow around a sphere. Contour isosurfaces for the spanwise velocity $v$ in the $(x,y)$ plane, in the $(y,z)$ plane and 3D plot.}
    \label{fig.Sph.theta.1}
		\end{center}
\end{figure}

%A time evolution of the particle tracker is finally reported in figure $\ref{fig.Sph.theta.3}$.

\begin{figure}[ht]
    \begin{center}
		\includegraphics[width=0.8\textwidth]{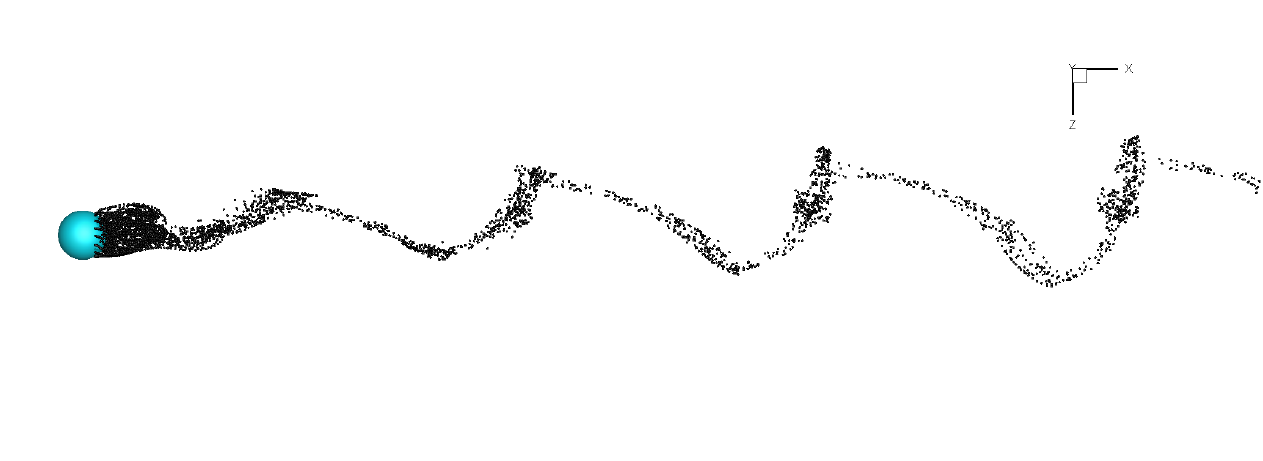}
		\includegraphics[width=0.8\textwidth]{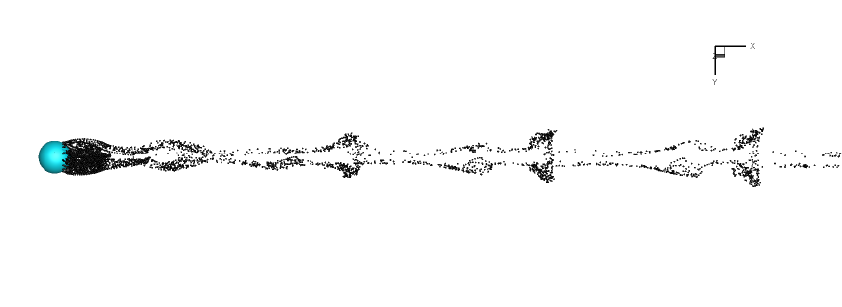}
    \caption{Flow around a sphere. Side view and upper view of the particle path at $t=300$.}
    \label{fig.Sph.theta.2}
		\end{center}
\end{figure}

%\begin{figure}[ht]
    %\begin{center}
		%\includegraphics[width=0.45\textwidth]{./Figures/Sphere3D/pp_262_8.png}
		%\includegraphics[width=0.45\textwidth]{./Figures/Sphere3D/pp_265_5.png}
		%\includegraphics[width=0.45\textwidth]{./Figures/Sphere3D/pp_267_8.png}
		%\includegraphics[width=0.45\textwidth]{./Figures/Sphere3D/pp_270_2.png}
    %\caption{Time evolution of the particle path at times, from top left to bottom right, $t=[262.8, 265.5, 267.8, 270.2]$}
    %\label{fig.Sph.theta.3}
		%\end{center}
%\end{figure}

\begin{figure}[ht]
    \begin{center}
		\includegraphics[width=0.7\textwidth]{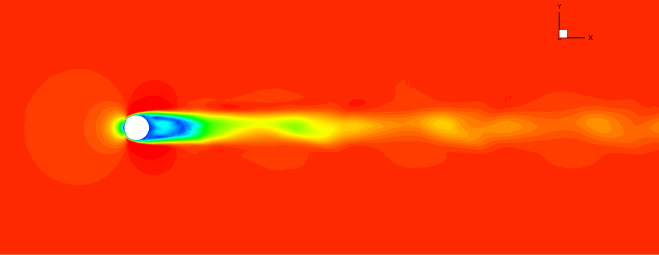}
		\includegraphics[width=0.7\textwidth]{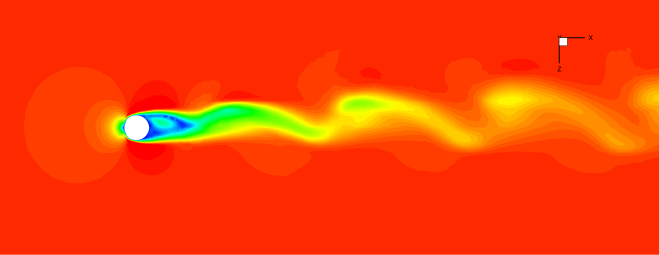}
    \caption{Flow around a sphere. Velocity magnitude at $t=t_{end}$ in the $(x,y)$ and $(x,z)$-plane.}
    \label{fig.Sph.theta.4}
		\end{center}
\end{figure}

A lateral and upper view of a particle tracer is reported in Figure $\ref{fig.Sph.theta.2}$ at $t=300$. The obtained results look very similar to the experimental ones 
obtained by H.Sakamoto et al. in \cite{Sakamoto1990}. The resulting Strouhal number for this simulation is $St=0.145$, which is close to the experimental range 
$St=0.15-0.18$ obtained in \cite{Sakamoto1990}.

\subsection{3D flow past a circular cylinder}
In this last test case we want to treat another classical problem for the incompressible Navier-Stokes equations that is the 3D flow around a circular cylinder. For this test, some numerical and experimental cases are available for a large range of Reynolds numbers. In particular several papers focus the attention on the formation of two instability modes characterized by large and small-scale streamwise vortex structures (see e.g. \cite{Williamson1988}), which act on the Reynolds-Strouhal number relationship. We consider here the problem of the flow past a circular cylinder in a confined channel and for a Reynolds number large enough to have three-dimensional effects and small-scale streamwise vortex structures. We define the blockage ratio $\beta=d/H$ where $d$ indicates the cylinder diameter and $H$ is the distance separating the two walls. In \cite{Rehimi2008} an experimental investigation for a blockage ratio of $\beta=1/3$ was presented, producint the 
$Re-St\cdot Re$ relation up to $Re=277$. Other numerical studies of Kanaris et al in \cite{Kanaris2011} give us a numerical analysis in the case of lower blockage ratio of $\beta=1/5$, finding a similar relation with respect to the unconfined experimental case of Williamson in \cite{Williamson1988}. 
We consider here two domains that are $\Omega_1=[-10,30]\times[-2.5,2.5]\times [-12,12]-C_{0.5,24}$ and $\Omega_2=[-10,30]\times[-1.5,1.5]\times [-12,12]-C_{0.5,24}$ where $C_{r,z}$ represents the cylinder of of radius $r$ and height $z$ centered in $0$ and corresponding to a blockage ratio of  $\beta=1/5$ and $\beta=1/3$, respectively. The first domain $\Omega_1$ is covered with a total number of $\Ni=50761$ tetrahedra and $\Omega_2$ is covered with $\Ni=32527$ elements. A sketch of the grid used in both the cases is shown in Figure $\ref{fig.Cyl.1}$.
\begin{figure}[ht]
    \begin{center}
		\includegraphics[width=0.45\textwidth]{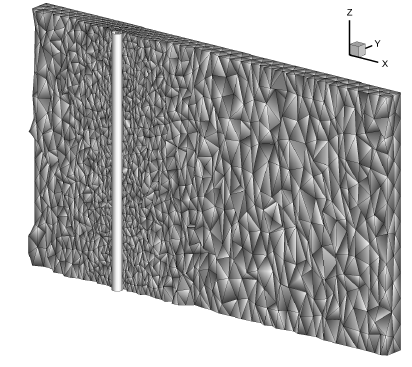}
		\includegraphics[width=0.45\textwidth]{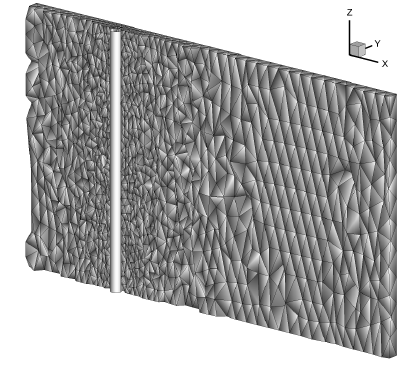}
    \caption{Flow around a cylinder. Half grid plot of $\Omega_1$ (left) and $\Omega_2$ (right).}
    \label{fig.Cyl.1}
		\end{center}
\end{figure}
As numerical parameters we use $(p,p_\gamma)=(3,0)$, $\theta=0.51$ and $t_{end}=200$. As initial condition we take a fully developed laminar Poiseuille profile and we impose velocity boundary conditions  on the inlet, transmissive boundary conditions on the outlet and no slip boundary conditions otherwise. Finally we impose for the two tests $\nu_1=1.66667 \cdot 10^{-3}$ and $\nu_2=1.80505 \cdot 10^{-3}$ corresponding to $Re_1=300$ and $Re_2=277$. Furthermore, isoparametric elements are considered for both the cases in order to fit better the curved cylinder. The resulting velocity contours 
at $t_{end}$ are reported in Figure $\ref{fig.Cyl.3}$, where we can observe the generation of the Von Karman vortex street past the cylinder, as well as the three-dimensional mixing effects given by the spanwise velocity 
$w$. 

%\begin{figure}[ht]
    %\begin{center}
		%\includegraphics[width=0.4\textwidth]{./Figures/Cylinder/1o5_Re300/u.png}
		%\includegraphics[width=0.4\textwidth]{./Figures/Cylinder/1o3_Re277/u.png}
		%\includegraphics[width=0.4\textwidth]{./Figures/Cylinder/1o5_Re300/v.png}
		%\includegraphics[width=0.4\textwidth]{./Figures/Cylinder/1o3_Re277/v.png}
		%\includegraphics[width=0.4\textwidth]{./Figures/Cylinder/1o5_Re300/w.png}
		%\includegraphics[width=0.4\textwidth]{./Figures/Cylinder/1o3_Re277/w.png}
    %\caption{Flow around a cylinder. Velocity profile $u$, $v$ and  $w$ from top to bottom past the circular cylinder in the sub-domain $\Omega=[-5,20]\times[-\frac{r}{\beta},\frac{r}{\beta}]\times[-7,7]$ for $(Re,\beta)=(300,\frac{1}{5})$ and $(Re,\beta)=(277,\frac{1}{3})$ in the left and right column, respectively.}
    %\label{fig.Cyl.2}
		%\end{center}
%\end{figure}

\begin{figure}[ht]
    \begin{center}
		\includegraphics[width=0.45\textwidth]{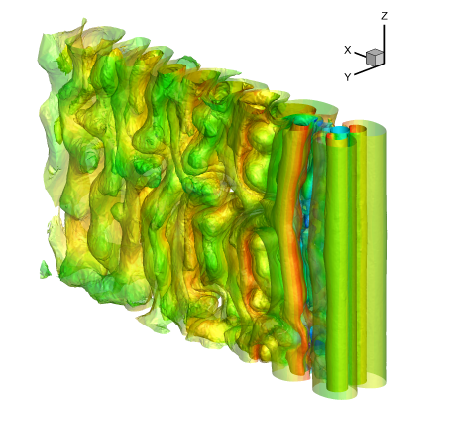}
		\includegraphics[width=0.45\textwidth]{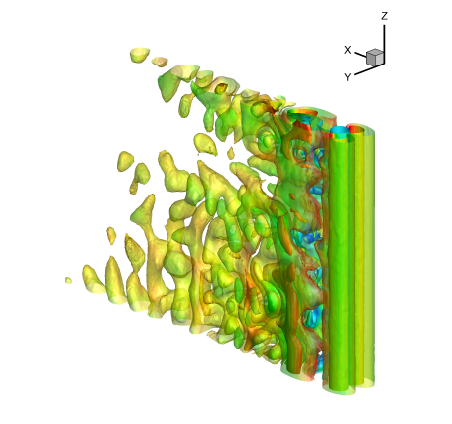}
		\includegraphics[width=0.45\textwidth]{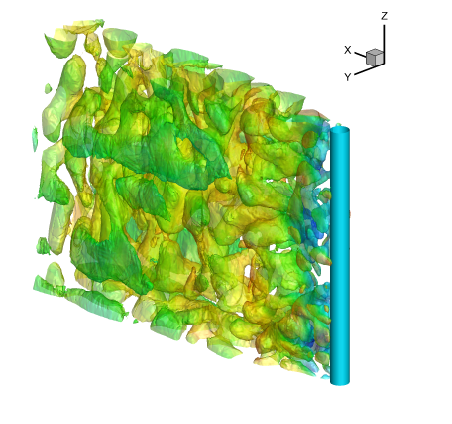}
		\includegraphics[width=0.45\textwidth]{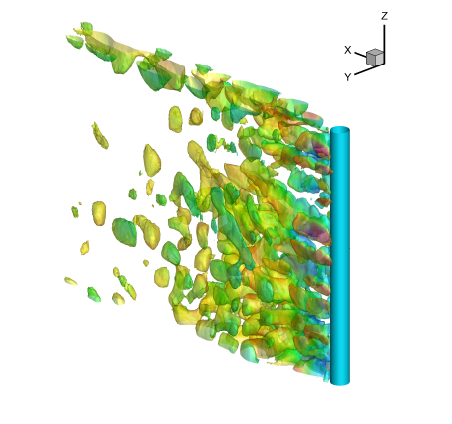}
    \caption{Flow around a cylinder. Isosurfaces of spanwise velocity $v=[\pm 0.1, \pm 0.03]$ and $w=\pm 0.03$ from top to bottom for the case $(Re,\beta)=(300,\frac{1}{5})$ (left column) and $(Re,\beta)=(277,\frac{1}{3})$ (right column). }
    \label{fig.Cyl.3}
		\end{center}
\end{figure}

%In Figure $\ref{fig.Cyl.3}$ the isosurfaces are reported for of the three-dimensional pattern of the spanwise velocities $v$ and $w$ at $t=t_{end}$. 
The resulting Strouhal number for the first case is $St=0.198$ which is in good agreement with the numerical one $St=0.1989$ of Kanaris in \cite{Kanaris2011} and the experimental one of Williamson in \cite{Williamson1988}. In the second case the obtained Strouhal number is $St=0.2414$, which corresponds to a value of $St \cdot Re=66.877$ that is in line with the experimental one of Rehimi et al. in \cite{Rehimi2008}, whose extrapolated value 
is $St \cdot Re=66.929$. This confirms the suggestion given in \cite{Rehimi2008} that the Strouhal number increases with increasing blockage. 

\section{Conclusions}
\label{sec.concl}

In this paper we have proposed a new arbitrary high order accurate space-time DG method on \textit{staggered} unstructured tetrahedral meshes for the solution of the incompressible Navier-Stokes 
equations in three space dimensions. The key idea of our approach is indeed the use of a \textit{staggered mesh}, where the pressure is defined on the main tetrahedral grid, while the velocity 
is defined on a face-based staggered dual mesh, composed of non-standard five-point hexahedral elements. To avoid the solution of nonlinear systems due to the presence of the nonlinear convective 
terms, we opt for a semi-implicit discretization in combination with an outer Picard iteration, leading to a rather simple space-time pressure correction algorithm. 
To the knowledge of the authors, this is the first time that a \textit{staggered} space-time DG scheme has been proposed for the 3D incompressible Navier-Stokes equations on unstructured 
tetrahedral meshes. 

The use of a staggered grid follows the ideas of classical finite difference schemes for the incompressible Navier-Stokes equations, but it is not yet very widespread in the 
DG community. However, it allows to produce a linear system to be solved in each time step with the smallest number of unknowns (only the scalar pressure) and with the smallest possible 
stencil (5-point stencil). 
The same DG algorithm on a \textit{collocated mesh} would either lead to a 17-point stencil (if the pressure is used as the only unknown, substituting the momentum equation into the  
continuity equation), or to a four times larger linear system with pressure and velocity as unknowns (if a 5-point stencil is used, hence \textit{not} substituting the momentum equation 
into the continuity equation). 
In the special case of piecewise constant polynomials in time ($p_\gamma=0$), the final system matrix becomes even symmetric and positive definite for appropriate boundary conditions, 
thus allowing the use of the conjugate gradient method. In all test cases shown in this paper, the pressure system could be solved with a simple matrix-free version of the GMRES/CG  
method, without the use of any preconditioner. 
In addition, all the coefficient matrices needed by the scheme can be precomputed and stored in a preprocessing step. In this way also the extension to high order isoparametric geometry becomes 
natural and does not affect the computational effort during run time. 
The staggered DG approach further allows to avoid the use of numerical flux functions (Riemann solvers) in the scheme, since all quantities are readily defined where they are needed, 
apart from the nonlinear convective terms, which are treated in a classical manner.  

The new numerical method has been applied to a large set of different steady and unsteady benchmark problems. It has been shown that the method achieves high order of accuracy in both, space and time, allowing thus the use of very coarse meshes in space and the use of very large time steps, without compromising the overall accuracy of the method. 
%An alternative treatment of the nonlinear convective-viscous term computed on the main tetrahedral mesh has been discussed. In this case we have a more regular linear system for the computation of the viscous velocity contribution as well as a faster way to compute the explicit convective part, compared with the natural extension of two dimensional to three dimensional method.

Future work will include the extension of the proposed staggered space-time DG method to the compressible Navier-Stokes equations in order to produce a novel family of 
all Mach number flow solvers, similar to the ideas proposed in \cite{CasulliCompressible,klein,munz2003,KleinMach,MeisterMach,MunzPark,CordierDegond,DumbserIbenIoriatti,DumbserCasulli2016} 
in the context of semi-implicit finite difference and finite volume schemes for compressible flows.

\section*{Acknowledgments}
The research presented in this paper was partially funded by the European Research Council (ERC) under the European Union's Seventh Framework
Programme (FP7/2007-2013) within the research project \textit{STiMulUs}, ERC Grant agreement no. 278267.
\clearpage
\bibliographystyle{elsarticle-num}
\bibliography{SIDG}
%\bibliographystyle{plain}

\end{document}